\documentclass[english,british]{amsart}
\usepackage[T1]{fontenc}
\usepackage[latin9]{inputenc}
\usepackage[a4paper]{geometry}
\usepackage{amsthm}
\usepackage{amssymb}
\usepackage{graphicx}
\usepackage{setspace}
\makeatletter
\numberwithin{equation}{section}
\numberwithin{figure}{section}
\theoremstyle{plain}
\newtheorem{thm}{\protect\theoremname}[section]
  \theoremstyle{definition}
  \newtheorem{defn}[thm]{\protect\definitionname}
\newtheorem*{defn*}{\protect\definitionname}
  \theoremstyle{remark}
  \newtheorem{rem}[thm]{\protect\remarkname}
  \theoremstyle{plain}
  \newtheorem{lem}[thm]{\protect\lemmaname}
  \theoremstyle{plain}
  \newtheorem{prop}[thm]{\protect\propositionname}
  \theoremstyle{definition}
  \newtheorem{example}[thm]{\protect\examplename}
  \theoremstyle{remark}
  \newtheorem*{rem*}{\protect\remarkname}
  \theoremstyle{plain}
  \newtheorem{cor}[thm]{\protect\corollaryname}
  \newtheorem{clm}[thm]{\protect\claimname}
  \newtheorem*{clm*}{\protect\claimname}

\usepackage{amscd}
\usepackage{mathptmx}

\newcommand{\e}{\mathrm{e}}

\newcommand{\N}{\mathbb{N}}

\newcommand{\R}{\mathbb{R}}
\newcommand{\C}{\mathbb{C}}

\renewcommand{\Pi}{\pi}
\renewcommand{\emptyset}{\varnothing}
\renewcommand{\hat}{\widehat}

\DeclareMathOperator*{\Int}{Int}

\newcommand\id{\mathrm{id}}
\newcommand\Chat{\hat{\mathbb{C}}}

\DeclareMathOperator*{\card}{card}
\DeclareMathOperator*{\supp}{supp}

\DeclareMathOperator*{\dist}{dist_{Rat}}

\DeclareMathOperator*{\Jpre}{\mathit{J}_{pre}}
\DeclareMathOperator*{\Fpre}{\mathit{F}_{pre}}
\DeclareMathOperator*{\Jur}{\mathit{J}_{ur}}
\DeclareMathOperator*{\Jr}{\mathit{J}_{r}}
\DeclareMathOperator*{\Aut}{Aut}
\DeclareMathOperator*{\CV}{CV}

\DeclareMathOperator*{\Rat}{Rat}

\usepackage{hyperref}

\@ifundefined{textmu}
 {\usepackage{textcomp}}{}

\makeatother

\usepackage{babel}
  \addto\captionsbritish{\renewcommand{\corollaryname}{Corollary}}
  \addto\captionsbritish{\renewcommand{\definitionname}{Definition}}
  \addto\captionsbritish{\renewcommand{\examplename}{Example}}
  \addto\captionsbritish{\renewcommand{\lemmaname}{Lemma}}
  \addto\captionsbritish{\renewcommand{\propositionname}{Proposition}}
  \addto\captionsbritish{\renewcommand{\remarkname}{Remark}}
  \addto\captionsbritish{\renewcommand{\theoremname}{Theorem}}
  \addto\captionsenglish{\renewcommand{\corollaryname}{Corollary}}
  \addto\captionsenglish{\renewcommand{\definitionname}{Definition}}
  \addto\captionsenglish{\renewcommand{\examplename}{Example}}
  \addto\captionsenglish{\renewcommand{\lemmaname}{Lemma}}
  \addto\captionsenglish{\renewcommand{\propositionname}{Proposition}}
  \addto\captionsenglish{\renewcommand{\remarkname}{Remark}}
  \addto\captionsenglish{\renewcommand{\theoremname}{Theorem}}
  \providecommand{\corollaryname}{Corollary}
  \providecommand{\definitionname}{Definition}
  \providecommand{\examplename}{Example}
  \providecommand{\lemmaname}{Lemma}
  \providecommand{\propositionname}{Proposition}
  \providecommand{\remarkname}{Remark}
\providecommand{\theoremname}{Theorem}
\providecommand{\claimname}{Claim}

\begin{document}
\author{Johannes Jaerisch}
\address{Department of Mathematics,  Faculty of Science and Engineering, Shimane University, Nishikawatsu 1060, Matsue, Shimane, 690-8504 Japan}
\email{jaerisch@riko.shimane-u.ac.jp}
\urladdr{http://www.math.shimane-u.ac.jp/$\sim$jaerisch/} \vspace{-2.5mm}
\author{Hiroki Sumi}
\address{Department of Mathematics, Graduate School of Science, Osaka University, 1-1 Machikaneyama, Toyonaka, Osaka, 560-0043 Japan }
\email{sumi@math.sci.osaka-u.ac.jp}
\urladdr{http://www.math.sci.osaka-u.ac.jp/$\sim$sumi/}
\title{Dynamics of infinitely generated nicely expanding rational
semigroups and the inducing method}



\date{21th February 2017.  \quad To appear in Trans. Amer. Math. Soc. \quad 2010 \emph{Mathematics Subject Classification.} 30D05, 37F15.}

\maketitle
\begin{abstract}
We investigate the dynamics of semigroups of rational maps on the Riemann sphere.  To establish  a fractal theory of the  Julia sets of  infinitely generated semigroups of rational maps, we introduce a new class of  semigroups which we call nicely expanding rational semigroups.  More precisely, we  prove Bowen's formula for the Hausdorff dimension of the  pre-Julia sets, which we also introduce in this paper.  We apply our  results to the study of the  Julia sets of non-hyperbolic rational semigroups.  For these results, we do not assume the cone condition, which has been assumed in 
the study of infinite contracting iterated function systems.  Similarly, we show that Bowen's formula holds for the limit set of a contracting conformal iterated function system without the cone condition.

\end{abstract}

\section{Introduction and Statement of Results}
\label{Introduction}
Let $\Rat$ be the set of all non-constant rational maps on the Riemann
sphere $\Chat$. 
A subsemigroup
of $\Rat$ with semigroup operation being the functional composition is
called a rational semigroup. 
A semigroup of non-constant polynomial maps is called a polynomial semigroup. 
The work on the dynamics of rational semigroups was initiated 
by A. Hinkkanen and G. J. Martin (\cite{MR1397693}), 
who were interested in the role of the dynamics of polynomial semigroups 
while studying various one-complex-dimensional moduli spaces for discrete 
groups of M\"{o}bius transformations, and by F. Ren's group 
(\cite{ZR92}), who studied such semigroups from the perspective 
of random dynamical systems. 
  The theory of the dynamics of rational semigroups on $\Chat $ 
has developed in many directions since the 1990s (\cite{MR1397693, ZR92, MR2900562, MR2813416,  MR2773145, 
SUtransversality}, \cite{MR1625944} -- \cite{cooperation}).  
We recommend \cite{MR2900562} as an introductory article.

Throughout, let   $I$ be  a topological space. We consider a family  $\{ f_i:i\in I\}$ of  $\Rat$ such that  $f_i \in \Rat$  depends  continuously  on $i\in I$. We will use $G=\left\langle f_{i}:i\in I\right\rangle $ to denote the rational semigroup  generated by $\{ f_i:i\in I\}$,  i.e.,  $G= \{ f_{i_1} \circ   \dots \circ f_{i_n} : n\in \N, i_1,  \dots ,i_n \in I \}$.
The Fatou set $F\left(G\right)$ and the Julia set $J\left(G\right)$  of  $G$ are given
by 
\[
F\left(G\right):=\left\{ z\in\Chat:G\mbox{ is normal in a neighborhood of }z\right\}\quad \text{and} \quad  J\left(G\right):=\Chat\setminus F\left(G\right). 
\]
Since the Julia set $J(G)$ of a rational semigroup 
$G=\langle f_{1},\ldots ,f_{m}\rangle :=\langle f_{i}: i\in \{ 1,\ldots, m\} \rangle$ 
generated by finitely many elements $f_{1},\ldots ,f_{m}$ 
has backward self-similarity,  i.e.,  
\begin{equation}
\label{bsseq}
 J(G)=f_{1}^{-1}(J(G))\cup \cdots \cup f_{m}^{-1}(J(G)),
\end{equation}  
(see \cite{MR1625944, MR1767945}), rational semigroups 
can be viewed as a significant generalization and extension of 
both the theory of iteration of rational maps (see \cite{MR1128089, MR2193309}) 
and conformal 
iterated function systems (see \cite{MR1387085}). 
Indeed, because of (\ref{bsseq}), 
 for the analysis of the Julia sets of rational semigroups, we have to consider  
 ``backward iterated functions systems'', however since each map 
$f_{j}$ is not injective and may  have critical points in general, we have to deal with critical orbits and some 
qualitatively different extra effort in the case  of semigroups is needed. 
Also, since one semigroup has many kinds of maps, we have another difficulty. The theory of the dynamics of 
rational semigroups borrows and develops tools 
from both of these theories. It has also developed its own 
unique methods, notably the skew product approach 
(see \cite{MR1767945}--\cite{MR2237476}, \cite{MR2736899, MR2747724, cooperation, SUtransversality}). 
It is a very  exciting problem to estimate the Hausdorff dimension of  Julia sets of rational semigroups. 
Some studies of the Hausdorff dimension of  Julia sets of semi-hyperbolic finitely generated 
rational semigroups were given in  \cite{MR1827119}--\cite{MR2237476}, \cite{MR2773145, SUtransversality}. 

However, there have been no studies on the Hausdorff dimension of Julia sets of infinitely generated expanding rational semigroups $G=\langle  f_i : i\in I  \rangle$ (see the definition below), or non-semi-hyperbolic rational semigroups.  In this paper, we investigate the dynamics of infinitely generated expanding rational semigroups and non-hyperbolic rational semigroups.  If $I$ is
 countable (in this paper, a countable set is a set which is bijective to a subset of $\Bbb{N}$), then $I$ is endowed with the discrete topology. We endow $\Rat$ with distance $\dist$ given by $\dist\left(h_{1},h_{2}\right):=\sup_{z\in\Chat}d\left(h_{1}\left(z\right),h_{2}\left(z\right)\right)$,
where $d$ denotes the spherical distance on $\Chat$. We denote by $C\!\left(I,\Rat\right)$ the set of continuous maps
from $I$ to $\Rat$.

Let $\left(f_{i}\right)_{i\in I}\in C\!\left(I,\Rat\right)$. The
skew product associated to the  generator system $\left\{ f_{i}:i\in I\right\} $
of the rational semigroup $G=\left\langle f_{i}:i\in I\right\rangle $
is given by 
\[
\tilde{f}:I^{\N}\times\hat{\C}\rightarrow I^{\N}\times\hat{\C},\quad\tilde{f}\left(\omega,z\right):=\left(\sigma\left(\omega\right),f_{\omega_{1}}\left(z\right)\right),
\]
where $\sigma:I^{\N}\rightarrow I^{\N}$ denotes the left shift defined 
by $\sigma\left(\omega\right)_{i}=\omega_{i+1}$,  for each $\omega\in I^{\N}$ and
$i\in\N$.   For  $\gamma=(\gamma _{i})\in G^{\N}$ we set 
\[
F_{\gamma}:=\left\{ z\in\Chat:\left(\gamma_{n}\circ\gamma_{n-1}\circ\dots\circ\gamma_{1}\right)_{n\in\N}\mbox{ is normal in a neighborhood of }z\right\} \mbox{ and }J_{\gamma}:=\Chat\setminus F_{\gamma}.
\]
Also, for  $\omega\in I^{\N}$, we set $\gamma\left(\omega\right):=\left(f_{\omega_{i}}\right)_{i\in\N}$,   $F_{\omega}:=F_{\gamma\left(\omega\right)}$ 
and $J_{\omega}:=J_{\gamma\left(\omega\right)}$. We use $J^{\omega}$ to denote the set $\left\{ \omega\right\} \times J_{\omega}\subset I^{\N}\times\Chat$ 
and we set 
\[
J\left(\tilde{f}\right):=\overline{\bigcup_{\omega\in I^{\N}}J^{\omega}},\quad F\left(\tilde{f}\right):=\left(I^{\N}\times\hat{\C}\right)\setminus J\left(\tilde{f}\right),
\]
where the closure is taken with respect to the product topology on
$I^{\N}\times\Chat$. 
For a holomorphic map $h:\Chat\rightarrow\Chat$
and $z\in\Chat$, the norm of the derivative of $h$ at $z\in\Chat$
with respect to the spherical metric is denoted by $\left\Vert h'\left(z\right)\right\Vert $.

\label{d:expanding}
For $n\in \N$ and $(\tau_1,\dots,\tau_n)\in I^n$, we set $f_{(\tau_1,\dots,\tau_n)}:=f_{\tau_{n}}\circ f_{\tau_{n-1}}\circ\dots\circ f_{\tau_{1}}$.  For $\omega \in I^\N$ and $n\in \N$, we write $\omega =(\omega_1,\omega_2,\dots)$ and we set $\omega |_{n}:=(\omega_1,\dots,\omega_n)$.   For each $n\in\N$ and $\left(\omega,z\right)\in J\left(\tilde{f}\right)$, we set $\left(\tilde{f}^{n}\right)'\left(\omega,z\right):=(f_{\omega |_{n}})'(z)$.  We say that 
$\tilde{f}$ is expanding along fibers if $J\left(\tilde{f}\right)\neq\emptyset$
and if there exist constants $C>0$ and $\lambda>1$ such that for
all $n\in\N$, 
\[
\inf_{\left(\omega,z\right)\in J\left(\tilde{f}\right)}\Vert\left(\tilde{f}^{n}\right)'\left(\omega,z\right)\Vert\ge C\lambda^{n},
\]
where $\Vert\left(\tilde{f}^{n}\right)'\left(\omega,z\right)\Vert$
denotes the norm of the derivative of $f_{\omega_{n}}\circ f_{\omega_{n-1}}\circ\dots\circ f_{\omega_{1}}$
at $z$ with respect to the spherical metric. $G$ is called expanding
with respect to $\left\{ f_{i}:i\in I\right\} $ if $\tilde{f}$ is
expanding along fibers. 

For a rational semigroup $G$, we say that a subset $A$ of $\hat{\Bbb{C}}$ is $G$-forward invariant 
if $g(A)\subset A$ for each $g\in G.$

Our first main definition is the following. 
\begin{defn} \label{nicelyexpanding}
We say that  $G=\langle  f_i : i\in I  \rangle$ is nicely expanding,   if $G$ is expanding with respect to $\left\{ f_{i}:i\in I\right\} $ and if there exists a non-empty, compact, $G$-forward invariant set $P_{0}\left(G\right)\subset F\left(G\right)$ such that  $P(G)\subset P_0(G)$,  where $P\left(G\right)$
denotes the postcritical set of $G$ given by  $P\left(G\right):=\overline{\bigcup_{g\in G} \{ \mbox{all critical values of } 
g: \Chat \rightarrow \Chat \}      }$. Here, the closure is taken in $\Chat$.
\end{defn}
\begin{rem*}
It will follow from Proposition \ref{prop:nicelyexpanding-characterisation}
below, that the property of a rational semigroup to be nicely expanding
is in fact independent of the choice of the generator system. For
this reason, we do not  refer to the generator system, and
we  simply say that $G$ is nicely expanding.
\end{rem*}
We say that a rational semigroup $G=\left\langle f_{i}:i\in I\right\rangle $
(or the system $\left\{ f_{i}:i\in I\right\} $) is hyperbolic if
$P\left(G\right)\subset F\left(G\right)$. 
\begin{rem*}
We have   $P\left(G\right)=\overline{\bigcup_{g\in G\cup \{ \id \} }g\left(\bigcup_{i\in I}\CV\left(f_{i}\right)\right)}$, where $\CV$ denotes the set of critical values. Thus  $P(G)$ is $G$-forward invariant. \end{rem*}

We give some criteria for $G$ to be  nicely expanding in Proposition~\ref{prop:nicelyexpanding-characterisation} and  Lemma~\ref{lem:expandingness-osc}.  If $G$ is nicely expanding then we are able to control the distortion of inverse branches of maps in  $G$. Note that an expanding rational semigroup is not nicely expanding in general  (see Examples~\ref{ex:nonnice2}, \ref{ex:nonnice1}). 

Regarding the dynamics of  infinitely generated nicely expanding rational semigroups,   it turns out that 
we shall work on pre-Julia sets,  which is the second main definition of  this paper.
\begin{defn}
[Pre-Fatou and pre-Julia set]For a rational semigroup $G$, 
 the pre-Fatou set $\Fpre\left(G\right)$ and the pre-Julia set   $\Jpre\left(G\right)$ of $G$ are defined by 
\[
\Fpre\left(G\right):=\bigcap_{\gamma\in G^{\N}}F_{\gamma}  \mbox{ and } \Jpre\left(G\right):=\Chat\setminus\Fpre\left(G\right).
\]
\end{defn}
Note that if $G=\left\langle f_{i}:i\in I\right\rangle $, then  $\Jpre(G)=\bigcup_{\omega \in I^{\Bbb{N}}}J_{\omega }$ and
$\Jpre(G)=\bigcup _{i\in I}f_{i}^{-1}(\Jpre(G)).$
The pre-Julia sets of infinitely generated nicely  expanding rational semigroups of this paper 
correspond to the limit sets of infinite contracting conformal iterated function systems  in \cite{MR1387085}  
(see Remark~\ref{r:RSandIFS}).  Roughly speaking, those are the reasons why the pre-Julia sets are the right objects to study and many results (e.g. Bowen's formula) hold for
them.

We remark that the pre-Julia set is  not necessarily closed in $\Chat$.
In fact, by the density of the repelling fixed points (\cite{MR1397693}, \cite[Lemma 2.3 (g)]{MR1767945}),
we have that,  if $\card\left(J\left(G\right)\right)\ge3$,  then 
$\overline{\Jpre\left(G\right)}=J\left(G\right)$.
However, there are many examples of rational semigroups $G$,  
for which  $\dim _{H}(\Jpre(G))<\dim _{H}(J(G))$ or  $\Jpre(G)\neq J(G)$ (see Example~\ref{ex:dimHdimB}),  
even if we assume that $G$ is nicely expanding.  

In order to state the main result of this paper,  we define the following critical exponents  associated to rational semigroups. 
If $I$ is countable and $G=\left\langle f_{i}:i\in I\right\rangle $, then the critical exponent $s(G)$ of the Poincar\'{e}  series of $G$  and the critical exponent $t(I)$ are  for each $x\in \Chat$  given by 
 \[
s\left(G,x\right):=\inf\left\{ v \ge 0:\sum_{g\in G}\sum_{y\in g^{-1}(x)}\left\Vert g'\left(y\right)\right\Vert ^{-v}<\infty\right\} ,\quad s\left(G\right):=\inf\left\{ s\left(G,x\right):x\in\Chat\right\}, 
\]
\[
t\left(I,x\right):=\inf\left\{ v\ge 0:\sum_{n\in\N} \sum _{\omega \in I^{n}}\sum _{y\in f_{\omega }^{-1}(x)}
\| f_{\omega }'(y)\| ^{-\nu }   <\infty\right\} ,\quad t\left(I\right):=\inf\left\{ t\left(I,x\right):x\in\Chat\right\} .
\]
Here, the sums $\sum _{y\in g^{-1}(x)}$ and $\sum _{y\in f_{\omega }^{-1}(x)}$ 
count the multiplicities, and we set  $\inf\left\{ \emptyset\right\} :=\infty$ and $0^{-v}:=\infty$, for all $v\ge 0$.

Let $I\subset\N$ be the finite set $\left\{ 1,\dots,n\right\} $,
for some $n\in\N$, or let $I=\N$, endowed with the discrete topology.
Let $\left(f_{i}\right)_{i\in I}\in\Rat^{I}$ and let $\tilde{f}:J\left(\tilde{f}\right)\rightarrow J\left(\tilde{f}\right)$
be the associated skew product. Suppose that $\|f_{\omega_{1}}'\left(z\right)\|\neq0$,
for each $\left(\omega,z\right)\in J\left(\tilde{f}\right)$. 
We introduce the  Gurevi\v c pressure of the geometric potential $\tilde{\varphi}:J\left(\tilde{f}\right)\rightarrow\R$,
$\tilde{\varphi}\left(\omega,z\right):=-\log\|f_{\omega_{1}}'\left(z\right)\|$,
with respect to the skew product $\tilde{f}:J\left(\tilde{f}\right)\rightarrow J\left(\tilde{f}\right)$.
This notion of topological pressure  was introduced  in the context of countable Markov shifts 
by Sarig (\cite{MR1738951})  extending the notion of topological entropy
due to Gurevi\v c (\cite{MR0263162}).

The pressure function of the system $\{ f_i : i\in I\}$ is for each $t\in\R$ given by 
\[
\mathcal{P}\left(t\right):=\mathcal{P}\left(t\tilde{\varphi},\tilde{f}\right):=\sup_{K\subset J\left(\tilde{f}\right),K\,\text{compact},  \tilde{f}(K)=K}\mathcal{P}\left(t\tilde{\varphi}_{|K},\tilde{f}_{|K}\right).
\]
Here, for a continuous function $\psi:X\rightarrow\R$
on a compact metric space $X$, and for a continuous dynamical system
$T:X\rightarrow X$, we use $\mathcal{P}\left(\psi,T\right):=\sup\{h_\mu(T)+\int \psi d\mu\}$ to denote
the  classical notion of topological pressure introduced
by Walters (\cite{MR0390180})  following the work of Ruelle (\cite{MR0417391}), where the supremum is taken over all $T$-invariant Borel probability measures on $X$,  and $h_\mu(T)$ refers to the measure-theoretic entropy of the dynamical system $(T,\mu)$. Note that the classical pressure is independent of the choice of the metric on $X$ (\cite{MR648108}). 
\begin{defn*} We say that the rational semigroup $G=\left\langle f_{i}:i\in I\right\rangle $
(or the system $\left\{ f_{i}:i\in I\right\} $) satisfies the \emph{open
set condition} if there exists a non-empty open set $U\subset\hat{\C}$
such that $\left(f_{i}^{-1}\left(U\right)\right)_{i\in I}$ consists
of mutually disjoint subsets of $U$.
\end{defn*}
We refer to \cite{MR556580} (see also \cite{MR684247}) for the by now classical results on the relation between the  pressure and the Hausdorff dimension of associated limit sets, which is  known as Bowen's formula. 
We now present the main result of this paper, which establishes Bowen's formula for pre-Julia sets. 
\begin{thm}[Bowen's formula for pre-Julia sets: see 
Theorem~\ref{thm:bowen-for-prejulia}, Proposition~\ref{prop:critical-exponents}]
\label{i:thm:bowen-for-prejulia} 
Let $I$ be a countable set. 
Let $G=\langle f_{i}:i\in I\rangle $ be a nicely expanding rational semigroup. 
Then we have 
\begin{equation}
\label{upperBowen}
\dim_{H}\left(\Jpre\left(G\right)\right)\le  s\left(G\right)\le 
t\left(I\right)=\inf \{ \beta \in \R : \mathcal{P}(\beta )<0\}. 
\end{equation}
If $\{ f_{i}:i\in I\}$ additionally satisfies the open set condition, then all inequalities in  (\ref{upperBowen}) become equalities.
\end{thm}
We apply the above result to  the dynamics of non-hyperbolic rational semigroups 
by using the method of inducing.  We deal with critical orbits which do not appear in contracting iterated function systems. 
\begin{thm}[Inducing method: see Theorem~\ref{thm:inducing-prejulia}]
\label{i:thm:inducing-prejulia}
Let $I$ be a countable set and let $G=\left\langle f_{i}:i\in I\right\rangle $
be a rational semigroup. 
Suppose that there exists a decomposition
$I=I_{1}\cup I_{2}$ with $I_2 \neq \emptyset$, such that each of the following (1)--(4) holds for the
rational semigroups $G_{j}:=\left\langle f_{i}:i\in I_{j}\right\rangle $,
$j\in\left\{ 1,2\right\} $, and $H:=\left\langle H_{0}\right\rangle $  given by 
$$ 
H_{0}:=\left\{ f_{i}:i\in I_{2}\right\} \cup\left\{ f_{i}f_{j_{1}}\dots f_{j_{r}}:i\in I_{2},r\in\N,\left(j_{1},\dots,j_{r}\right)\in I_{1}^{r}\right\} ,\quad \langle H_{0}\rangle := \langle g: g\in H_{0}\rangle.
$$ 
\begin{enumerate}
\item There exists an  $H$-forward invariant non-empty compact set $L\subset F\left(H\right)$ 
such that  $P\left(G_{2}\right)\subset L$ and 
 $f_{i}\left(P\left(G_{1}\right)\right)\subset L$, for each $i\in I_{2}$. 
\item $\deg\left(g\right)\ge2$ for all $g\in G$. 
\item There exists a $G$-forward invariant non-empty compact set $ L_{0}\subset F(G)$. 
\item  $\left\{ f_{i}:i\in I\right\} $ satisfies the open set condition.
\end{enumerate}
Then we have that $H$ is nicely expanding (we endow $H_{0}$ with the discrete topology), $H_0$ satisfies the open set condition, $s(H)=s(G)$ and 
\[
\dim_{H}\left(\Jpre\left(G\right)\right)=\max\left\{ s\left(G\right),\dim_{H}\left(\Jpre\left(G_{1}\right)\right)\right\} .
\]
If in addition to the assumptions, we have  $\card(I)<\infty$,   $f_i$  is a polynomial   for each $i\in I_1$, and if there exists a compact $G_1$-forward invariant subset $K\subset F(G_1)$ such that $f_j(P(f_i))\subset K$  for all $i,j\in I_1$ with $i \neq j$, then 
\[
\dim_{H}\left(J\left(G\right)\right)=\max\big\{ s\left(G\right), \max_{i\in I_1} \{ \dim_{H}\left(J(f_i)\right) \} \big\} .
\]
\end{thm} 
We point out that, even if the semigroup $G$ is finitely generated (e.g. $G$ has two generators), then the inducing method leads to an infinitely generated semigroup $H$, which is shown to be nicely expanding. This fact is one of the main motivations to develop Bowen's formula for infinitely generated semigroups. 

There are many applications of Theorem \ref{i:thm:inducing-prejulia}  (see Example \ref{inducing-example},
Theorem  \ref{t:c:B1main},   Corollary \ref{c:B1main}, Lemmas \ref{PBD-has-SOSC}, \ref{PBD-class}).
We are interested in the space ${\mathcal A}$ of couples $(f_{1},f_{2})$ of 
polynomials with $\deg (f_{i})\geq 2$  for each $i$,  for which 
the planar postcritical set $P(\langle f_{1},f_{2}\rangle )\setminus \{ \infty \} $ 
is bounded but the Julia set $J(\langle f_{1},f_{2}\rangle )$ is disconnected. 
It is well-known that for a polynomial $f$ with $\deg (f)\geq 2$, 
the Julia set $J(f)$ of $f$ is connected if and only if $P(f)\setminus \{ \infty \} $ is bounded 
(see \cite{MR2193309}).  
However, the space ${\mathcal A}$ is not empty, and this is a special phenomenon in the dynamics of polynomial semigroups. 
There have been some studies on the dynamics of the semigroups $\langle f_{1},f_{2}\rangle $ for elements $(f_{1},f_{2})\in \overline{\mathcal{A}}$ employing potential theory (see \cite{MR2773173, MR2736899, MR2553369, s10, twogenerator, MR2813416}). 
In this paper, we focus on elements $(f_{1},f_{2})\in \mathcal{A}$ and some elements $(f_{1},f_{2})\in \partial \mathcal{A}.$  
Applying Theorem~\ref{i:thm:inducing-prejulia}, we obtain the following. 
\begin{thm}[see Corollary~\ref{c:B1main}]
\label{t:c:B1main}
Let $f_{1}$ and $f_{2}$ be polynomials of
degree at least two. 
Let $G=\left\langle f_{1},f_{2}\right\rangle $. 
Suppose that all of the following hold.  
\selectlanguage{british}%
\begin{enumerate}
\item \label{i:enu:postcritically-bounded}$P(G)\setminus\left\{ \infty\right\} $
is a bounded subset of $\C$.
\item \label{i:enu:Kg1-in-interiorKg2}$K\left(f_{1}\right)\subset\Int K\left(f_{2}\right)$, where $K(f_{i})$ denotes 
the filled-in Julia set of $f_{i}.$ 
\item \label{i:enu:OSC} $\{f_1, f_2 \}$ satisfies the open set condition with the  open set   $\left( \Int K\left(f_{2}\right) \right) \setminus K\left(f_{1}\right)$.
\item \label{i:enu:.CVg2-in-interiorKg1}$\CV\left(f_{2}\right)\setminus\left\{ \infty\right\} \subset\Int K\left(f_{1}\right)$, 
where $\CV $ denotes the set of critical values. 
\end{enumerate} 
 Then we have \foreignlanguage{english}
{\textup{$\dim_{H}\left(J\left(G\right)\right)=\max \{ s\left(G\right), \dim _{H}(J(f_{1}))\} $. }}
\end{thm}  
 \vspace{-3mm}
\begin{rem*}A sufficient condition for $\dim _{H}(J(f_{1}))\leq s(G)$ is that $f_1$ is a 
non-recurrent critical point map (\cite{MR1279476}) or a Collet-Eckmann map (\cite{MR1407501}).  
\end{rem*}
Note that all elements $(f_{1},f_{2})\in \mathcal{A}$ and some elements $(f_{1},f_{2})\in \partial \mathcal{A}$ satisfy 
the assumptions of Theorem~\ref{t:c:B1main} (\cite{s10,Sumi11coliseum, twogenerator}). 
There are many examples of $(f_{1},f_{2})$ satisfying the assumption of 
Theorem~\ref{t:c:B1main}  (see Sections ~\ref{ExampleSection}, \ref{Inducing}). For example, for each polynomial $f_{1}$ such that
$J(f_{1})$ is connected and Int$K(f_{1})\neq \emptyset$,
there exists a polynomial $f_{2}$ such that
$(f_{1},f_{2})\in \mathcal{A}$ (\cite{MR2773173}).
Therefore even if $f_{1}$ with connected Julia set has a Siegel disk,
there exists  $f_{2}$ such that $(f_{1},f_{2})\in \mathcal{A}$
and  Theorem \ref{t:c:B1main} applies  to  $(f_{1},f_{2}).$ For the Julia sets of  the semigroups generated by  elements $\{ f_{1},f_{2}\} $ 
with $(f_{1},f_{2})\in \overline{\mathcal{A}}$, 
see Figures \ref{fig:Siegelcircle6} -- \ref{fig:caulicircle2}.  
\begin{rem}
\label{r:RSandIFS} 
Let $G=\langle f_{i}:i\in I\rangle $ be a rational semigroup with $G\subset \mbox{Aut}(\Chat)$, where $ \mbox{Aut}(\Chat)$ denotes the group of M\"{o}bius transformations on $\Chat$.
Suppose that $G$ satisfies the open set condition with a bounded connected open set $U$ in $\mathbb{C}.$ 
Suppose also that there exist two bounded open connected subsets $V_{1},V_{2}$ with $\overline{U}\subset V_{j}$, $j\in \{1,2\}$, 
 such that $f_{i}^{-1}(V_{1})\subset V_{2}$,  for each $i\in I.$ 
Suppose further that there exists a constant $0<s<1$ such that 
$|(f_{i}^{-1})'(z)|\leq s$,  for each $z\in \overline{U}$ and for each $i\in I.$ 
Then, $G$ is a nicely expanding rational semigroup and the system 
$\Phi =\{ f_{i}^{-1}:\overline{U}\rightarrow \overline{U}\} _{i\in I} $ is a contracting conformal 
iterated function system in the sense of \cite{MR1387085}, which does not necessarily satisfy the cone condition. Here, the cone condition refers to the property that there exist $\gamma,l>0$ such that, for every $z\in \overline{U}$, there exists an open cone $C$ with vertex $z$,  central angle $\gamma$ and altitude $l$ such that  $C\subset U$ (\cite[(2.7) on page 110]{MR1387085}). 
Moreover, 
we have that the pre-Julia set of $G$ is equal to the limit set of the system $\Phi .$
\end{rem}

\begin{rem}[see Section~\ref{Remarks}]
\label{i:r:cc}
For the results of this paper, the cone condition for the open set in the open set condition is not needed. We
will see in Section \ref{sec:Inducing-Methods} that there are many
examples of semigroups which do not satisfy the cone condition, and
for which our results can be applied. For such examples, see Section~\ref{sec:Inducing-Methods} and  
Figures~\ref{fig:caulicircle1}, \ref{fig:caulicircle2}. 
In \cite[Theorem 3.15]{MR1387085} it is proved that, for the Hausdorff
dimension of the limit set $J\left(\Phi\right)$ of an infinitely
generated contracting conformal iterated function system $\Phi$ satisfying 
the cone condition, we have 
\begin{equation}
\dim_{H}J\left(\Phi\right)=\inf\left\{ \delta:P\left(\delta\right)<0\right\} =\sup_{\Phi_F}\left\{ J\left(\Phi_{F}\right)\right\} .
\label{i:eq:cifs-bowen}
\end{equation}
Here, $P$ refers to the  associated pressure function,
and $\Phi_{F}$ runs over all finitely generated subsystems of $\Phi$.
By the methods employed in the proof of Theorem \ref{thm:bowen-for-prejulia} of this paper,
one can show that (\ref{i:eq:cifs-bowen}) holds, even if the cone condition
is not satisfied. Instead of the cone condition (2.7) in \cite{MR1387085}, we need to assume that $\vert\phi_i'(x)\vert \le s$, for each $x\in X$ in the notation of  \cite{MR1387085}. For the details, see Section~\ref{Remarks}.  
\end{rem}

By using Bowen's formula for pre-Julia sets for nicely expanding
rational semigroups which is established in this paper,
we will investigate the parameter dependence of
$\dim_{H}(\Jpre(G))$ for nicely expanding rational semigroups $G$ and
$\dim _{H}(J(S))$ for
non-hyperbolic finitely generated rational semigroups $S$,
which is a further interesting task.

Let us briefly comment on the history of the method of inducing. This method was used to investigate invariant measures of non-hyperbolic dynamical systems using  ideas of Schweiger (\cite{MR0384735}). In \cite{MR1107025} the method of inducing is used to develop the ergodic theory of  Markov-fibred systems with applications to parabolic rational maps. Using  results on the thermodynamic formalism for  symbolic dynamical systems with a countable alphabet (\cite{MR1387085, MR1738951}),   the method of inducing was used to prove Bowen's formula for parabolic iterated function systems (\cite{MR1786722}, see also \cite{MR2003772}). 

The theory of the dynamics of rational semigroups is intimately 
related to that of the random dynamics of rational maps. 
The first study of random complex dynamics was given in \cite{MR1145616}. 
For a  recent study, see \cite{MR1721617, MR2866474, MR2747724, MR2736899, 
s10, Sumi11coliseum, cooperation}.  
The deep relation between these fields 
(rational semigroups, random complex dynamics, and backward iterated function systems) 
is explained in detail in the papers (\cite{MR1827119}--\cite{twogenerator}) of the second author and in
\cite{SUtransversality}. For a random dynamical system generated by a family of 
 polynomial maps on $\Chat $, 
 let the function  $T_{\infty }:\Chat \rightarrow [0,1]$ be given by the probability of tending to $\infty \in \Chat .$ 
In \cite{MR2747724, Sumi11coliseum} 
it was shown that under certain conditions, 
$T_{\infty }$ is continuous on $\Chat $ and varies only on the Julia set of the associated rational semigroup 
(further results were announced in \cite{s10}).  
For example, there exists  a random dynamical system,  for which 
$T_{\infty }$ is continuous on $\Chat $ and the set of varying points of $T_{\infty }$ 
is equal to the Julia set of  Figure \ref{fig:Siegelcircle6},   Figure \ref{fig:caulicircle1} or Figure~\ref{fig:caulicircle2}, 
which is a thin fractal set.  
This function $T_{\infty }$ is a complex analogue of the devil's staircase (Cantor function) 
or Lebesgue's singular functions and this is called a ``devil's coliseum'' 
(see \cite{MR2747724, Sumi11coliseum, cooperation, s10, twogenerator}).   
From this point of view also, it is very interesting and important to 
investigate the figure, the properties
and the dimension of the Julia sets of rational semigroups.

 The outline of this paper is as follows. In Section 2, we give various interesting examples which 
motivate the study in this paper. In Section~\ref{Preliminaries}, we give  some basic  definitions and results on rational semigroups and their associated skew products.  In Section~\ref{Expanding}  we investigate (nicely) expanding rational semigroups.  Proposition \ref{prop:nicelyexpanding-characterisation} is the key to investigating infinitely generated nicely expanding rational semigroups. In Section~\ref{Topological}, we consider the associated skew products  systems whose phase spaces are not compact,  and we derive basic properties of two notions of topological
pressure.  We use results on equidistributional measures from \cite{MR1767945}. Also, we use some idea similar to the finitely primitive condition 
(\cite{MR2003772}) for topological Markov chains with an infinite alphabet. In Section~\ref{Bowen}, we establish  Bowen's formula for  pre-Julia sets of (possibly infinitely generated) nicely expanding rational semigroups 
and we prove the main theorem 
(Theorem~\ref{i:thm:bowen-for-prejulia}, Theorem~\ref{thm:bowen-for-prejulia}) of this paper. To verify the lower bound of the Hausdorff dimension in Bowen's formula, we use a reduction  to the finitely generated case and
we apply  \cite[Theorem B]{MR2153926}.
In Section~\ref{Inducing}, by applying the main theorem to the dynamics of non-hyperbolic rational semigroups 
and by using the method of inducing, we prove Theorems \ref{i:thm:inducing-prejulia}, 
\ref{thm:inducing-prejulia}, and \ref{t:c:B1main}. 
In Section~\ref{Remarks}, we give some remarks on the cone condition which has been assumed in the study of 
infinite contracting iterated function systems. 

Proposition \ref{prop:nicelyexpanding-characterisation} is not a simple generalization of finitely generated case.
We use a completely new idea based on careful observations 
on the hyperbolic metric in the proof. 
Proposition \ref{prop:nicelyexpanding-characterisation} is used to prove results in Section~\ref{Inducing}. 
In Section~\ref{Inducing}, we also use some observations on the family 
$\{J_{\omega }\} _{\omega \in I^{\N}}$ of 
fiberwise Julia sets. The ideas in the proof of Proposition \ref{prop:nicelyexpanding-characterisation}  and Section~\ref{Inducing} 
are new and have not been used in the study of iteration dynamics of holomorphic maps and conformal IFSs so far.  

\section{Examples}
\label{ExampleSection}
In this section we give various interesting examples of (nicely) expanding 
rational semigroups and non-hyperbolic rational semigroups which motivate 
the study of this paper.

First we give an example of an infinitely generated expanding M\"{o}bius
semigroup satisfying the open set condition, which is not nicely expanding.
Note that this does not happen for finitely generated rational semigroups. We say that $g\in \Aut(\hat{\Bbb{C}})\setminus \{ \id\} $ is loxodromic 
if $g$ has two fixed points for which the modulus of the multiplier is not 
equal to one.
\begin{example}
\label{ex:nonnice2}
Let $\left(a_{i}\right)\in\mathbb{D}^{\N}$ and $\left(b_{i}\right)\in\left(\C\setminus\mathbb{D}\right)^{\N}$
be two sequences of pairwise distinct points, which have a common
accumulation point $a_\infty \in \mathbb{S}:=\partial(\mathbb{D})$. For each $i\in\N$
we choose a loxodromic M\"{o}bius transformation $f_{i}$ with repelling
fixed point $a_{i}$ and attracting fixed point $b_{i}$. Then  there
exists a sequence $\left(n_{i}\right)\in\N^{\N}$ such that
$\{f_{i}^{n_{i}} : i\in \N \} $ satisfies the open set condition
with respect to $\mathbb{D}$, and such that $\sup_{x\in\overline{\mathbb{D}}}\Vert\big(f_{i}^{-n_{i}}\big)'\left(x\right)\Vert\le1/2$,
for each $i\in\N$. Set $G:=\left\langle f_{i}^{n_{i}}:i\in\N\right\rangle $. Clearly, we have that $J(G)\subset \overline{\mathbb{D}}$ and that  $G$ is expanding with respect to $\left\{ f_{i}^{n_{i}}:i\in\N\right\} $.  However, $G$ is not nicely expanding. To prove this, let $P_0$ denote a non-empty compact $G$-forward invariant subset of $F(G)$. Then we have  $\left\{ b_{n}:n\in\N\right\} \subset P_{0}$, which implies $a_\infty \in P_0$.  Moreover, since $a_\infty$ is an accumulation point of the repelling
fixed points $\left(a_{n}\right)$, we have  $a_\infty \in J\left(G\right)$. Hence,  $G$ is not nicely expanding. 
\end{example}

The following example shows that an infinitely generated expanding
rational semigroup, satisfying the open set condition, is  not necessarily hyperbolic.  In particular, such a rational semigroup is not nicely expanding. Note that this can not happen for finitely generated rational semigroup (cf. \cite[Remark 5]{MR2153926}). 
\begin{example}
\label{ex:nonnice1}
Let $\left(a_{i}\right)_{i\in\N}\in\mathbb{D}^{\N}$ be a sequence
of pairwise distinct points in the open unit disc $\mathbb{D}$, such that $\left(a_{i}\right)_{i\in\N}$
has an accumulation point $a_{\infty}\in \mathbb{S}$. Let $\left(r_{i}\right)\in\R^{\N}$
be a sequence such that the sets  $\overline{B\left(a_{i},r_{i}\right)}$, $i\in \N$, 
are pairwise disjoint. There exists a sequence $\left(b_{i}\right)\in\R^{\N}$
such that the quadratic polynomials $h_i$, given by  $h_{i}\left(z\right):=b_{i}\left(z-a_{i}\right)^{2}+a_{i}$, 
satisfy $J\left(h_{i}\right)=\partial B\left(a_{i},r_{i}\right)$, 
for each $i\in\N$. Set $V:=\mathbb{D}\setminus\bigcup_{i\in\N}\overline{B\left(a_{i},r_{i}\right)}$.
Observe that there exists a sequence $\left(n_{i}\right)\in\N^{\N}$
such that  $\{ f_i : i\in \N\}$, given by $f_{i}:=h_{i}^{n_{i}}$,
satisfies the open set condition with respect to $V$. Clearly, we have $J\left(f_{i}\right)=J\left(h_{i}\right)=\partial B\left(a_{i},r_{i}\right)$.
To show that $G:=\left\langle f_{i}:i\in I\right\rangle $ is
expanding, let $\tilde{f}:I^{\N}\times\Chat\rightarrow I^{\N}\times\Chat$
denote the associated skew product. Let $\omega\in I^{\N}$ and $z\in J_{\omega}$ be given. Since $J\left(G\right)\subset\overline{V}$,
we have  $z\in\overline{V}$ and $f_{\omega_{1}}\left(z\right)\in\overline{V}$.
Since $\left|f_{\omega_{1}}'\left(z\right)\right|\ge2$, and since
the spherical metric and the Euclidean metric are equivalent on $\overline{V}$,
we obtain the $G$ is expanding with respect to $\left\{ f_{i}:i\in I\right\} $.
Finally, we observe that $G$ is not hyperbolic, because  $a_\infty \in J\left(G\right)\cap P\left(G\right)$. 
\end{example}

In the next example, we show that there exists a nicely expanding  infinitely generated
rational semigroup  $G=\left\langle f_{i}:i\in I\right\rangle $,  for which
$2=\dim_{H}\left(J\left(G\right)\right)=\dim_{B}\left(J\left(G\right)\right)>s\left(G\right)=t\left(I\right)=\dim_{H}\left(\Jpre\left(G\right)\right)$
and $\dim_{H}\left(\Jpre\left(G\right)\right)<\overline{\dim_{B}}\left(\Jpre\left(G\right)\right)=\dim_{P}\left(\Jpre\left(G\right)\right)=\overline{\dim_{B}}\left(\overline{\Jpre\left(G\right)}\right)=\dim_{B}\left(J\left(G\right)\right)=2$, where  $\overline{\dim_{B}}$ denotes the upper box dimension and $\dim_{P}$ denotes the packing dimension. The idea is as follows: we put infinitely many repelling fixed points such that the Hausdorff dimension of the closure of the set of repelling fixed points is equal to two. Simultaneously, we can make the multipliers sufficiently large to make the critical exponent close to one.
\begin{example}
\label{ex:dimHdimB}
Let $V$ denote a bounded  open set in $\Bbb{C}$ for which $\dim_{H}\partial V=2$.
For convenience, suppose that $\overline{\mathbb{D}}\subset V$. Let
$\left(a_{i}\right)\in\left(V\setminus\overline{\mathbb{D}}\right)^{\N}$
be a sequence of pairwise distinct points such that $\partial V\subset\overline{\left\{ a_{i}:i\in\N\right\} }$. We assume that $\{ a_{i}:i\in\N\}$ is discrete in $V$. 
Let $\left\{ f_{i}:i\in\N\right\} $ be a generator system such that
the following holds for the rational semigroup $G:=\left\langle f_{i}:i\in \N \right\rangle $.
Let $f_{1}$ be given by $f_{1}\left(z\right)=z^{d}$, for some $d\ge2$
to be specified later. For each $j\in\N$, $j\ge2$, let $f_{j}$
be given by $f_{j}\left(z\right)=\alpha_{j}\left(z-a_{j}\right)+a_{j}$,
for some sequence $\left(\alpha_{j}\right)_{j=2}^\infty$ with
$\alpha_{j}>1$, such that $\left\{ f_{i}:i\in \N \right\} $ satisfies
the open set condition with respect to $V$. We may also assume that
$P\left(G\right)\cap\overline{V}=P\left(f_{1}\right)=\left\{ 0\right\} $.
Since $\left\{ f_{i}:i\in \N \right\} $ satisfies the open set condition
with respect to $V$, we have that $J\left(G\right)\subset\overline{V}$,
which implies that $G$ is hyperbolic. Further, we have that  $G$ is nicely  expanding by Lemma \ref{lem:expandingness-osc}.  Finally, since the repelling fixed
points $\left\{ a_{i}:i\in \N\right\} $ are contained in $J\left(G\right)$, we have  $\dim_{H}\left(J\left(G\right)\right)=\dim_{B}\left(J\left(G\right)\right)=2$. 

Next, we show that for each $t>1$ there exist $d\ge2$ and $\left(\alpha_{j}\right)_{j=2}^\infty$
such that $s\left(G\right)\le t$ (thus $G$ depends on $t$). In order to show it, let $x\in J(G)$. Since the spherical metric and the
Euclidean metric are equivalent on  $J\left(G\right)$,
there  exists   a constant $C>0$ such that,  for each $d\ge2$, 
\[
\eta (x):=\sum_{y\in f_{1}^{-1}\left(x\right)}\|f_{1}'\left(y\right)\|^{-t}+\sum_{j=2}^{\infty}\sum_{y\in f_{j}^{-1}\left(x\right)}\|f_{j}'\left(y\right)\|^{-t}\le Cdd^{-t}+C\sum_{j=2}^{\infty}\alpha_{j}^{-t}.
\]
Choose  $d\ge2$ and the sequence  $\left(\alpha_{j}\right)_{j=2}^\infty$ sufficiently
large such that  $\sup _{x\in J(G)}\eta (x)<1$.  Hence, we have $s\left(G\right)\le t$. Moreover, by Theorem \ref{i:thm:bowen-for-prejulia}, we have $s(G)=t(\N)=
\dim _{H}\Jpre(G).$ We have thus shown that $2=\dim_{H}\left(J\left(G\right)\right)=\dim_{B}\left(J\left(G\right)\right)>s\left(G\right)=t\left(\N\right)=\dim_{H}\left(\Jpre\left(G\right)\right)$. 
\end{example}

We give some  examples to which we can apply Theorems \ref{i:thm:inducing-prejulia} and \ref{t:c:B1main}. Recall that the filled-in Julia set $K(g)$ of a polynomial $g$ is defined by \vspace{-2mm} \[
K\left(g\right):=\left\{ z\in\C:\left(g^{n}\left(z\right)\right)_{n\in\N}\mbox{ is bounded}\right\} .
\]
\begin{example} \label{inducing-example} Let  $I:= \{ 1,\dots,m+\ell \}$, $m,\ell \in \N$,  and let  $\{ h_i : i\in I \}$ be polynomials of degree at least two. Set  $I_2:=\{m+1, \dots , m+\ell \}$ and  suppose that $h_i $ is hyperbolic, for  each $i\in I_2$. Suppose that $K(h_i)$ is connected, for each $i\in I$, and that $K(h_i)\cap K(h_j)=\emptyset$, for all $i,j\in I$ with $i\neq j$.  Let $R>0$ such that $K(h_i)\subset B(0, R)$, for each $i\in I$. Then there exists $N\in \N$ such that $h_i^{-N}(\overline{B(0,R)}) \subset B(0,R)$, for each $i\in I$, and that $\{h_i^N : i\in I \}$ satisfies the open set condition with open set $B(0,R)$. Set $f_i:=h_i^N$ and consider the rational semigroup $G:=\langle f_i : i\in I \rangle $.   Set $A:=\Chat \setminus B(0,R)$ and observe that $A$ is $G$-forward invariant.  Since $J(G)=\bigcup_{i\in I}f_i^{-1}(J(G))\subset \bigcup_{i\in I}f_i^{-1}(\overline{B(0,R)})\subset B(0,R)$, we have that $A\subset F(G)$. For all $i,j\in I$ with  $i\neq j$ we have $f_j(K(f_i))\subset A$, because $f_j (K(f_j))\subset B(0,R)$ and $f_i^{-1}(B(0,R)) \cap f_j^{-1}(B(0,R))= \emptyset  $.  To see that $G$ satisfies the assumption (1) in Theorem \ref{i:thm:inducing-prejulia}, we set $L:=A \cup P(G_2)$.   $L$ is $G$-forward invariant because $g(P(G_2)) \subset g( \bigcup_{i\in I_2}P(f_i)\cup A)\subset  \bigcup_{i\in I_2}P(f_i) \cup A \subset L$. To prove that $P(G_2)\subset  F(G)$, it suffices to prove that $P(f_i)\subset F(G)$, for each $i\in I_2$. We observe that $P(f_i)\setminus \{ \infty \}\subset \Int(K(f_i))$ because $K(f_i)$ is connected and  $f_i$ is hyperbolic, for each $i\in I_2$.  Since $g(K(f_i)) \subset A\cup  K(f_i)$,  for each $g\in G$, it follows from Montel's Theorem that $P(f_i)\subset F(G)$.  We have thus shown that $G$ satisfies the assumptions of Theorem \ref{i:thm:inducing-prejulia}. Hence, we have $\dim_{H}\left(J\left(G\right)\right)=\max\left\{ s\left(G\right), \max_{i\in \{ 1,\dots , m\}} \{\dim_{H}\left(J(f_i)\right)\} \right\}$. If additionally, each $f_i$ is a non-recurrent critical point map, then  $\dim_{H}\left(J\left(G\right)\right)=s(G)$ (see Figure \ref{fig:2cauli1circle1b} for an example).  Using this formula, the numerical value of the Hausdorff dimension may be computed but we do not do that in this paper.  Also, using this formula the parameter-dependence of $\dim _{H}(J(G))$ may be investigated which is a further interesting task. \end{example}
 \vspace{0mm}
\begin{figure}[h]
\includegraphics[scale=0.7]{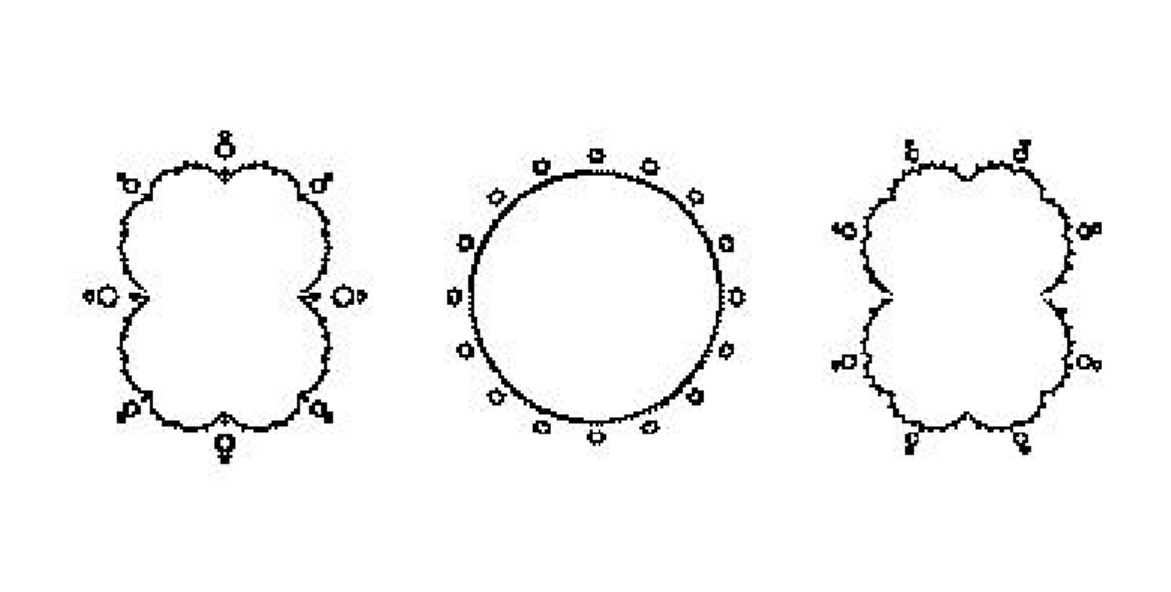}
 \vspace{-12mm}\caption{The Julia set of a $3$-generator polynomial semigroup $G$. $G=\left\langle f_1, f_2, f_3\right\rangle $ is given by  $f_i =h_i^3$, $i\in \{ 1,2,3\}$, where $h_1(z)=(z+3)^2+0.25 - 3$, $h_2(z)=z^2$ and $h_3(z)=(z-3)^2+0.25 + 3$.  $f_1$ and $f_3$ are not hyperbolic, but they are non-recurrent critical point maps. We have $\dim_{H}\left(J\left(G\right)\right)=s\left(G\right)$.}
\label{fig:2cauli1circle1b}
\end{figure}

\begin{defn}
[PB-D] We say that $G=\left\langle f_{1},f_{2}\right\rangle $
satisfies PB-D, if $f_{1}$ and $f_{2}$ are polynomials of
degree at least two, such that each of the following holds. 

\begin{enumerate}
\item $P(G)\setminus\left\{ \infty\right\} $
is a bounded subset of $\C$.
\item $J(G)$ is disconnected.
\end{enumerate}
\end{defn}

\begin{lem} \label{PBD-has-SOSC} If $G=\left\langle f_{1},f_{2}\right\rangle $ satisfies PB-D, then it also satisfies the assumptions of Theorem \ref{t:c:B1main}  (by renumbering $f_{1}$ and $f_{2}$
if necessary), and the  open set condition is satisfied.\end{lem}
\begin{proof}
\cite{twogenerator} or \cite[Proof of Theorem 2.11, Claim 2]{Sumi11coliseum}.\end{proof}
\begin{lem}\label{PBD-class}
Let $f_{1}$ be a polynomial of degree at least two with $\Int\left(K\left(f_{1}\right)\right)\neq\emptyset$
such that $K\left(f_{1}\right)$ is connected. Let $b\in\Int\left(K\left(f_{1}\right)\right)$.
Let $d\in\N$ with $d\ge2$ such that $\left(\deg\left(f_{1}\right),d\right)\neq\left(2,2\right)$.
Then, there exists a  number  $c>0$ such that for each $a\in\C$ with $0<\left|a\right|<c$, setting $f_{2}:=a\left(z-b\right)^{d}+b$, the polynomial  semigroup $G=\left\langle f_{1},f_{2}\right\rangle $
satisfies PB-D.\end{lem}
\begin{proof}
\cite[Proposition 2.40]{MR2773173} and Lemma \ref{PBD-has-SOSC}.\end{proof}
\begin{rem}
If $f_{1}$ is a non-recurrent critical point map or a Collet-Eckmann map, 
then $\dim_{H}\left(J\left(f_{1}\right)\right)\le s\left(\left\langle f_{1}\right\rangle \right)$
by (\cite{MR1279476}) and (\cite{MR1407501}). Thus, if in addition to this assumption, $G=\left\langle f_{1},f_{2}\right\rangle$ satisfies the assumptions of Theorem \ref{t:c:B1main}, then
 $\dim_{H}\left(J\left(G\right)\right)=s\left(G\right)$. 
\end{rem}

\begin{figure}[h]
\includegraphics[scale=0.45]{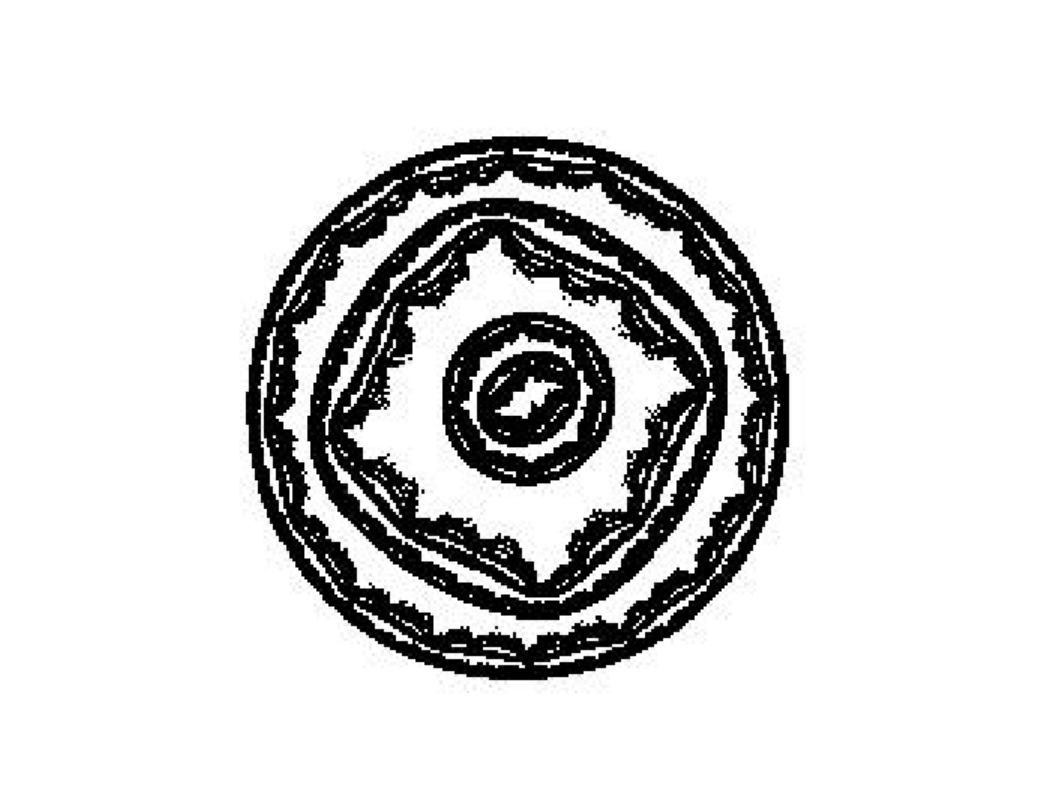}
 \vspace{-6mm}\caption{The Julia set of $G=\left\langle f_1, f_2 \right\rangle $,  where $f_1(z)=z^{2}+\e^{2 \pi i \sqrt[3]{0.25}}z$, $f_2= h_2^2$ and  $h_2(z)=0.1z^2$.   $G$ satisfies PB-D. $f_1$ has a Siegel disc with center in $0$. We have $\dim_{H}\left(J\left(G\right)\right)=\max \{ s\left(G\right), \dim _{H}(J(f_{1}))\}$.  }
\label{fig:Siegelcircle6}
\end{figure}

\begin{figure}[h]
\includegraphics[scale=0.25]{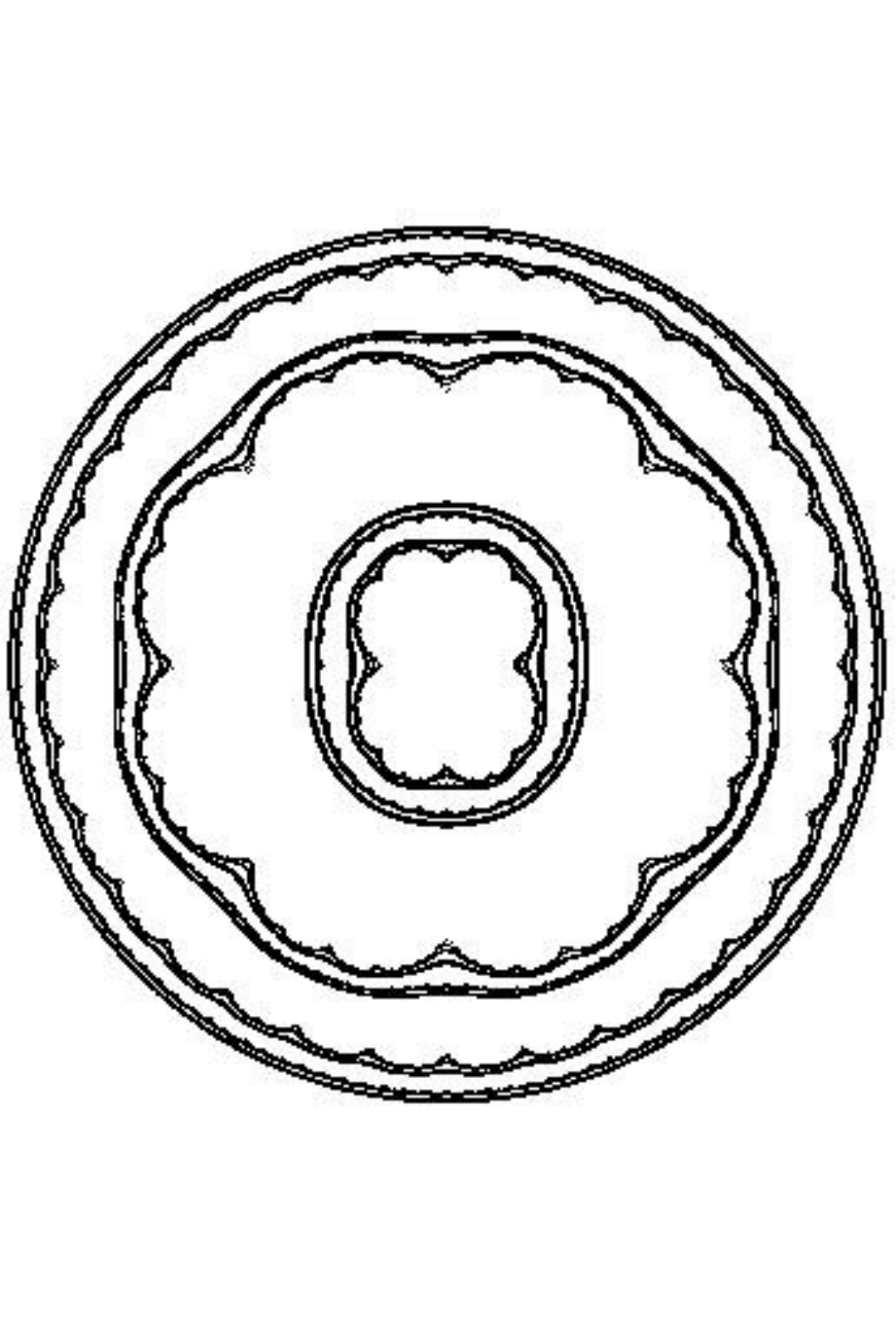}
\caption{The Julia set of $G=\langle f_1, f_2\rangle$ satisfying PB-D, where $f_1$ is a non-hyperbolic map (but is a non-recurrent critical point map).  $J\left(G\right)$ is disconnected. The
cone condition is not satisfied. We have $\dim_{H}\left(J\left(G\right)\right)=s\left(G\right)=s(\langle f_{2}\circ f_{1}^{r}: r\in \N\cup \{ 0\}\rangle )$. }
\label{fig:caulicircle1}
\end{figure}
 \vspace{-5mm}
\begin{figure}[h]
\includegraphics[scale=0.3]{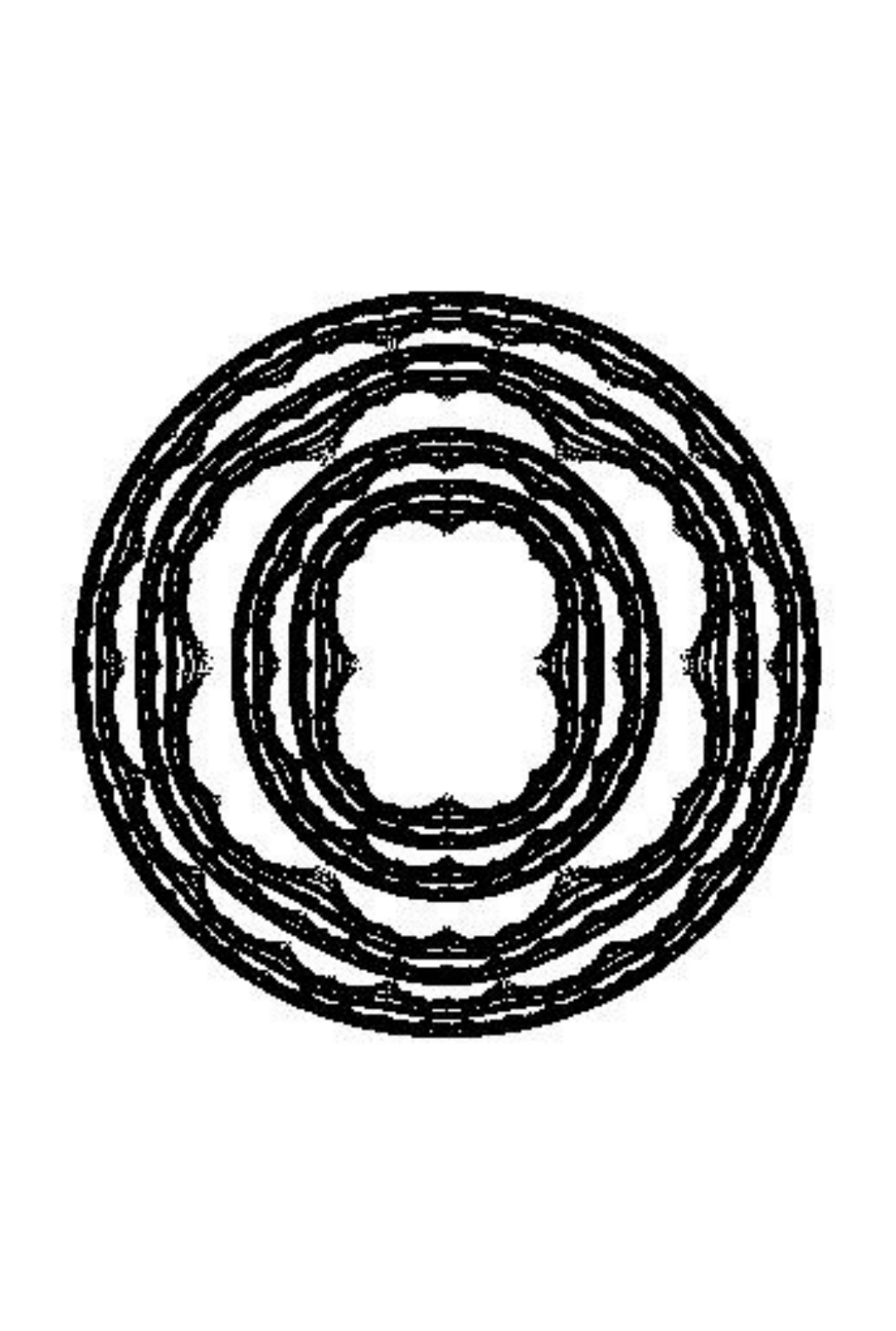}
 \vspace{-8mm}\caption{The Julia set of $G=\left\langle z^{2}+\frac{1}{4},az^{3}\right\rangle $, where $a\in \mathbb{C}$ is a complex number. $G$ satisfies the assumptions of Theorem \ref{t:c:B1main} and  $J\left(G\right)$ is connected.
The cone condition is not satisfied. We have $\dim_{H}\left(J\left(G\right)\right)=s\left(G\right)=s(\langle f_{2}\circ f_{1}^{r}: r\in \N\cup \{ 0\}\rangle )$.}
\label{fig:caulicircle2}
\end{figure}

\begin{rem*}
Regarding Figure \ref{fig:caulicircle2}, we remark there exists $a\in \mathbb{C}$ such that for
$f_{1}\left(z\right):=z^{2}+\frac{1}{4}$ and $f_{2}\left(z\right):=az^{3}$,
we have that $f_{1}^{-1}\left(J\left(G\right)\right)\cap f_{2}^{-1}\left(J\left(G\right)\right)\neq\emptyset$ (see \cite{twogenerator}).
Then, it follows from  \cite[Theorem 1.5, Theorem 1.7]{MR2553369}
that $J\left(G\right)$ is connected. Like this, for each polynomial $f_1$ of degree at least two with $\Int K(f_1) \neq \emptyset$ such that $K(f_1)$ is connected, there are many examples of $G=\langle f_1, f_2 \rangle$ satisfying the assumptions of Theorem \ref{t:c:B1main} for which  $J\left(G\right)$ is
connected (see \cite[Theorem 2.27 and its proof]{twogenerator}).
\end{rem*}

\section{Preliminaries on rational semigroups and skew products}
\label{Preliminaries}

In this section, we collect some of the basic results on rational
semigroups and the associated skew products.

\begin{defn}
Let $G$ be a rational semigroup and let  $z\in\Chat$. The backward orbit
$G^{-}\left(z\right)$ of $z$ and the set of exceptional points $E\left(G\right)$
are defined by $G^{-}\left(z\right):=\bigcup_{g\in G}g^{-1}\left(z\right)$
and $E\left(G\right):=\left\{ z\in\Chat:\card\left(G^{-}\left(z\right)\right)<\infty\right\} $.  We say that  a set $A\subset \Chat$ is $G$-backward invariant, if $g^{-1}(A)\subset A$, for each $g\in G$.
\end{defn}
\vspace{-3pt}We refer to \cite{MR1397693,MR1767945} for the fundamental properties of rational semigroups and their Julia sets.  

We review some of the basics of the skew product
associated to rational semigroups introduced in \cite{MR1767945}.
We will always assume that $I$ is a topological space. We denote by ${\mathcal S}$ the set of all non-constant polynomial maps on $\Chat $ endowed with 
the relative topology inherited from $\Rat$. Note that,  for each $d\in \mathbb{N}$,  the subspace $\Rat _{d}:=\{ f\in \Rat : \deg (f)=d\} $ of $\Rat $ is a connected component  
of $\Rat $,  and $\Rat _{d}$  is an open subset of $\Rat $.  Similarly, 
the subspace ${\mathcal S}_{d}:=\{ f\in {\mathcal S}: \deg (f)=d\}$ of 
${\mathcal S}$ is a connected component of ${\mathcal S}$,  and ${\mathcal S}_{d}$ is an open subset of ${\mathcal S}$.  
 A sequence $\{ f_{n}\} $ tends to $f$ in ${\mathcal S}$ if and only if 
there exists a number $N\in \mathbb{N}$ such that $\deg (f_{n})=\deg (f)$, 
for each $n\ge N$,   and if the coefficients of $f_{n}$ converge to those of $f$ appropriately. 
For the topology of $\Rat$ and ${\mathcal S}$, see \cite{MR1128089}.

Let $\pi_{1}:I^{\N}\times\Chat\rightarrow I^{\N}$ and
$\pi_{\Chat}:I^{\N}\times\Chat\rightarrow\Chat$ denote the canonical
projections. For each $F\subset I$, we also set 
\[J^{F}:=\bigcup_{\omega\in F^{\N}}J^{\omega} \mbox{ and }
J_{F}:=\bigcup_{\omega\in F^{\N}}J_{\omega}.
\] 
For a finite word $\omega =(\omega _{1},\omega _{2},\ldots ,\omega _{n})\in 
I^{n}$ and an infinite word $\alpha =(\alpha _{1},\alpha _{2},\ldots )\in I^{\N}$, we set 
$\omega \alpha =(\omega _{1},\ldots, \omega _{n}, \alpha _{1},\alpha _{2},\ldots )\in I^{\N}.$

\begin{prop} \label{prop:skewproduct-facts} Let $\left(f_{i}\right)_{i\in I}\in C\!\left(I,\Rat\right)$ and let
$\tilde{f}:I^{\N}\times\hat{\C}\rightarrow I^{\N}\times\hat{\C}$
be the skew product associated to the generator system $\left\{ f_{i}:i\in I\right\} $. 
Then we have the following. 
\begin{enumerate}
\item \label{enu:ratsemi-facts-1}$\tilde{f}\left(J^{\omega}\right)=J^{\sigma\omega}$
and $(\tilde{f}_{|\pi_{1}^{-1}\left(\omega\right)})^{-1}\left(J^{\sigma\omega}\right)=J^{\omega}$,
for each $\omega\in I^{\N}$. 
\item \label{enu:ratsemi-facts-2}$\tilde{f}\left(J\left(\tilde{f}\right)\right)=J\left(\tilde{f}\right),$
$\tilde{f}^{-1}\left(J\left(\tilde{f}\right)\right)=J\left(\tilde{f}\right)$,
$\tilde{f}\left(F\left(\tilde{f}\right)\right)=F\left(\tilde{f}\right),$
$\tilde{f}^{-1}\left(F\left(\tilde{f}\right)\right)=F\left(\tilde{f}\right)$. 
\selectlanguage{english}%
\item \textup{\label{enu:ratsemi-facts-3}Suppose that $I$ is a finite
set endowed with the  discrete topology. Let $G=\left\langle f_{i}:i\in I\right\rangle $
and suppose that $\card\left(J\left(G\right)\right)\ge3$. Then we
have $J\left(\tilde{f}\right)=\bigcap_{n\in\N_{0}}\tilde{f}^{-n}\left(I^{\N}\times J\left(G\right)\right)$
and $\pi_{\Chat}\left(J\left(\tilde{f}\right)\right)=J\left(G\right)$. Here, $\N_0 := \N \cup \{ 0\}$.}
\end{enumerate}
\end{prop}
\vspace{-6mm}
\begin{proof}
The proof of (\ref{enu:ratsemi-facts-1}) is straightforward and therefore
omitted. We give a proof of (\ref{enu:ratsemi-facts-2}), following
\cite[Lemma 2.4]{MR1827119}. The inclusion $\tilde{f}\left(J\left(\tilde{f}\right)\right)\subset J\left(\tilde{f}\right)$
is obvious. In order to prove $\tilde{f}^{-1}\left(J\left(\tilde{f}\right)\right)\subset J\left(\tilde{f}\right)$,
let $\left(\omega,x\right)\in\tilde{f}^{-1}\left(J\left(\tilde{f}\right)\right)$
be given. Clearly, $\left(\sigma\omega,f_{\omega_{1}}\left(x\right)\right)\in J\left(\tilde{f}\right)$.
For each neighborhood $U$ of $\sigma\omega$ and for each neighborhood 
$V$ of $f_{\omega_{1}}\left(x\right)$, there exists $\left(\alpha,y\right)\in U\times V$
such that $y\in J_{\alpha}$ and $f_{\omega_{1}}^{-1}\left(y\right)\subset J_{\omega_{1}\alpha}$.
Consequently, for each neighborhood $W$ of $\left(\omega,x\right)$,
we have $W\cap\left(\bigcup_{\rho\in I^{\N}}J^{\rho}\right)\neq\emptyset$.
The assertions on the $F\left(\tilde{f}\right)$ follow by taking
complements. The assertion in (\ref{enu:ratsemi-facts-3}) is proved
in \foreignlanguage{english}{(}\cite[Proposition 3.2 (b)]{MR1767945}\foreignlanguage{english}{). }
\end{proof}

\begin{rem*}Regarding Proposition \ref{prop:skewproduct-facts}  (\ref{enu:ratsemi-facts-3}), we remark that in general, we have $\Jpre\left(G\right)\subset\pi_{\Chat}\left(J\left(\tilde{f}\right)\right)\subset J\left(G\right)$. 
If $G$ is expanding,  then $\Jpre(G)=\pi _{\hat{\Bbb{C}}}(J(\tilde{f}))=
\pi _{\hat{\Bbb{C}}}(\cup _{\omega \in I^{\Bbb{N}}}J^{\omega })$ by Lemma \ref{lem:expanding-implies-skewproductjulia-is-closed}.
In particular, if $G$
is a finitely generated expanding rational semigroup, then $\Jpre\left(G\right)=J\left(G\right)$.
The inclusion $\cup _{\omega \in I^{\N}}J^{\omega }\subset J\left(\tilde{f}\right)$
can be strict, see \cite[Remark 2.8]{MR2736899} for an example. Theorem
\ref{thm:inducing-prejulia} provides examples of infinitely generated
rational semigroups for which the inclusion $\pi_{\Chat}\left(J\left(\tilde{f}\right)\right)\subset J\left(G\right)$
is strict. Namely, let $G=\langle f_{1},f_{2}\rangle $ be as in Lemma  \ref{PBD-has-SOSC}.
Then $G=\langle f_{1},f_{2}\rangle $ satisfies the assumptions of Theorem \ref{t:c:B1main} 
 (by renumbering $f_{1},f_{2}$ if necessary). For the infinitely generated rational semigroup
$H=\langle f_2 f_1^n : n\in \N \rangle $   we have
that $\emptyset \neq J(f_1) \subset J(G)\setminus \Jpre(H) = J(H)\setminus \Jpre(H)$ by Lemma \ref{PB-OSC-inducing}, 
Theorem \ref{thm:inducing-prejulia} (\ref{enu:inducingthm-prejuliaset-decomposition})
and Lemma  \ref{lem:juliasetof-H-equals-G}. Further, since $H$
is nicely expanding by Lemma \ref{lem:inducedsemigroup-is-nicelyexpanding},
we have  $\Jpre\left(H\right)=\pi_{\Chat}\left(J\left(\tilde{f}\right)\right)$
by Lemma \ref{lem:expanding-implies-skewproductjulia-is-closed},  where $\tilde{f}$ is the skew product 
associated to generator system $H_{0}=\{ f_{2}f_{1}^{r}: r\in \Bbb{N}\}$. 
Thus $\pi _{\hat{\Bbb{C}}}(J(\tilde{f}))\subsetneqq J(H)$. An example of a nicely expanding rational semigroup for which $\dim_H(\Jpre(G))<\dim_H(J(G))$ is given in Example \ref{ex:dimHdimB}. We remark that the pre-Julia set is a  continuous image of a Borel set. In particular,  $\Jpre\left(G\right)$ is a Suslin set and thus  universally measurable. For details on Suslin sets we refer to \cite[p.65--70]{MR0257325}. 
\end{rem*}

\section{Expanding Rational Semigroups}
\label{Expanding}
In this section, we  study the dynamics of expanding rational  semigroups.
For an expanding rational semigroup $G=\left\langle f_{i}:i\in I\right\rangle$ we have that  $J^{I}$ is a closed subset of  $I^{\N}\times\Chat$. 
\begin{lem}
\label{lem:expanding-implies-skewproductjulia-is-closed}Let $I$
be a topological space and let $\left(f_{i}\right)_{i\in I}\in C\!\left(I,\Rat\right)$.
If $G=\left\langle f_{i}:i\in I\right\rangle $ is expanding with
respect to $\left\{ f_{i}:i\in I\right\} $, then we have that $J\left(\tilde{f}\right)=J^{I}$.
In particular, we have that $\pi_{\Chat}\left(J\left(\tilde{f}\right)\right)=\Jpre\left(G\right)$. \end{lem}
\begin{proof}
Let $\left(\omega,z\right)\in J\left(\tilde{f}\right)$.
Since $G$ is expanding with respect to $\left\{ f_{i}:i\in I\right\} $
we have $\lim_{n}\big\Vert\left(\tilde{f}^{n}\right)'\left(\omega,z\right)\big\Vert=\infty$. This gives $z\in J_\omega$ and hence, $\left(\omega,z\right)\in J^{\omega}$. 
\end{proof}

Our next aim is to prove the following characterization for a rational
semigroup to be nicely expanding. 
\begin{prop}
\label{prop:nicelyexpanding-characterisation}Let
$I$ be a topological space and let $\left(f_{i}\right)_{i\in I}\in C\!\left(I,\Rat\right)$.
For the rational semigroup $G=\left\langle f_{i}:i\in I\right\rangle $,
the following statements are equivalent.
\begin{enumerate}
\item $G$ is nicely expanding. 
\item $G$ is hyperbolic, each element $h\in G\cap\Aut(\Chat)$ is loxodromic,
$\id\notin\overline{G\cap\Aut(\Chat)}$ and there exists a non-empty
compact $G$-forward invariant set $P_{0}\left(G\right)\subset F\left(G\right)$. \end{enumerate}
\end{prop}

In order to prove Proposition \ref{prop:nicelyexpanding-characterisation} we need the following Lemmas \ref{lem:moebiussemigroup-expanding-implies-loxodromic}--\ref{lem:hyperbolic-lox-pzero-implies-expanding}. We use the following classification of M\"{o}bius  transformations. Let $g\in \Aut(\Chat)\setminus \{ \id\}$. We say that $g$ is \emph{elliptic} if $g$ has two fixed points, for which the modulus of the multipliers is equal to one. If $g$ is neither loxodromic nor elliptic, then $g$ is \emph{parabolic}. 
\begin{lem}
\label{lem:moebiussemigroup-expanding-implies-loxodromic}Let $I$
be a topological space and let $\left(f_{i}\right)_{i\in I}\in C\!\left(I,\Rat\right)$.
If $G=\left\langle f_{i}:i\in I\right\rangle $ is expanding with
respect to $\left\{ f_{i}:i\in I\right\} $, then each element $g\in G\cap\Aut(\Chat)$
is loxodromic.\end{lem}
\begin{proof}
Let $g\in G\cap\Aut(\Chat)$ and  suppose by way of contradiction that $g$ is not loxodromic.  If $g$ is parabolic, then the parabolic
fixed point $z$ of $g$ satisfies $z\in J\left(g\right)$ and $\Vert(g^{n})'(z)\Vert=1$,
which contradicts that $G$ is expanding. Now suppose that $g$ is
elliptic or the identity map, and let $z_0 \in \Jpre(G)$.
For each $n\in\N$, set $z_{n}:=g^{-n}\left(z_0\right)$ and observe
that $z_{n}\in \Jpre(G)$. By conjugating $G$ by a M\"{o}bius  transformation, we may assume that $z_n \in \C$ for each $n\in \N$ and $g(z)=\e^{i \theta}z$, for some $\theta \in \R$. We see that  the modulus of
$\left(g^{n}\right)'\left(z_{n}\right)$ is equal to one. Letting
$n$ tend to infinity, contradicts that $G$ is expanding and completes
the proof. \end{proof}
\begin{lem}
\label{lem:expanding-identity-notin-closure}Let $I$ be a topological
space and let $\left(f_{i}\right)_{i\in I}\in C\!\left(I,\Rat\right)$.
If $G=\left\langle f_{i}:i\in I\right\rangle $ is expanding with
respect to $\left\{ f_{i}:i\in I\right\} $, then $\id\notin\overline{G\cap\Aut(\Chat)}$,  where the closure is taken in $\Aut(\Chat)$.\end{lem}
\begin{proof}
Since $G$ is expanding with respect to $\left\{ f_{i}:i\in I\right\} $,
there exist $C>0$ and $\lambda>1$ such that for all $n\in\N$ we
have $\inf_{(\omega,z)\in J\left(\tilde{f}\right)}\Vert\left(\tilde{f}^{n}\right)'\left(\omega, z\right)\Vert\ge C\lambda^{n}$.
Suppose by way of contradiction that there exists a sequence $\left(g_{n}\right)\in\big(G\cap\Aut(\Chat)\big)^{\N}$
such that $\lim_n \dist(g_n,\id)=0$. For each $n\in\N$, let $g_{n}$
be given by a product of $a_{n}$ generators in $\left\{ f_{i}:i\in I\right\} $,
for a sequence $\left(a_{n}\right)\in\N^{\N}$. We may  assume without
loss of generality that the sequence $\left(a_{n}\right)$
is unbounded. Otherwise, we choose for each $r\in\N$ an element $n_{r}\in\N$
such that $\sup_{z\in\Chat}d\left(g_{n_{r}}^{r}\left(z\right),z\right)<r^{-1}$.
Then we have $\lim_n \dist(g_{n_{r}}^{r},\id)=0$, as $r$ tends to infinity,
and $g_{n_{r}}^{r}$ is a product of $ra_{n_{r}}$ generators. By passing to a subsequence, we may assume that $\lim_{n}a_{n}=\infty$.
Choose an arbitrary $\left(\omega,z_{0}\right)\in J^I$
and write $g_{n}=f_{\alpha_{a_{n}}}\circ\dots\circ f_{\alpha_{1}}$, 
for some $\alpha\in I^{a_{n}}$. Clearly, $g_{n}^{-1}\left(z_{0}\right)\in J_{\alpha\omega}$.
 Consequently, as $n$ tends to infinity,  
\[
\sup_{z\in\Chat}\Vert g_{n}'\left(z\right)\Vert\ge\Vert g_{n}'\left(g_{n}^{-1}\left(z_{0}\right)\right)\Vert\ge\inf_{\left(\tau,y\right)\in J\left(\tilde{f}\right)}\Vert\left(\tilde{f}^{a_{n}}\right)'\left(\tau,y\right)\Vert\ge C\lambda^{a_{n}}\rightarrow\infty,
\]
which is impossible since $g_n$ tends uniformly to the identity  on $\Chat$. This contradiction finishes the proof. \end{proof}
\begin{lem}
\label{lem:atmosttwoelements-expandingness}Let $I$ be a topological
space and let $\left(f_{i}\right)_{i\in I}\in C\!\left(I,\Rat\right)$.
Suppose that $G=\left\langle f_{i}:i\in I\right\rangle $ is a  rational
semigroup such that $\card\left(J\left(G\right)\right)\le2$ and such
that each element in $G\cap\Aut(\Chat)$ is loxodromic. Further, assume 
that there exists a non-empty compact, $G$-forward invariant  subset $P_{0}\left(G\right)\subset F\left(G\right)$. Then  we have the following.
\begin{enumerate}
\item $\card\left(J\left(G\right)\right)=1$.
\item $G$ is expanding with respect to $\left\{ f_{i}:i\in I\right\} $
if and only if $\id\notin\overline{G\cap\Aut(\Chat)}$. 
\item If $I$ is finite or if $G$ satisfies the open set condition with
respect to $\left\{ f_{i}:i\in I\right\} $, then $G$ is expanding
with respect to $\left\{ f_{i}:i\in I\right\} $. 
\end{enumerate}
\end{lem}
\begin{proof}
It is clear that $G\subset \Aut(\Chat)$ because $\card(J(G))\le 2$. Let us start with the proof of (1).  Let $g\in G$. It follows from $g^{-1}\left(J\left(G\right)\right)\subset J\left(G\right)$  \cite[Theorem 2.1]{MR1397693}
and $\card(J(G))\le 2$, that  $g\left(J\left(G\right)\right)=J\left(G\right)$.
Since $g$ is loxodromic, we have  $g\left(x\right)=x$,
for each $x\in J\left(G\right)$. By way of contradiction, suppose
that  $J(G)$ consists of two points, say $J\left(G\right)=\left\{ a,b\right\} $. We may assume that $\Vert g'\left(a\right)\Vert>1$
and $\Vert g'\left(b\right)\Vert<1$. Now, for
each $z\in P_{0}\left(G\right)$, we have  $\lim_{n}g^{n}\left(z\right)=b$.
Since $g\left(P_{0}\left(G\right)\right)\subset P_{0}\left(G\right)$,
we conclude that $b\in P_{0}\left(G\right)\subset F\left(G\right)$,
which is a contradiction. Hence, $\card\left(J\left(G\right)\right)=1$.
For simplicity, we may assume that $J\left(G\right)=\left\{ 0\right\} $
in the following.

Next, we turn to the proof of (2). By Lemma \ref{lem:expanding-identity-notin-closure},
it remains to show that,  if $G$ is not expanding with respect to $\left\{ f_{i}:i\in I\right\} $,
then $\id\in\overline{G\cap\Aut(\Chat)}$.  Let $b_{i}$ denote the attracting fixed point of $f_{i}$. 
Since $f_{i}(P_{0}(G))\subset P_{0}(G)\subset F(G)$,
we have $b_{i}\in P_{0}(G)\subset F(G)$. 
Since $f_{i}(0)=0$, it follows that $\| f_{i}'(0)\| >1$
for each $i\in I$.   Let
$h_{i}\in\Aut(\Chat)$ be given by $h_{i}\left(z\right):=z/\left(b_{i}^{-1}z-1\right)$,
and observe that $h_{i}^{-1}f_{i}h_{i}\left(z\right)=c_i  z$, where $c_i$ denotes the multiplier of $f_i$ at $0$.
If $G$ is not expanding with respect to $\left\{ f_{i}:i\in I\right\} $,
then there exists a sequence $\left(i_{n}\right)\in I^{\N}$
tending to infinity such that $\lim_n \Vert f_{i_{n}}'\left(0\right)\Vert =1$.
After  passing to a subsequence, we may assume that  $\lim_n \dist(h_{i_n},h)=0$, 
for some $h\in\Aut(\Chat)$, which gives that $\lim_n \dist(f_{i_{n}},\id)=0$,
as $n$ tends to infinity. The proof of (2) is complete.

To prove (3), recall that by the proof of (2), we have that  $G$ is expanding if   $\id \notin \overline{ \bigcup_{i\in I} f_i}$. Clearly, if $I$ is finite or if   $\left\{ f_{i}:i\in I\right\} $ satisfies the open set condition, then $\id \notin \overline{ \bigcup_{i\in I} f_i}$. The proof is complete. \end{proof}
\begin{lem}
\label{lem:hyperbolic-lox-pzero-implies-expanding}Let $I$ be a finite
set endowed with the discrete topology. Let $G=\left\langle f_{i}:i\in I\right\rangle $
denote a hyperbolic rational semigroup such that each element in $G\cap\Aut\big(\Chat\big)$
is loxodromic and there exists a non-empty compact $G$-forward invariant set $P_{0}\left(G\right)\subset F\left(G\right)$.  Then $G$ is nicely expanding (see Definition \ref{nicelyexpanding}). \end{lem}
\begin{proof}
If $\card\left(J\left(G\right)\right)\ge3$, then we can follow the
proof of \cite[Theorem 2.6]{MR1625124} by replacing $P\left(G\right)$
by $P_{0}\left(G\right)$. The remaining case $\card\left(J\left(G\right)\right)\le2$
follows from Lemma \ref{lem:atmosttwoelements-expandingness} (3). 
\end{proof}
We now give the proof of Proposition  \ref{prop:nicelyexpanding-characterisation}.
\begin{proof}[Proof of Proposition  \ref{prop:nicelyexpanding-characterisation}]
The proof  that $(1)$ implies $(2)$ follows from Lemma \ref{lem:moebiussemigroup-expanding-implies-loxodromic} and Lemma \ref{lem:expanding-identity-notin-closure}. \\
We now turn our attention to the proof of the converse implication.
Our aim is to show that $G$ is expanding with respect to $\{f_i:i\in I\}$. 
By Lemma \ref{lem:atmosttwoelements-expandingness} (2), we are left
to consider the case $\card\left(J\left(G\right)\right)\ge3$. Since
$G$ is hyperbolic, we may assume that $P\left(G\right)\subset P_{0}\left(G\right)$.
We denote by $V_{1},\dots,V_{r}$, $r\in\N$, the finitely many connected
components of $F\left(G\right)$ which have non-empty intersection
with the non-empty compact set $P_{0}\left(G\right)$. Since $\card\left(J\left(G\right)\right)\ge3$,
we have that each $V_i$ is a hyperbolic Riemann surface. We denote by $d_h$  the Poincar\'{e} metric on $V_{i}$ and we set $U_{i}:=\left\{ z\in V_{i}:d_{h}\left(z,P_{0}\left(G\right)\cap V_{i}\right)<1\right\} $. 
Our main task is to verify the following claim. For $g\in G$ with $g\left(V_{i}\right)\subset V_{i}$, for some $i\in\left\{ 1,\dots,r\right\} $, we denote by  $\Vert g'\left(z\right)\Vert_{h}$ the norm of the derivative of $g$ at $z$ with respect to the Poincar\'{e} metric on $V_i$. 

\begin{clm} \label{expandingness-claim} For each $i\in\left\{ 1,\dots,r\right\} $
there exists $0<c_{i}<1$ such that, for all $g\in G$ satisfying
$g\left(V_{i}\right)\subset V_{i}$, we have $\sup_{z\in U_{i}}\Vert g'\left(z\right)\Vert_{h}\le c_{i}$.
\end{clm} \vspace{-5mm}
\begin{proof}[Proof of Claim \ref{expandingness-claim}]
 Suppose for a contradiction that the claim
is false. Since $V_{i}$ is hyperbolic and $g:V_{i}\rightarrow V_{i}$
is holomorphic, it follows by  Pick's Theorem (\cite[Theorem 2.11]{MR2193309})
that $\Vert g'\left(z\right)\Vert_{h}\le1$, for each $z\in V_{i}$.
Hence, by our assumption, there exist $i\in\left\{ 1,\dots,r\right\} $
and sequences $\left(g_{n}\right)\in G^{\N}$ and $\left(z_{n}\right)\in U_{i}^{\N}$
such that $g_{n}\left(V_{i}\right)\subset V_{i}$, for each $n\in\N$,
and $\lim_{n}\Vert g_{n}'\left(z_{n}\right)\Vert_{h}=1$. We may assume
that $\lim_{n}z_{n}=z_{\infty}\in\overline{U_{i}}$ by passing to
a subsequence. Since each family of holomorphic maps between hyperbolic
surfaces is normal (\cite[Corollary 3.3]{MR2193309}), we may assume
that there exists a holomorphic map $g_{\infty }:V_{i}\rightarrow 
\Chat$ such that  $g_{n}\rightrightarrows g_{\infty}$ on $V_{i}$, where $\rightrightarrows$
denotes uniform convergence on compact subsets of $V_{i}$, and $g_{\infty}:V_{i}\rightarrow V_{i}$
is holomorphic. We show that for each $g_{n}$ there exists a fixed
point $w_{n}\in P_{0}\left(G\right)\cap V_{i}$. This is clear in
the case that the degree of $g_{n}$ is equal to one by our assumption
that each element in $G\cap\mathrm{Aut}(\Chat)$ is loxodromic.
We consider the case that the degree of $g_{n}$ is at least two.
Since $g_{n}$ is hyperbolic, for each $x\in V_{i}\cap P_{0}\left(G\right)$,
the $g_{n}$-orbit of $x$ converges to some attracting $p$-periodic
point $w_{n}\in P_{0}\left(G\right)\cap V_{i}$  of $g_{n}$. Since
$g_{n}\left(V_{i}\right)\subset V_{i}$, we conclude that $p=1$.
We may assume that $\lim_{n}w_{n}=w_{\infty}\in P_{0}\left(G\right)\cap V_{i}$.
It then follows that $g_{\infty}\left(w_{\infty}\right)=w_{\infty}$
and $\Vert g_{\infty}'\left(z_{\infty}\right)\Vert_{h}=1$. From the
Classification Theorem (\cite[Theorem 5.2]{MR2193309}) for holomorphic
maps between hyperbolic surfaces, we have four possibilities for $g_{\infty}:V_{i}\rightarrow V_{i}$,
namely, attracting, escape, finite order and irrational rotation.
By Pick's Theorem and the fact that $\Vert g_{\infty}'\left(z_{\infty}\right)\Vert_{h}=1$
it follows that $g_{\infty}$ is a local isometry, which implies that
$\Vert g_{\infty}'\left(w_{\infty}\right)\Vert_{h}=1$. Thus, $g_{\infty}$
is not attracting. Escape is impossible since we have a fixed point.
We conclude that we have finite order or irrational rotation, hence,
in every case we have that there exists a sequence $\left(m_{j}\right)_{j\in\N}\in\N^{\N}$
such that $g_{\infty}^{m_{j}}\rightrightarrows\id_{V_{i}}$ on $V_{i}$.
Combining with $g_{n}\rightrightarrows g_{\infty}$, we conclude that
there exists $\left(h_{n}\right) \in G^{\N}$ such that $h_{n}\rightrightarrows\id_{V_{i}}$
on $V_{i}$. Let $A:=\big\{ z\in\Chat:\left(h_{n}\right)\mbox{ is normal in  a neighborhood of }z\big\} $
and let $A_{0}$ denote the connected component of $A$ containing
$V_{i}$. By Vitali's theorem \cite[ Theorem 3.3.3]{MR1128089} we conclude that $h_{n}\rightrightarrows\id_{A_{0}}$
in $A_{0}$. By our assumption that $\id\notin\overline{G\cap\mathrm{Aut}(\Chat)}$,
it follows that $A_{0}\neq\Chat$. 

We have  $\partial A_{0}\subset J\left(G\right)\subset\Chat\setminus P_{0}\left(G\right)$.
Let $x_{0}\in\partial A_{0}$ and choose $s>0$ such that $B\left(x_{0},s\right)\subset\Chat\setminus\bigcup_{i=1}^r U_{i}$.
There exists $y_{0}\in A_{0}$ and $s_{0}>0$ such that $B\left(y_{0},s_{0}\right)$
is a relatively compact subset of $B\left(x_{0},s\right)\cap A_{0}$.
Since $h_{n}\rightrightarrows\id_{A_{0}}$, we conclude that there
exists  $n_{0}\in\N$ such that for $n\ge n_{0}$ we have $h_{n}\left(B\left(y_{0},s_{0}\right)\right)\subset B\left(x_{0},s\right)$.
Now choose inverse branches $\gamma_{n}:B\left(x_{0},s\right)\rightarrow\Chat$
of $h_{n}$, that is, $h_{n}\circ\gamma_{n}=\id_{B\left(x_{0},s\right)}$
such that $\gamma_{n}\left(h_{n}\left(y_{0}\right)\right)=y_{0}$.
Since $\bigcup_{i=1}^r U_{i}$ is $G$-forward invariant, 
we have that $\gamma_{n}\left(B\left(x_{0},s\right)\right)\cap\bigcup_{i=1}^r U_{i}=\emptyset$,
which implies that $\left(\gamma_{n}\right)_{n\in\N}$ is normal in
$B\left(x_{0},s\right)$. From this and the equivalence 
between normality and equicontinuity \cite[ Theorem 3.3.2]{MR1128089}, we conclude that there exist
$\epsilon_{1},\epsilon_{2}>0$ such that for all $n\in\N$ we have
$\gamma_{n}\left(B\left(h_{n}\left(y_{0}\right),\epsilon_{1}\right)\right)\subset B\left(y_{0},\epsilon_{2}\right)\subset B\left(y_{0},s_{0}\right)$.
Now suppose that $\gamma_{n_{j}}\rightrightarrows\gamma_{\infty}$
on $B\left(x_{0},s\right)$, for some sequence $\left(n_{j}\right)$
tending to infinity and $\gamma_{\infty}:B(x_0,s)\rightarrow \Chat$ holomorphic. By \cite[p154]{Ah79} we obtain that $\gamma_{\infty}=\id$
on $B\left(y_{0},\epsilon_{3}\right),$ for some $\epsilon_{3}>0$
such that $B\left(y_{0},\epsilon_{3}\right)\subset B\left(h_{n}\left(y_{0}\right),\epsilon_{1}\right)$
for sufficiently large $n$. 
Hence, $\gamma_{n}\rightrightarrows\id_{B\left(x_{0},s\right)}$.
Combining this with $h_{n}\circ\gamma_{n}=\id_{B\left(x_{0},s\right)}$, 
we deduce that there is  $s_{1}<s_{0}$ such that  $h_{n}\left(B\left(x_{0},s_{1}\right)\right)\subset B\left(x_{0},s_{0}\right)$, for  sufficiently
large $n$. 
Hence, $\left(h_{n}\right)$ is normal in $B\left(x_{0},s_{1}\right)$ contradicting  the definition of $A_{0}$. The  claim
follows. 
\end{proof}
\vspace{-7pt}
We now continue the proof of  Proposition  \ref{prop:nicelyexpanding-characterisation}.  With $r\in \N$ and $U_1,\dots ,U_r$ from above, we set $W:=\bigcup_{i=1}^r U_{i}$. Next, we  verify that  there exists a compact set $K_{1}\subset W $
such that $f_{\omega}\left(W\right)\subset K_{1}$, for each $\omega\in I^{r}$.
To prove this, note that for each $i\in\left\{ 1,\dots,r\right\} $
and $\omega\in I^{r}$, there exist $j,k,q\in\left\{ 1,\dots,r\right\} $
with $k\le q$ such that $f_{\omega_{k-1},\dots,\omega_{1}}\left(U_{i}\right)\subset U_{j}$
and $f_{\omega_{q},\dots,\omega_{k}}\left(U_{j}\right)\subset U_{j}$,
where we set $f_{\emptyset}:=\id_{\Chat}$. Now it follows from Claim \ref{expandingness-claim} that, for each $j\in\left\{ 1,\dots,r\right\} $, there exists $0<c_{j}<1$ such that $f_{\omega_{q},\dots,\omega_{k}}\left(U_{j}\right)\subset\left\{ z\in V_{j}:d_{h}\left(z,P_{0}\left(G\right)\right)\le c_{j}\right\} $,
which is a compact subset of $U_{j}$. Hence, by Pick's Theorem
and using that $P_{0}\left(G\right)$ is $G$-forward invariant, we obtain
that, for each  $i\in\left\{ 1,\dots,r\right\}$ and $\omega \in I^r$, 
\begin{equation}
f_{\omega}\left(U_{i}\right)\subset \bigcup_{j=1}^r \big\{ z\in V_{j}:d_{h}\left(z,P_{0}\left(G\right)\right)\le\max\left\{ c_1,\dots,c_r\right\} \big\} :=K_{1}. \label{existenceK1}
\end{equation}
We have thus shown that  there exists a compact set  $K_{1}\subset W $
such that $f_{\omega}\left(W\right)\subset K_{1}$, for each $\omega\in I^{r}$.

\vspace{-5pt}
The next step is prove the existence of a compact set $K_{2}\subset\Chat\setminus\overline{W}$,
such that $f_{\omega}^{-1}\big(\Chat\setminus\overline{W}\big)\subset K_{2}$
for each $\omega\in I^{r}$. To prove this,  we verify that $a:=\inf_{\omega\in I^{r}}d\big(f_{\omega}^{-1}\big(\Chat\setminus\overline{W}\big),W\big)>0$,  where $d(A,B):=\inf_{a\in A, b\in B}d(a,b)$.
The claim then  follows by setting $K_{2}:=\left\{ z:d\left(z,W\right)\ge a\right\} $.
Assume for a contradiction that $a=0$. Then  there exist $\left(x_{n}\right)\in\big(\Chat\setminus\overline{W}\big)^{\N}$,
$(\omega^n)\in\left(I^{r}\right)^{\N}$
and $\left(y_{n}\right)_{n\in\N}$ with $y_{n}\in f_{\omega^{n}}^{-1}\left(x_{n}\right)$, 
for each $n\in\N$, such that $\lim_{n}d\left(y_{n},W\right)=0$.
After passing to a subsequence, we may assume  that there exists 
$y_{0}\in \overline{W}$ such that $\lim_{n}y_{n}=y_{0}$. By (\ref{existenceK1}), we conclude that $f_{\omega^{n}}\left(y_{0}\right)\in K_{1}$ for each $n\in \N$. Since $y_{0}\in F\left(G\right)$, we may assume
that $f_{\omega^{n}}\rightrightarrows f$,  as $n$ tending to infinity,  for some holomorphic
map $f$ in a neighborhood of $y_{0}$. Hence, $f\left(y_{0}\right)\in K_{1}\subset W$. On the other hand,
we have $f_{\omega^n}\left(y_{n}\right)=x_{n}\in\Chat\setminus W$,
for each $n\in\N$, which implies that $f\left(y_{0}\right)\in\Chat\setminus W$.
This contradiction shows that $a>0$. 

\vspace{-5pt}
We now turn to the final step of this proof. We denote by $A_{1},\dots,A_{\ell}$,
$\ell\in\N$, the connected components of $\Chat\setminus\overline{W}$ for which $A_{i}\cap J\left(G\right)\neq\emptyset.$
Clearly, we have $J\left(G\right)\subset\bigcup_{i=1}^{\ell}A_{i}$. 
Let $z\in\Jpre\left(G\right),$ $z\in J_{\tau}$, $\tau\in I^{\N}$,
and set $\omega:=\left(\tau_{1},\dots,\tau_{r}\right)$. There exist
$j_{1},j_{2}\in\N$ such that $z\in A_{j_{1}}$ and $f_{\omega}\left(z\right)\in A_{j_{2}}$. For a holomorphic map $h: S_1 \rightarrow S_2$ between hyperbolic  Riemann surfaces $S_1$ and $S_2$ and $w\in S_1$, we use   $\Vert h'(w) \Vert_{S_1,S_2}$ to denote the operator norm of the derivative $Dh(w):T_w S_1 \rightarrow T_{h(w)} S_2$ with respect to the norms  $\Vert \cdot \Vert_{S_1}$ and $\Vert \cdot \Vert_{ S_2}$ on the  tangent spaces $T S_1$ and $T S_2$,  given by the Poincar\'{e}  metrics. 
Now we consider $\Vert f_{\omega}'\left(z\right)\Vert_{A_{j_{1}},A_{j_{2}}}$.   Let $B_{j_{2}}$ denote the connected
component of $f_{\omega}^{-1}\left(A_{j_{2}}\right)$ which contains
$z$. By the previous step,  we have  $B_{j_{2}}\subset A_{j_{1}}\cap K_{2}$.
For each $j\in\left\{ 1,\dots,\ell\right\} $ let $D_{j}$ be an open connected
subset of $A_{j}$ such that $\overline{D_{j}}$ is compact in $A_{j}$
and $A_{j}\cap K_{2}\subset D_{j}$. For domains $\Omega_{1}\subset\Omega_{2}\subset\Chat$,
let $\iota_{\Omega_{1},\Omega_{2}}:\Omega_{1}\rightarrow\Omega_{2}$
denote the inclusion map. Note that, by Pick's Theorem, we have
that,  for each $j_{1},j_{2}\in\left\{ 1,\dots,r\right\} $, there exists
a constant $c_{j_{1},j_{2}}:=\sup_{z\in A_{j_{1}}\cap K_{2}}\|\iota_{D_{j_{1}},A_{j_{1}}}'\left(z\right)\|_{D_{j_{1}},A_{j_{1}}}<1$.
Consequently, for each $z\in B_{j_{2}}$, we have that 
\begin{align*}
\Vert\iota_{B_{j_{2}},A_{j_{1}}}'\left(z\right)\Vert_{B_{j_{2}},A_{j_{1}}} & =\Vert(\iota_{D_{j_{1}},A_{j_{1}}}\circ\iota_{B_{j_{2}},D_{j_{1}},})'\left(z\right)\Vert_{B_{j_{2}},A_{j_{1}}}\\
 & \le\Vert\iota_{D_{j_{1}},A_{j_{1}}}'\left(z\right)\Vert_{D_{j_{1}},A_{j_{1}}}\cdot\Vert\iota_{B_{j_{2}},D_{j_{1}}}'\left(z\right)\Vert_{B_{j_{2}},D_{j_{1}}}\le c_{j_{1},j_{2}}\cdot1.
\end{align*}
Finally, since $f_{\omega}:B_{j_{2}}\rightarrow A_{j_{2}}$ is a covering
map, it follows by Pick's Theorem that $f_{\omega}$ is locally
a Poincar\'{e} isometry, which implies that for $ v\in T_{z}B_{j_{2}}$ with $v\neq 0$ 
that $\Vert Df_{\omega}(z)(v)\Vert_{A_{j_{2}}}/\Vert v\Vert_{B_{j_{2}}}=1$.
Hence, we obtain that for each $ v\in T_{z}A_{j_{1}}$ with $v\neq 0$, 
\[
\Vert f_{\omega}'\left(z\right)\Vert_{A_{j_{1}},A_{j_{2}}}=\frac{\Vert Df_{\omega}(z)(v)\Vert_{A_{j_{2}}}}{\Vert v\Vert_{A_{j_{1}}}}=\frac{\Vert Df_{\omega}(z)(v)\Vert_{A_{j_{2}}}}{\Vert v\Vert_{B_{j_{2}}}}\cdot\frac{\Vert v\Vert_{B_{j_{2}}}}{\Vert v\Vert_{A_{j_{1}}}}= \|\iota_{B_{j_{2}},A_{j_{1}}}'\left(z\right)\|_{B_{j_{2}},A_{j_{1}}}^{-1}  \ge c_{j_{1},j_{2}}^{-1}>1.
\]
Using the same argument, one verifies that for each $z\in\Jpre\left(G\right),$
$z\in J_{\tau}$, $\tau\in I^{\N}$, $\ell\in\N$ with $\ell<r$, and $\rho:=\left(\tau_{1},\dots,\tau_{\ell}\right)$, if $z\in A_{j_{1}}, f_{\rho }(z)\in A_{j_{2}}$ then we have $\Vert f_{\rho}'\left(z\right)\Vert_{A_{j_{1}},A_{j_{2}}}\ge1$. 

To finish the proof, we remark that the  Poincar\'{e} metric and the
spherical metric are equivalent on the compact subset $J(G)\cap A_{i}$ of $A_{i}$,
which proves that there exist a constant $\lambda >1$ and a constant $C>0$ such that,  for
each $\left(\tau,z\right)\in J^{\tau}$, $\tau\in I^{\N}$ and  $n\in\N$,  we have $\Vert\left(\tilde{f}^{n}\right)'\left(\tau,z\right)\Vert\ge C\lambda^{n}$.
Finally, continuity of  $(\tilde{f}^n)'$ completes the proof of  Proposition  \ref{prop:nicelyexpanding-characterisation}.  
\end{proof}


In the following  lemma we give a sufficient condition for  an infinitely generated
rational semigroup to be nicely expanding in terms of the open set condition. 
\begin{lem}
\label{lem:expandingness-osc}Let $I$ be a countable set endowed
with the discrete topology. Let $G=\left\langle f_{i}:i\in I\right\rangle $
denote a hyperbolic rational semigroup. Suppose that each element in $G\cap\Aut\big(\Chat\big)$
is loxodromic,  there exists a non-empty compact, $G$-forward invariant subset $P_{0}\left(G\right)\subset F\left(G\right)$
 and that $\left\{ f_{i}:i\in I\right\} $ satisfies
the open set condition. Then $G$ is nicely expanding. \end{lem}
\begin{proof}
By Lemma \ref{lem:hyperbolic-lox-pzero-implies-expanding} we may assume that $I=\N$. By Lemma \ref{lem:atmosttwoelements-expandingness} (3), we may assume
that $\card\left(J\left(G\right)\right)\ge3$. Since $G$ is hyperbolic,
we may assume that $P\left(G\right)\subset P_{0}\left(G\right)$.
Let us first show that we can assume without loss of generality that
$\left\{ f_{i}:i\in I\right\} $ satisfies the open set condition
with respect to an open set $V$ such that $\overline{V}\subset\mathbb{C}$
compact. To prove this, first note that there exists a neighborhood
$W$ of $P_{0}\left(G\right)$ in $F\left(G\right)$ which is $G$-forward
invariant. By conjugating $G$ with an element of $\Aut\big(\Chat\big)$
we may also assume that $W$ contains infinity. Now, for any $g\in G$, since $g\left(\overline{W}\right)\subset\overline{W}$
we have that $g^{-1}\big(\Chat\setminus\overline{W}\big)\subset\Chat\setminus\overline{W}$.
Finally, if $G$ satisfies the open set condition with respect to
some open set $V'\subset\Chat$, then  $G$ satisfies
the open set condition with respect to $V:=V'\cap\big(\Chat\setminus\overline{W}\big)$
and $\overline{V}\subset\mathbb{C}$. 

Our next  aim is to verify that
\begin{equation}
\lim_{n\rightarrow\infty}\inf_{z\in f_{n}^{-1}\left(J\left(G\right)\right)}\Vert f_{n}'\left(z\right)\Vert=\infty.\label{eq:expanding-lemma-derivativesthendtoinfinity}
\end{equation}
Since $\card\left(J\left(G\right)\right)\ge3$, the Julia set  $J\left(G\right)$
is the smallest non-empty compact $G$-backward invariant subset of $\Chat$  (\cite{MR1397693, MR1767945}).
Hence, $J\left(G\right)\subset\overline{V}$. Since $P_{0}\left(G\right)\subset F\left(G\right)$,
 there exists $r_{1}>0$ such that  $\overline{B\left(x,r_{1}\right)}\subset\Chat\setminus P_{0}\left(G\right)$, for all $x\in J\left(G\right)$.
Using that $J\left(G\right)$ is compact we deduce the existence of $r_{2}>0$ with the property that,  for each $x\in J\left(G\right)$, 
there exists  $y_{x}\in V$ such that $B\left(y_{x},r_{2}\right)\subset B\left(x,r_{1}\right)\cap V$.
By the open set condition and Koebe's distortion theorem, it follows that,  for each $\epsilon>0$, 
there exists  $n_{0}$ such that $\Vert\gamma'\left(x\right)\Vert\le\epsilon$, for all $n\ge n_{0}$,  $x\in J\left(G\right)$
and for all inverse branches $\gamma$ of $f_{n}$ on $B\left(x,r_{1}\right)$.  The proof
of (\ref{eq:expanding-lemma-derivativesthendtoinfinity}) is complete.

We will now verify that 
$\id\notin\overline{G\cap\mathrm{Aut}\big(\Chat\big)}$, 
from which the lemma follows by  Proposition  \ref{prop:nicelyexpanding-characterisation}.  For the proof,  it suffices to show that $H:=\big\{ h^{-1}:h\in\mathrm{Aut}\big(\Chat\big)\cap G\big\} $ is closed in $\Rat$. Let $h\in\overline{H}$ and $\left(h_{n}\right)\in H^{\N}$
with $h_{n}\rightrightarrows h$ on $\Chat$ be given, where 
\[
h_{n}=f_{\omega_{1}^n}^{-1}\circ\dots\circ f_{\omega_{\ell_{n}}^n}^{-1},\quad  \omega^{n}=(\omega_{1}^{n},\dots,\omega_{\ell_{n}}^{n}) \in I^{\ell_{n}}, \ell_n \in \N.
\]
In order to show that $h\in H$, we will verify that $\sup \ell_n<\infty$ and that there exists a finite set $F\subset I$,  such that  $\omega^{n}\in F^{\ell_{n}}$, for all $n\in\N$. To prove this, we will show that each of the following assumptions (1) and (2)   gives a contradiction:
\begin{enumerate}
\item $\sup \ell_n=\infty$  and there exists a finite set $F\subset I$ such that  $\omega^{n}\in F^{\ell_{n}}$, for all $n\in \N$. 
\item There exists a sequence $\left(j_{n}\right)\in\N^{\N}$ with $j_{n}\le \ell_{n}$
such that $\lim_{n}\omega_{j_{n}}^{n}=\infty$.  
\end{enumerate}

Suppose for a contradiction that (1) holds.  Set $G_{F}:=\left\langle f_{i}:i\in F\right\rangle $.
We may assume that $\card\left(J\left(G_{F}\right)\right)\ge3$. Since
$G_{F}$ is a finitely generated hyperbolic rational semigroup, we have that $G_{F}$ is expanding with respect to $\left\{ f_{i}:i\in F\right\} $
by Lemma \ref{lem:hyperbolic-lox-pzero-implies-expanding}. Since $\sup l_n=\infty$, we have $h'=0$ on $J\left(G_{F}\right)$. Consequently, since $J\left(G_{F}\right)$ is perfect (\cite{MR1397693, MR1767945}),   the identity theorem gives that $h$ is a constant function, which contradicts the continuity of the degree function. 

To derive a contradiction from (2), we  assume
that $\lim_{n}\omega_{j_{n}}^{n}=\infty$. By (\ref{eq:expanding-lemma-derivativesthendtoinfinity})
we conclude that 
\begin{equation}
\lim_{n\rightarrow\infty}\sup_{z\in J\left(G\right)}\big\Vert\big(f_{\omega_{j_{n}}^{n}}^{-1}\big)'\left(z\right)\big\Vert=0.\label{eq:expandinglemma-inversederivativestendtozero}
\end{equation}
To deduce a contradiction, let $W$ denote a $G$-forward invariant relatively compact open neighborhood of $P_{0}\left(G\right)$ in $F\left(G\right)$. Then we have $g^{-1}\big(\Chat\setminus\overline{W}\big)\subset\Chat\setminus\overline{W}$,
for each $g\in G$, which implies that $H$ is normal in the  neighborhood $\Chat \setminus \overline W$ of $J\left(G\right)$.
After passing to a subsequence, combining (\ref{eq:expandinglemma-inversederivativestendtozero})
with the identity theorem gives  that $f_{\omega_{j_{n}}^{n}}^{-1}\rightrightarrows c_A$
on $A$, for each connected component $A$ of $\Chat\setminus\overline{W}$ and some $c_A\in J(G)$. Writing $h_{n}=r_{n}f_{\omega_{j_{n}}^{n}}^{-1}s_{n}$
with $r_{n},s_{n}\in H\cup\left\{ \id\right\} $,  for each $n\in\N$, we may assume
that  $r_{n}\rightrightarrows r$ and 
$s_{n}\rightrightarrows s$ 
in the  neighborhood $\Chat\setminus\overline{W}$ of $J\left(G\right)$. Consequently, $h$ is a
constant function, which gives the desired contradiction.
\end{proof}
Regarding the dynamics on the Fatou set of a nicely expanding rational
semigroup, we prove the following lemma, which will be
useful in Section \ref{sec:Inducing-Methods}.
\begin{lem}
\label{lem:fatou-forward-dynamics}Let $G=\langle f_{i}:i\in I\rangle$ be a nicely expanding rational semigroup with $G$-forward invariant set $P_{0}(G)$.
Suppose that $\card(J(G))>1$. Then,  for each $\omega \in I^{\Bbb{N}}$ and  $x\in F_{\omega }$, 
we have  $\lim_n{d(f_{\omega |_{n}}(x), P_{0}(G))}= 0$ and 
each limit function of $( f_{\omega |_{n}}) _{n\in \mathbb{N}}$ in a connected neighborhood of 
$x$ in $F_{\omega }$ is a constant function whose value is in $P_{0}(G)$.   

\end{lem}
\begin{proof}
By Lemmas \ref{lem:moebiussemigroup-expanding-implies-loxodromic} and \ref{lem:atmosttwoelements-expandingness} we have $\card(J(G))\geq 3$. Let  $V_1,\dots,V_r$, $r\in \N$,  denote the  connected components  of $F(G)$ which meet $P_{0}(G).$  Let $\omega \in I^{\N}$ and $x \in F_{\omega }.$ Then   the family $( f_{\omega |_{n}}) _{n}$ is normal in a neighborhood of $x$. 
Suppose for a contradiction that there exists a subsequence $( g_{j})$ of $( f_{\omega |_{n}}) $ which converges to a non-constant map $h$ in a neighborhood of $x$. Since $h$ is non-constant, it follows from Claim \ref{expandingness-claim} in the proof of  Proposition  \ref{prop:nicelyexpanding-characterisation} that  $g_{j}(x)\in \hat{\Bbb{C}}\setminus \cup _{s=1}^{r}V_{s}$, for each $j$.    By the method employed in the final step of the  proof of  Proposition  \ref{prop:nicelyexpanding-characterisation}, we can show that $\| g_{j}'(x)\| \rightarrow \infty $, as $j\rightarrow \infty$, which  is a contradiction. Thus, each limit function of $( f_{\omega |_{n}}) $ in a connected neighborhood of $x$ in $F_{\omega }$ is  constant.
Now suppose that a subsequence $( g_{j})$ of $( f_{\omega |_{n}}) $  converges to a constant $c$ in a neighborhood of $x$. We will show that   $c\in P_{0}(G)$. Otherwise,  there exists  $\delta>0$ such that $\overline{B(c,\delta)}\cap P_0(G) = \emptyset$ and,  for each large $j$, there exists  a well defined inverse branch $h_{j}:B(c,\delta)\rightarrow \Chat$ of $g_{j}$  such that  $h_{j}(g_{j}(x))=x$ and $g_j \circ h_j=\id$ on $B(c,\delta)$. Since there exists a $G$-forward invariant neighborhood $W$ of $P_{0}(G)$ with $W\cap B(c,\delta)=\emptyset $, we conclude that  $( h_{j}) $ is normal 
in a neighborhood $\Omega $ of $x$. Hence $(h_{j})$ is equicontinuous in 
$\Omega .$ Since $g_{j}(x)\rightarrow c$ as $j\rightarrow \infty $, there exist $\delta _{0}\in (0,\delta )$ and a relative compact subset $\Omega _{0}$ of 
$\Omega $  such that for each $j\in \Bbb{N}$, $h_{j}(B(c,\delta _{0}))\subset 
\Omega _{0}.$ Since $g_{j}\rightrightarrows c $ 
on $\Omega _{0}$ as $j\rightarrow \infty $, we have 
$g_{j}\circ h_{j}\rightarrow c$ on $B(c,\delta _{0})$ as 
$j\rightarrow \infty .$ However, this contradicts 
$g_{j}\circ h_{j}=id $ on $B(c,\delta )$ for each $j.$
 Therefore, $\lim_j{d(g_{j}(x),P_{0}(G))}=0$ and hence,  $\lim_n{d(f_{\omega |_{n}}(x), P_{0}(G))}=0$. 
\end{proof}

Finally, we prove some useful facts about the exceptional sets of
expanding rational semigroups.
\begin{lem}
\label{lem:moebiussemigroup-expanding-implies-goodexceptionalset}Let
$I$ be a topological space and let $\left(f_{i}\right)_{i\in I}\in C\!\left(I,\Rat\right)$.
Suppose that $G=\left\langle f_{i}:i\in I\right\rangle $
is expanding with respect to $\left\{ f_{i}:i\in I\right\} $ and
 $G\subset\Aut\big(\Chat\big)$. Let $G_{0}\subset G$ be a
subsemigroup such that $\card\left(J\left(G_{0}\right)\right)\ge3$.
Then  we have $E\left(G_{0}\right)\subset F\left(G\right)$. \end{lem}
\begin{proof}
Suppose for a contradiction that there exists $z_{0}\in E\left(G_{0}\right)\cap J\left(G\right)$.
Since $\card\left(J\left(G_{0}\right)\right)\ge3$, it follows from
the density  of the repelling fixed points in the Julia set 
and the perfectness of the Julia set 
(\cite[Theorem 3.1, Lemma 3.1]{MR1397693}, \cite[Lemma 2.3]{MR1767945}) 
 that there exist $z_{1}\in J\left(G_{0}\right)$
and $g_{1}\in G_{0}$, such that $z_{1}\neq z_{0}$, $g_{1}\left(z_{1}\right)=z_{1}$
and $\Vert g_{1}'\left(z_{1}\right)\Vert>1$. Furthermore, we have
 $\card\left(E\left(G_{0}\right)\right)\le2$ (\cite[Lemma 3.3]{MR1397693}, \cite[Lemma 2.3]{MR1767945}) . Combining with the fact
that $g^{-1}\left(E\left(G_{0}\right)\right)\subset E\left(G_{0}\right)$ 
for each $g\in G_{0}$, we conclude that $g_{1}^{2}\left(z_{0}\right)=z_{0}$.
Since $G$ is expanding, we have that $g_1$ is loxodromic by Lemma
\ref{lem:moebiussemigroup-expanding-implies-loxodromic}. Thus, $z_{0}$
is the attracting fixed point of $g_{1}^{2}$. Let $V$ be a neighborhood of $z_0$ and let $0<c<1$ such that $g_1^2(V)\subset V$ and $\Vert (g_1^2)'(z)\Vert <c$, for each $z\in V$.  Since the Julia set is perfect and by the density of the repelling fixed points in the Julia set (\cite{MR1397693, MR1767945}),  there exists a sequence $(a_n)$ with $a_n \in \Jpre(G)\setminus \{ z_0\}$  such that $\lim_n a_n =z_0$.  Then there exists a sequence $(n_k)\in \N^{\N}$ tending to infinity, such that $g_1^{-2n_k}(a_{n_k})\in V$. Hence, $\lim_k \Vert (g_1^{2n_k})'(g_1^{-2n_k}(a_{n_k}))\Vert  \le\lim_k c^{n_k}=0$.   Moreover, write  $g_{1}^{2}=f_\alpha$, 
for some $m\in\N$ and $\alpha\in I^{m}$, and denote by $\alpha^n:=(\alpha \dots \alpha)\in I^{mn}$ the $n$-fold concatenation of $\alpha$. Let $(\beta_k)\in I^\N$ with $(\beta_k, a_{n_k})\in J(\tilde{f})$.  Then $\big(\alpha^{n_k}\beta_k, (g_1^{-2n_k})(a_{n_k})\big)\in J(\tilde{f})$.  This contradicts that $G$ is expanding and finishes the proof.  
\end{proof}
\begin{lem}
\label{lem:moebiussemigroup-expanding-implies-J-not-twoelements}Let
$I$ be a topological space and let $\left(f_{i}\right)_{i\in I}\in C\!\left(I,\Rat\right)$.
Suppose  that $G=\left\langle f_{i}:i\in I\right\rangle $
is expanding with respect to $\left\{ f_{i}:i\in I\right\} $ and
 $1\le\card\left(J\left(G\right)\right)\le2$. Then we have $\card\left(J\left(G\right)\right)=1$. \end{lem}
\begin{proof}
Clearly, we have $G\subset\Aut\left(\Chat\right)$ and  each
element of $G$ is loxodromic by Lemma \ref{lem:moebiussemigroup-expanding-implies-loxodromic}. Now, suppose by way of contradiction that $J\left(G\right)=\left\{ a,b\right\} $.
Without loss of generality, we may assume that $a=0$ and $b=\infty$.
Since $J(G)$ is $G$-backward invariant, we have $g\left(a\right)=a$ and $g\left(b\right)=b$.
Thus, there exists a sequence $\left(c_{i}\right)\in\C^{I}$
such that $f_{i}\left(z\right)=c_{i}z$, for each $z\in\Chat$ and
$i\in I$. We may assume that there exists $i_0 \in I$ such that $\Vert f_{i_0}'(a) \Vert >1$. Since $G$ is expanding with respect to $\left\{ f_{i}:i\in I\right\} $,
there exists a constant $c_0>1$ such that $\Vert f_{i}'\left(a\right)\Vert =\left|c_{i}\right|\ge c_0>1$,  for all $i\in I$ and $z\in\C$. Hence, we have  $\Vert f_{i}'\left(b\right)\Vert  \le c_0^{-1}<1$, for all $i\in I$, which gives that $b\in F\left(G\right)$.
This contradiction proves the lemma. \end{proof}
\begin{lem}
\label{exceptionalset-good}
Let $I$ be a topological space and let $\left(f_{i}\right)_{i\in I}\in C\!\left(I,\Rat\right)$.
Suppose that $G=\left\langle f_{i}:i\in I\right\rangle $
is expanding with respect to $\left\{ f_{i}:i\in I\right\} $, $\card\left(J\left(G\right)\right)>1$
and  $G\subset\Aut\big(\Chat\big)$. Then we have the following.
\begin{enumerate}
\item $E\left(G\right)\subset F\left(G\right)$. 
\item If $E\left(G\right)\neq\emptyset$, then we have $g(x)=x$ and $\Vert g'\left(x\right)\Vert <1$, for all   $g\in G$ and $x\in E\left(G\right)$. 
\end{enumerate}
\end{lem}
\begin{proof}
By  Lemma \ref{lem:moebiussemigroup-expanding-implies-J-not-twoelements}
we have  $\card\left(J\left(G\right)\right)\ge3$. Hence, the assertion in (1) follows from Lemma \ref{lem:moebiussemigroup-expanding-implies-goodexceptionalset}.
Further, $\card\left(J\left(G\right)\right)\ge3$ implies that $\card\left(E\left(G\right)\right)\le2$ (\cite{MR1397693, MR1767945}). Hence,  $g\left(E\left(G\right)\right)=E\left(G\right)$,
for each $g\in G$. Since $G$ is expanding, we have that each element
$g\in G$ is loxodromic by Lemma \ref{lem:moebiussemigroup-expanding-implies-loxodromic}.
By the first assertion, we can now conclude that,  if $E\left(G\right)$
is non-empty, then $\card\left(E\left(G\right)\right)=1$. Moreover,
for each $g\in G$, the element in $E\left(G\right)$ is the attracting
fixed point of $g$, which proves (2).
\end{proof}
\section{Topological Pressure}
\label{Topological}

In this section we derive basic properties of two notions of topological
pressure associated with the dynamics of rational semigroups. We start
with a preparatory lemma. The property derived in this lemma is similar to the finitely primitive condition (\cite{MR2003772}) for  topological Markov chains with an infinite alphabet.
\begin{lem}
\label{lem:finitely_primitivity}Let $I$ be a topological space and
let $\left(f_{i}\right)_{i\in I}\in C\!\left(I,\Rat\right)$. Suppose
that either (1) $G=\left\langle f_{i}:i\in I\right\rangle $ is a
hyperbolic rational semigroup which contains an element of degree
at least two or (2) $G$ is nicely expanding. Then for each finite
family $\left(U_{i}\right)_{i\in\left\{ 1,\dots,s\right\} }$ of non-empty
open subsets of $J\left(G\right)$, there exists $\ell_0\in\N$ and a finite
set $I_{0}\subset I$ such that for each $z\in J(G)$, for each $\ell \in \Bbb{N}$ with $\ell \geq \ell _{0}$ and
for each $i\in\left\{ 1,\dots,s\right\} $ there is $\omega\in I_{0}^{\ell}$
such that $f_{\omega}^{-1}\left(z\right)\cap U_{i}\neq\emptyset$. \end{lem}
\begin{proof}
Let us first suppose that $G$ satisfies the assumptions in (1).  By the density of the
repelling fixed points in the Julia set (\cite[Theorem 3.1]{MR1397693}),  we have that
for each $i\in\left\{ 1,\dots,s\right\} $ there exists $g\in G$
such that $J\left(g\right)\cap U_{i}\neq\emptyset$. Hence, there
exists a finitely generated subsemigroup $G_{0}:=\left\langle f_{i}:i\in I_{0}\right\rangle $, 
including an element of degree at least  two, such that $J\left(G_{0}\right)\cap U_{i}\neq\emptyset$, for
all $i\in\left\{ 1,\dots,s\right\} $.
Since $G_{0}$ contains an element of degree at least two, we have
 $E\left(G_{0}\right)\subset P\left(G\right)$. Combining with our assumption that  $G$ is
hyperbolic, we obtain that   $E\left(G_{0}\right)\subset F\left(G\right)$. 
In particular, we have  $J\left(G\right)=\Chat\setminus F\left(G\right)\subset\Chat\setminus E\left(G_{0}\right)$. 

Our aim is to apply \cite[Theorem 4.3]{MR1767945} to the finitely
generated semigroup $G_{0}$. We have seen that $E\left(G_{0}\right)\subset F\left(G\right)\subset F\left(G_{0}\right)$.
We now verify that $J\left(G_{0}\right)\subset F\big(\big\{ h^{-1}:h\in\mathrm{Aut}(\Chat)\cap G_{0}\big\} \big)$.
To prove this, let $V_{1},\dots,V_{r}$, $r\in\N$, denote the finitely
many connected components of $F\left(G\right)$ which have non-empty
intersection with the compact set $P\left(G\right)$. Since $G$ contains
an element of degree at least two, we have that each $V_{i}$ is hyperbolic
and we denote by $d_h$ the Poincar\'{e} metric on $V_{i}$. Set
$W_{i}:=\left\{ z\in V_{i}:d_{h}\left(z,P\left(G\right)\cap V_i\right)<1\right\} $
and note that $A:=\bigcup_{i=1}^r W_{i}$ is $G$-forward invariant. Hence,  we have  $g^{-1}\big(\Chat\setminus\overline{A}\big)\subset\Chat\setminus\overline{A}$,
for each $g\in G$, which implies that $\Chat\setminus\overline{A}\subset F\big(\big\{ h^{-1}:h\in\mathrm{Aut}(\Chat)\cap G_{0}\big\} \big)$
by Montel's Theorem.  We have thus shown that  $J\left(G_{0}\right)\subset J\left(G\right)\subset\Chat\setminus\overline{A}\subset F\big(\big\{ h^{-1}:h\in\mathrm{Aut}(\Chat)\cap G_{0}\big\} \big)$. 

Set $K:=J\left(G\right)$ and observe that $K$ is a  compact subset of $\Chat\setminus E\left(G_{0}\right)$, which is $G_0$-backward invariant.   Let $\mu$ denote the Borel probability measure on $J\left(G_{0}\right)$
corresponding to the equidistribution on the generators $\{ f_{i}:i\in I_{0}\}$ of $G_{0}$, 
which exists by \cite[Theorem 4.3]{MR1767945} and whose topological support $\supp\mu$ is equal to $J(G_{0})$. By \cite[Theorem 4.3]{MR1767945} there exists $\ell_0 \in \N$ such that, for each $z\in J(G)$, for all $\ell \ge \ell_0$ and for each $i\in \{1,\dots, s \}$,  there exists $\omega\in I_{0}^{\ell}$ such that $f_{\omega}^{-1}\left(z\right)\cap U_{i}\neq\emptyset$.
The proof of the first case is complete.

Now suppose that $G$ is nicely expanding and  $G\subset\Aut(\Chat)$.
Since the theorem is obviously true in the case that $\card\left(J\left(G\right)\right)=1$,
we may assume that  $\card\left(J\left(G\right)\right)\ge3$ by Lemma
\ref{lem:moebiussemigroup-expanding-implies-J-not-twoelements}. Choose a finitely
generated subsemigroup $G_{0}$ of $G$ such that $\card\left(J\left(G_{0}\right)\right)\ge3$.
Hence, we have $J\left(G\right)\subset\Chat\setminus E\left(G_{0}\right)$
by Lemma \ref{lem:moebiussemigroup-expanding-implies-goodexceptionalset}.
By substituting $P\left(G\right)$ by $P_{0}\left(G\right)$ in the
proof of the first case, we verify that $J\left(G_{0}\right)\subset F\big(\big\{ h^{-1}:h\in\mathrm{Aut}(\Chat)\cap G_{0}\big\} \big)$.
 Set $\tilde{K}:=\pi_{\Chat}^{-1}\left(J\left(G\right)\right)$. Since
$\tilde{K}\subset\pi_{\Chat}^{-1}\left(\Chat\setminus E\left(G_{0}\right)\right)$,
it follows from \cite[Proof of Lemma 4.6]{MR1767945} that each unitary
eigenvector of the bounded linear  operator $\tilde{B}:{C}\left(\tilde{K}\right)\rightarrow{C}\left(\tilde{K}\right)$,
 given by $\tilde{B}\left(\tilde{\varphi}\right)\left(\tau,z\right):=\card\left(I_{0}\right)^{-1}\sum_{\omega\in I_{0}}\tilde{\varphi}\left(\omega\tau,f_{\omega}^{-1}\left(z\right)\right)$, for each $\tilde{\varphi}\in{C}\left(\tilde{K}\right)$
and $\left(\tau,z\right)\in\tilde{K}$, 
is a constant function on $\tilde{K}$. Finally, since $J\left(G_{0}\right)\subset F\big(\big\{ h^{-1}:h\in\mathrm{Aut}(\Chat)\cap G_{0}\big\} \big)$, we have that $\left(\tilde{B}^{n}\left(\tilde{\varphi}\right)\right)_{n\in\N}$
is a family of equicontinuous functions on $\tilde{K}$. Hence, we have that $\tilde{B}$ is an almost
periodic operator by
the Arzel\`{a}-Ascoli Theorem and  the result of \cite[Theorem 4.3]{MR1767945}
follows from the proof of \cite[Theorem 4.3]{MR1767945}.  The rest of the proof runs as in case (1). 
\end{proof}

\subsection{Bounded distortion property and topological pressure }

Throughout this subsection, let $I$ be a countable set
and let $\left(f_{i}\right)_{i\in I}\in\Rat^{I}$. For each $n\in\N$,  $x\in\Chat$ and $t\in \R$, we define
\[
Z_{n}\left(I,t,x\right):=\sum_{\omega\in I^{n}}\sum_{y\in f_{\omega}^{-1}\left(x\right)}\left\Vert f_{\omega}'\left(y\right)\right\Vert ^{-t},\,\,\,\, P\left(I,t,x\right):=\limsup_{n\rightarrow\infty}\frac{1}{n}\log Z_{n}\left(I,t,x\right).
\]
Here, the sum $\sum _{y\in f_{\omega }^{-1}(x)}$ counts the multiplicities, 
and we set $0^{-t}=\infty $ if $t\geq 0$, and we set 
$0^{-t}=0$ if $t<0.$
Using Lemma \ref{lem:finitely_primitivity} the proof of the following lemma is a straightforward application of Koebe's distortion theorem.
\begin{lem}
[Bounded distortion lemma]\label{lem:bounded_distortion_lemma}Let
$I$ be a  countable set and let $\left(f_{i}\right)_{i\in I}\in\Rat^{I}$.
Suppose that either (1) $G=\left\langle f_{i}:i\in I\right\rangle $
is a hyperbolic rational semigroup which contains an element of degree
at least two or (2) $G=\left\langle f_{i}:i\in I\right\rangle $ is
nicely expanding. Then, for each $t\in\R$, there exist 
$C>1$,  $\ell\in\N$  and a finite set $I_{0}\subset I$,  such that
for all $I_{1}\subset I$ with $I_{0}\subset I_{1}$, for all $n\in\N$
and for all $x,y\in J\left(G\right)$ we have 
\begin{equation}
Z_{n+\ell}\left(I_{1},t,y\right)\ge C^{-1}Z_{n}\left(I_{1},t,x\right).\label{eq:bounded-distortion-lowerbound}
\end{equation}
Furthermore, for each $x\in J\left(G\right)$ we have that $P\left(I_{1},t,x\right)<\infty$  if and only if $\sup_{y\in J\left(G\right)}Z_{1}\left(I_{1},t,y\right)<\infty$.
In particular, if $P\left(I,t,x_{0}\right)<\infty$  for some $x_{0}\in J\left(G\right)$,
then there exists  $C'>1$ such that,  for all $n\in \N$ and for all $x,y\in J\left(G\right)$, we have
\begin{equation}
(C')^{-1}Z_{n}\left(I_{1},t,x\right)\le Z_{n}\left(I_{1},t,y\right)\le C'Z_{n}\left(I_{1},t,x\right)<\infty.\label{eq:bounded-distortion-property}
\end{equation}
Moreover, if $P\left(I_{1},t,x_{0}\right)=\infty$  for some $x_{0}\in J\left(G\right)$,
then for all sufficiently large $n\in\N$ and for all $x\in J\left(G\right)$
we have $Z_{n}\left(I_{1},t,x\right)=\infty$.
\end{lem}
\vspace{-0.7cm}
\begin{proof}
We will verify the lemma under the assumptions given in (1). From this, one can deduce that the lemma holds under the assumptions in (2) by replacing $P(G)$ by $P_0(G)$.  Our first aim is to define the finite
set $I_{0}\subset I$. Since $G$ is hyperbolic, we have  $d\left(J\left(G\right),P\left(G\right)\right)>0$, which allows us to fix some $0<r<d\left(J\left(G\right),P\left(G\right)\right)/2$.
Since $J\left(G\right)$ is compact, there exist $s\in\N$ and $x_{1},\dots,x_{s}\in J\left(G\right)$, 
such that $J\left(G\right)\subset\bigcup_{j=1}^{s}B\left(x_{j},r\right)$.
By Lemma \ref{lem:finitely_primitivity} applied to the open sets
$U_{i}:=B\left(x_{i},r\right)$,  for $i\in\left\{ 1,\dots,s\right\} $,
we obtain that there exist $\ell\in\N$ and a finite set $I_{0}\subset I$
such that,  for all $j,k\in\left\{ 1,\dots,s\right\} $,  there exist
$\tau\left(j,k\right)\in I_{0}^{\ell}$ and $y_{j,k} \in\Chat$
with the property that $y_{j,k} \in f_{\tau\left(j,k\right)}^{-1}\left(x_{k}\right)\cap B\left(x_{j},r\right)$. 

In the following, let $I_{1}$ denote an arbitrary subset of $I$
containing $I_{0}$. For each $n\in\N$ and for all $j,k\in\left\{ 1,\dots,s\right\} $
we then have that 
\begin{eqnarray}
Z_{n+\ell}\left(I_{1},t,x_{k}\right) & = & \sum_{\omega\in I_{1}^{\ell}}\sum_{y\in f_{\omega}^{-1}\left(x_{k}\right)}\left\Vert f_{\omega}'\left(y\right)\right\Vert ^{-t}Z_{n}\left(I_{1},t,y\right)\ge\big\Vert f_{\tau\left(j,k\right)}'(y_{j,k} )\big\Vert ^{-t}Z_{n}\left(I_{1},t,y_{j,k} \right).\label{eq:boundeddistortion-proof-1}
\end{eqnarray}
We will now combine the previous estimate with the following consequence
of Koebe's distortion theorem. There exists a constant $C_{1}=C_{1}\left(r\right)$
such that, for each $m\in \N$,  $I'\subset I$ and for all  $z,z'\in J\left(G\right)$
with $d\left(z,z'\right)<r$,
\begin{equation}
C_{1}^{-1}Z_{m}\left(I',t,z'\right)\le Z_{m}\left(I',t,z\right)\le C_{1}Z_{m}\left(I',t,z'\right).\label{eq:koebe-Z_n}
\end{equation}
Combining (\ref{eq:boundeddistortion-proof-1}),  (\ref{eq:koebe-Z_n})
 and  $d\left(y_{j,k} ,x_{j}\right)<r$, we
obtain  $Z_{n+\ell}\left(I_{1},t,x_{k}\right)\ge\big\Vert f_{\tau\left(j,k\right)}'\left(y_{j,k} \right)\big\Vert ^{-t}C_{1}^{-1}Z_{n}\left(I_{1},t,x_{j}\right)$. Setting $S:=\min_{j,k\in\left\{ 1,\dots,s\right\} }\big\Vert f_{\tau\left(j,k\right)}'\left(y_{j,k} \right)\big\Vert ^{-t}>0$
and combining $J\left(G\right)\subset\bigcup_{j=1}^{s}B\left(x_{j},r\right)$
with (\ref{eq:koebe-Z_n}), we deduce that  $Z_{n+\ell}\left(I_{1},t,y\right)\ge SC_{1}^{-3}Z_{n}\left(I_{1},t,x\right)$, for all $n\in\N$ and  for all $x,y\in J\left(G\right)$. We have thus shown that  (\ref{eq:bounded-distortion-lowerbound}) holds with $C:=S^{-1}C_{1}^{3}$. 

\vspace{-7pt}
We now turn to the proof of the second assertion of the lemma. Let $x\in J\left(G\right)$. First suppose that $P\left(I_{1},t,x\right)<\infty$. Clearly, there exists
$n\ge2$ such that $Z_{n+\ell}\left(I_{1},t,x\right)<\infty$. By (\ref{eq:bounded-distortion-lowerbound})
we  have  $Z_{n}\left(I_{1},t,y\right)\le CZ_{n+\ell}\left(I_{1},t,x\right)<\infty$, 
for all $y\in J\left(G\right)$. Consequently, fixing one element  $a\in I_{0}$, we have for all $y\in J\left(G\right)$,
\begin{eqnarray*}
 &  & Z_{1}\left(I_{1},t,y\right)\big(\min_{z\in f_{a}^{-1}\left(J\left(G\right)\right)}\left\Vert f_{a}'\left(z\right)\right\Vert ^{-t}\big)^{n-1}\\
 & \le & \sum_{\omega\in I{}_{1}}\sum_{z\in f_{\omega}^{-1}\left(y\right)}\left\Vert f_{\omega}'\left(z\right)\right\Vert ^{-t}Z_{n-1}\left(I_{1},t,z\right)=Z_{n}\left(I_{1},t,y\right)\le CZ_{n+\ell}\left(I_{1},t,x\right)<\infty,
\end{eqnarray*}
which proves  $\sup_{y\in J\left(G\right)}Z_{1}\left(I_{1},t,y\right)<\infty$.
On the other hand, if $\sup_{y\in J\left(G\right)}Z_{1}\left(I_{1},t,y\right)<\infty$,
then 
\[
P\left(I_{1},t,x\right)=\limsup_{n\rightarrow\infty}\frac{1}{n}\log Z_{n}\left(I_{1},t,x\right)\le\log\sup_{y\in J\left(G\right)}Z_{1}\left(I_{1},t,y\right)<\infty,
\]
which finishes the proof of the second assertion.

In order to prove  (\ref{eq:bounded-distortion-property}), suppose that 
$P\left(I_1,t,x_{0}\right)<\infty$,  for some $x_{0}\in J\left(G\right)$,  and let $n\in \N$ and $x,y\in J(G)$. By (\ref{eq:bounded-distortion-lowerbound}) we have
\[
Z_{n}\left(I_{1},t,x\right)\le CZ_{n+\ell}\left(I_{1},t,y\right)=C\sum_{\omega\in I{}_{1}^{n}}\sum_{z\in f_{\omega}^{-1}\left(y\right)}\left\Vert f_{\omega}'\left(z\right)\right\Vert ^{-t}Z_{\ell}\left(I_{1},t,z\right)\le CZ_{n}\left(I_{1},t,y\right)\sup_{z\in J\left(G\right)}Z_{\ell}\left(I_{1},t,z\right).
\]
Now, by the second assertion of the lemma,  we have  $\sup_{z\in J\left(G\right)}Z_{\ell}\left(I_{1},t,z\right)<\infty$.
Hence,  the estimates in (\ref{eq:bounded-distortion-property}) hold  with $C':=C\sup_{z\in J\left(G\right)}Z_{\ell}\left(I_{1},t,z\right)$. 

Next, we prove the final assertion of the lemma. Suppose that $P\left(I_{1},t,x_{0}\right)=\infty$,
for some $x_{0}\in J\left(G\right)$. By the second assertion of the lemma, we have  $\sup_{y\in J\left(G\right)}Z_{1}\left(I_{1},t,y\right)=\infty$.
Let $x\in J(G)$. By (\ref{eq:bounded-distortion-lowerbound}) we have 
\[
Z_{1+\ell}\left(I_1,t,x\right)\ge C^{-1}\sup_{y\in J(G)} Z_{1}\left(I_{1},t,y\right)=\infty.
\]
Finally, using the estimate $Z_{1+\ell+1}\left(I_{1},t,x\right)\ge Z_{1+\ell}\left(I_1,t,x\right)\min_{z\in f_{a}^{-1}\left(J\left(G\right)\right)}\left\Vert f_{a}'\left(z\right)\right\Vert ^{-t}=\infty$,
one inductively verifies that $Z_{n}\left(I_{1},t,x\right)=\infty$, for all $n\ge 1+\ell$.  The proof 
is complete. 
\end{proof}
The third  assertion in the following proposition shows that an exhaustion
principle holds for $P\left(I,t,x\right)$. Recall that we say that
$\eta:\R\rightarrow\R\cup\left\{ \infty\right\} $ is convex if its epigraph $\mbox{epi}\left(\eta\right):=\left\{ \left(x,y\right)\in\R^{2}:y\ge\eta\left(x\right)\right\} $
is a convex set, and we say  that $\eta$ is closed if $\mbox{epi}\left(\eta\right)$
is closed subset of $\R^2$. Note that $\mbox{epi}(\eta)$ is closed if and only if $\eta$ is lower semicontinuous (\cite[Theorem 7.1]{rockafellar-convexanalysisMR0274683}). The following properties are well-known for countable topological Markov chains satisfying the finite primitivity condition (\cite{MR2003772}). The proof of Proposition \ref{prop:exhaustion_principle-etc} (\ref{enu:exhaustion-principle}) is inspired by (\cite[Theorem 2.1.5]{MR2003772}) .
\begin{prop}
\label{prop:exhaustion_principle-etc}Let $G=\left\langle f_{i}:i\in I\right\rangle $
be a rational semigroup. Suppose that either (1) $G$ is a hyperbolic
rational semigroup which contains an element of degree at least two 
or (2)  $G$ is nicely expanding. Then,  for
each $t\in\R$ and for each $x\in J\left(G\right)$,  the following holds. 
\begin{enumerate}
\item \label{enu:pressure-limitexists}If $P\left(I,t,x_{0}\right)<\infty$,
for some $x_{0}\in J\left(G\right)$, then 
$P\left(I,t,x\right)=\lim_{n\rightarrow\infty}n^{-1}\log Z_{n}\left(I,t,x\right)$.

\item \label{enu:P-I-is-locally-constant}There exists a neighborhood $V$
of $J\left(G\right)$, $V\subset\Chat\setminus P\left(G\right)$,
such that $z\mapsto P\left(I,t,z\right)$ is constant on $V$. 
\item \label{enu:exhaustion-principle}
$ \sup_{F\subset I,\,\card\left(F\right)<\infty}P\left(F,t,x\right)=P\left(I,t,x\right)$

\item \label{enu:pressure-is-convex}The map $s\mapsto P\left(I,s,x\right)$
is a closed convex function with values in $\R\cup\left\{ \infty\right\} $. 
\item \label{enu:pressure-poincareseries}For all $I_{1}\subset I$ with
$I_{0}\subset I_{1}$, where $I_{0}$ is the finite set in Lemma \ref{lem:bounded_distortion_lemma},
we have 
\[
P\left(I_{1},t,x\right)=\inf\left\{ \beta\in\R:\sum_{n\in\N}Z_{n}\left(I_{1},t,x\right)\e^{-\beta n}<\infty\right\} .
\]

\end{enumerate}
\end{prop}
\begin{proof}
We give the proof under the assumption that $G$ is hyperbolic and contains an element of degree at least two. By replacing $P(G)$ by $P_0(G)$, the proposition can be proved under the assumption that $G$ is nicely expanding. 

First note that, for all
$m,n\in\N$,  we have $Z_{n+m}\left(I,t,x\right)=\sum_{\omega\in I{}^{n}}\sum_{z\in f_{\omega}^{-1}\left(x\right)}\left\Vert f_{\omega}'\left(z\right)\right\Vert ^{-t}Z_{m}\left(I,t,z\right)$.
Since $P\left(I,t,x_{0}\right)<\infty$, Lemma \ref{lem:bounded_distortion_lemma}
(\ref{eq:bounded-distortion-property}) implies that there exists
a constant $C'>1$ such that 
\[
\sum_{\omega\in I{}^{n}}\sum_{z\in f_{\omega}^{-1}\left(x\right)}\left\Vert f_{\omega}'\left(z\right)\right\Vert ^{-t}Z_{m}\left(I,t,z\right)\le C'\sum_{\omega\in I{}^{n}}\sum_{z\in f_{\omega}^{-1}\left(x\right)}\left\Vert f_{\omega}'\left(z\right)\right\Vert ^{-t}Z_{m}\left(I,t,x\right)=C' Z_{n}\left(I,t,x\right)Z_{m}\left(I,t,x\right).
\]
Hence, the sequence $\left(a_{n}\right)\in\R^{\N}$, given by $a_{n}:=$$\log Z_{n}\left(I,t,x\right)$, $n\in \N$, 
is almost subadditive in the sense that  $a_{n+m}\le a_{n}+a_{m}+\log C'$,
for all $n,m\in\N$. Hence, $\lim_{n}a_{n}/n$
exists and the proof of (\ref{enu:pressure-limitexists}) is complete. 

For the proof of (\ref{enu:P-I-is-locally-constant}), one first observes
that $z\mapsto P\left(I,t,z\right)$
is constant on $J\left(G\right)$  by Lemma \ref{lem:bounded_distortion_lemma}. Since $G$ is hyperbolic, there
exists $r>0$ such that $B\left(y,r\right)\subset\Chat\setminus P\left(G\right)$, for each $y\in J\left(G\right)$.
Set $V:=\bigcup_{y\in J\left(G\right)}B\left(y,r\right)$. An application
of Koebe's distortion theorem shows that there exists a constant $C_{1}=C_{1}\left(r\right)$
such that $Z_{n}\left(I,t,y\right)C_{1}\ge Z_{n}\left(I,t,z\right)\ge C_{1}^{-1}Z_{n}\left(I,t,y\right)$,  for each $y\in J\left(G\right)$, $z\in B\left(y,r\right)$
and $n\in\N$.
It follows that $z\mapsto P\left(I,t,z\right)$ is constant on $V$. 

Let us turn to the proof of  (\ref{enu:exhaustion-principle}).  Clearly, 
we have $P\left(F,t,x\right)\le P\left(I,t,x\right)$, for each $F\subset I$. Hence,  we have  $\sup_{F\subset I}P\left(F,t,x\right)\le P\left(I,t,x\right)$.
For the opposite inequality, we consider two cases. First suppose
that $P\left(I,t,x\right)<\infty$ for any $x\in J(G)$ and let $\epsilon>0$. There exists $n\in\N$,  such that $n^{-1}\log Z_{n}\left(I,t,x\right)>P\left(I,t,x\right)-\epsilon$
and  $n^{-1}\log\left(C'\right)<\epsilon$, where $C'>1$
is the constant from (\ref{eq:bounded-distortion-property}) in Lemma \ref{lem:bounded_distortion_lemma}.
Let $I_{0}$ denote the finite subset of $I$ given by Lemma \ref{lem:bounded_distortion_lemma}.
Choose a finite set  $F$ with $I_{0}\subset F\subset I$ such that
$n^{-1}\log Z_{n}\left(F,t,x\right)>P\left(I,t,x\right)-2\epsilon$.
Let $k\in \N$. By (\ref{eq:bounded-distortion-property})
we have 
\begin{eqnarray*}
Z_{kn}\left(F,t,x\right) & = & \sum_{\omega_{k}\in F^{n},y_{k}\in f_{\omega_{k}}^{-1}\left(x\right)}\left\Vert f_{\omega_{k}}'\left(y_{k}\right)\right\Vert ^{-t}\sum_{\omega_{k-1}\in F^{n},y_{k-1}\in f_{\omega_{k-1}}^{-1}\left(y_{k}\right)}\big\Vert f_{\omega_{k-1}}'\left(y_{k-1}\right)\big\Vert ^{-t}\quad\dots\\
 &  & \quad\dots\quad\sum_{\omega_{1}\in F^{n},y_{1}\in f_{\omega_{1}}^{-1}\left(y_{2}\right)}\left\Vert f_{\omega_{1}}'\left(y_{1}\right)\right\Vert ^{-t}\,\,\,\ge\,\,\, (C')^{-k}\left(Z_{n}\left(F,t,x\right)\right)^{k},
\end{eqnarray*}
which gives 
\begin{eqnarray*}
\frac{1}{kn}\log Z_{kn}\left(F,t,x\right) & \ge & -\frac{1}{n}\log C'+\frac{1}{n}\log Z_{n}\left(F,t,x\right)\ge-\epsilon+P\left(I,t,x\right)-2\epsilon=P\left(I,t,x\right)-3\epsilon.
\end{eqnarray*}
We get  $P\left(F,t,x\right)\ge P\left(I,t,x\right)-3\epsilon$, because the previous estimate holds for every $k\in \N$. 
Since $\epsilon$ was chosen to be arbitrary, the proof is complete in the case
that $P\left(I,t,x\right)<\infty$ for any $x\in J(G)$. Now, we consider
the remaining case $P\left(I,t,x\right)=\infty$ for any $x\in J(G)$. Let $N\in\N$. Clearly,
there exists $n\in\N$ such that $\left(n+\ell\right)^{-1}\log Z_{n}\left(I,t,x\right)>N-\epsilon$
and such that $\left(n+\ell\right)^{-1}\log\left(C\right)<\epsilon$,
where $C>1$ and $\ell\in \N$ are  given by (\ref{eq:bounded-distortion-lowerbound}) in Lemma \ref{lem:bounded_distortion_lemma}. Choose a finite
set $F$ with  $I_{0}\subset F\subset I$ such that $\left(n+\ell\right)^{-1}\log Z_{n}\left(F,t,x\right)>N-2\epsilon$.
Let $k\in \N$. By  (\ref{eq:bounded-distortion-lowerbound}),
we have 
\begin{eqnarray*}
Z_{k\left(n+\ell\right)}\left(F,t,x\right) & = & \sum_{\omega_{k}\in F^{n+\ell},y_{k}\in f_{\omega_{k}}^{-1}\left(x\right)}\left\Vert f_{\omega_{k}}'\left(y_{k}\right)\right\Vert ^{-t}\sum_{\omega_{k-1}\in F^{n+\ell},y_{k-1}\in f_{\omega_{k-1}}^{-1}\left(y_{k}\right)}\big\Vert f_{\omega_{k-1}}'\left(y_{k-1}\right)\big\Vert ^{-t}\quad\dots\\
 &  & \quad\dots\quad\sum_{\omega_{1}\in F^{n+\ell},y_{1}\in f_{\omega_{1}}^{-1}\left(y_{2}\right)}\left\Vert f_{\omega_{1}}'\left(y_{1}\right)\right\Vert ^{-t}\,\,\,\ge\,\,\, C{}^{-k}\left(Z_{n}\left(F,t,x\right)\right)^{k},
\end{eqnarray*}
which gives
\begin{eqnarray*}
\frac{1}{k\left(n+\ell\right)}\log Z_{k\left(n+\ell\right)}\left(F,t,x\right) & \ge & -\frac{1}{n+\ell}\log C+\frac{1}{n+\ell}\log Z_{n}\left(F,t,x\right)\ge N-3\epsilon.
\end{eqnarray*}
We obtain  $P\left(F,s,x\right)>N-3\epsilon$,  and letting $N$ tend to infinity, finishes the proof of (\ref{enu:exhaustion-principle}). 

\vspace{-5pt}
To prove  (\ref{enu:pressure-is-convex}), let $x\in J\left(G\right)$.
If $\card\left(I\right)$ is finite, then it is standard to deduce (\ref{enu:pressure-is-convex}) from   H\"older's inequality. 
 Now, suppose that $I=\N$. Let $\left(F_{n}\right)_{n\in \N}$
be a sequence of finite subsets of $I$ such that $F_{n}\subset F_{n+1}$,
for each $n\in\N$, and $\bigcup_{n\in\N}F_{n}=I$. For each $n\in\N$,
we define $g_{n}:\R\rightarrow\R$, given by $g_{n}\left(s\right):=P\left(F_{n},s,x\right)$,
and $g:\R\rightarrow\R\cup\left\{ \pm\infty\right\} $ given by $g\left(s\right):=P\left(I,s,x\right)$.
For each $n\in\N$,  $g_{n}$ is a real-valued convex function. In particular,  $\mbox{epi}\left(g_{n}\right)$ is a closed convex subset of $\R^{2}$.
By Proposition \ref{prop:exhaustion_principle-etc} (\ref{enu:exhaustion-principle})
we have that $g_{n}(s)\le g_{n+1}(s)$, for each $s\in \R$ and $n\in \N$,  which implies  $\mbox{epi}\left(g\right)=\bigcap_{n\in \N}\mbox{epi}\left(g_{n}\right)$. Hence,  $\mbox{epi}\left(g\right)$ is  closed and convex. Since $g_n$ is real-valued and $g_n \le g$, we have $g(s)>-\infty$, for each $s\in \R$. 

\vspace{-7pt}
Finally, to prove  (\ref{enu:pressure-poincareseries}),  a straightforward
modification of the proof of \cite[Theorem 3.16]{JaerischDissertation11}
(see also \cite{JaerischKessebohmer10})  shows that 
\[
P\left(I_{1},t,x\right)=\inf\Big\{ \beta\in\R:\limsup_{T\rightarrow\infty}\sum_{n\in\N,n>T}Z_{n}\left(I_{1},t,x\right)\e^{-\beta n}<\infty\Big\} .
\]
Further, we clearly have that 
\begin{equation}
\inf\Big\{ \beta\in\R:\limsup_{T\rightarrow\infty}\sum_{n\in\N,n>T}Z_{n}\left(I_{1},t,x\right)\e^{-\beta n}<\infty\Big\} \le\inf\Big\{ \beta\in\R:\sum_{n\in\N}Z_{n}\left(I_{1},t,x\right)\e^{-\beta n}<\infty\Big\} .\label{eq:tailpoincare-vs-poincare}
\end{equation}
Hence, in the case that $P\left(I_{1},t,x\right)=\infty$, the assertion
in (\ref{enu:pressure-poincareseries}) follows. If $P\left(I_{1},t,x\right)<\infty$, then  (\ref{eq:bounded-distortion-property})
gives  $Z_{n}\left(I_{1},t,x\right)<\infty$,  for all $n\in\N$. Hence,
equality in (\ref{eq:tailpoincare-vs-poincare}) holds and the assertion
in (\ref{enu:pressure-poincareseries}) follows. 
\end{proof}

\subsection{The Gurevi\v c pressure}
\label{Gurevic}
Throughout this subsection, let $I\subset\N$ be the finite set $\left\{ 1,\dots,n\right\} $,
for some $n\in\N$, or let $I=\N$, endowed with the discrete topology.
Let $\left(f_{i}\right)_{i\in I}\in\Rat^{I}$ and let $\tilde{f}:J\left(\tilde{f}\right)\rightarrow J\left(\tilde{f}\right)$
be the associated skew product. Suppose that $G=\langle f_{i}:i\in I\rangle $ is expanding with respect to 
$\{ f_{i}:i\in I\}.$  For each $n\in\N$,  we set 
\[
\mathcal{P}_{n}\left(t\right):=\mathcal{P}\left(t\tilde{\varphi}_{|J^{\left\{ 1,\dots,n\right\} \cap I}},\tilde{f}_{|J^{\left\{ 1,\dots,n\right\} \cap I}}\right).
\]
\begin{lem}
\label{lem:pressures-coincide}Suppose that $G=\left\langle f_{i}:i\in I\right\rangle $
is a nicely expanding rational semigroup. Then, for each $t\in\R$
and for each $x\in J\left(G\right)$, we have
$\mathcal{P}\left(t\right)=\lim_{n\rightarrow\infty}\mathcal{P}_{n}\left(t\right)=P\left(I,t,x\right)$. 
\end{lem}
\begin{proof}
We may assume that $I=\N$. Let $t\in \R$. Our first aim is to prove that
\begin{equation} \label{exhaust1}
\mathcal{P}\left(t\right)=\sup_{F\subset I,\,\card\left(F\right)<\infty}\mathcal{P}\left(t\tilde{\varphi}_{|J^{F}},\tilde{f}_{|J^{F}}\right).
\end{equation}
To prove (\ref{exhaust1}), it suffices to verify that, if $K\subset J(\tilde{f})$ is compact and $\tilde{f}(K)=K$, then there exists $F\subset I$ finite
such that $K\subset J^{F}$ and $J^F$ is compact. Let $p_{k}:I^{\N}\rightarrow I$, $p_{k}\left(\omega\right):=\omega_{k}$,
denote the projection on the $k$th symbol for each $k\in\N$. As the set 
$p_{1}(\pi_{1}\left(K\right))\subset I$ is compact, we have that $F:=p_{1}(\pi_{1}\left(K\right))$
is finite. Using that  $\tilde{f}\left(K\right)= K$, we conclude
that $p_{k}(\pi_{1}\left(K\right))=p_{1}(\pi_{1}(\tilde{f}^{k-1}\left(K\right)))= p_{1}(\pi_{1}\left(K\right))= F$, for each $k\in \N$.
Hence, $\pi_{1}\left(K\right)\subset F^{\N}$. Since
$G$ is expanding, we have  $J\left(\tilde{f}\right)=\bigcup_{\omega\in I^{\N}}J^{\omega}$ by Lemma \ref{lem:expanding-implies-skewproductjulia-is-closed}.
Combining this with $\pi_{1}\left(K\right)\subset F^{\N}$, we see that  $K\subset J^{F}$. Since $J^I$ is closed in $I^\N \times \Chat$ by Lemma \ref{lem:expanding-implies-skewproductjulia-is-closed}, we obtain that $J^{F}=J^{I}\cap\big(F^{\N}\times\Chat\big)$
is compact, which completes the proof of (\ref{exhaust1}). By  (\ref{exhaust1}) we have $\mathcal{P}(t)=
\lim _{n\rightarrow \infty }\mathcal{ P}_{n}(t).$

To prove $\lim_{n}\mathcal{P}_{n}\left(t\right)=P(I,t,x)$, for $x\in J(G)$, it suffices to consider one point $x_0\in J(G)$. If $F$ is finite and $x\in J_F$, then it follows from  \cite[Lemma 3.6 (2) and (4)]{MR2153926} that 
\begin{equation} \label{exhaust2}
\mathcal{P}\left(t\tilde{\varphi}_{|J^{F}},\tilde{f}_{|J^{F}}\right)=P(F,t,x).
\end{equation}
Now fix some $x_0 \in J(f_1)\subset \bigcap_{k \in \N} J(\langle f_1,\dots,f_k\rangle)$. By (\ref{exhaust2}), we have, for each $n\in \N$, 
$\mathcal{P}_n\left(t\right)=P(\{1,\dots, n\},t,x_0)$. Letting $n$ tend to infinity, we obtain  $\mathcal{P}\left(t\right)=\lim_{n}P(\{1,\dots, n\},t,x_0)$ by (\ref{exhaust1}). 
By Proposition \ref{prop:exhaustion_principle-etc}
(\ref{enu:exhaustion-principle}) we have  $\lim_{n\rightarrow\infty}P\left(\{1,\dots,n\},t,x_{0}\right)=P\left(I,t,x_{0}\right)$, which completes the proof. 
\end{proof}

\section{Bowen's Formula for Pre-Julia sets}
\label{Bowen}
In this section we prove Bowen's formula for pre-Julia sets,
which is one of the main results of this paper.

\begin{defn}[\cite{{MR1625124}}]
Let $G$ be a rational semigroup and let $\delta \geq 0.$ A Borel probability measure $\mu$ on $\Chat$ is called
$\delta$-subconformal if,  for each $g\in G$ and for each Borel set
$A\subset \Chat$,  we have  
\[
\mu\left(g\left(A\right)\right)\le\int_{A}\left\Vert g'\right\Vert ^{\delta}d\mu.
\]
Also, we define the following critical exponent 
\[
u\left(G\right):=\inf\left\{ v\ge 0:\mbox{ there exists a }v\mbox{-subconformal measure for }G\right\} .
\]
\end{defn}
\begin{lem}
\label{lem:critical-exponents}Let $I$ be a countable set and let $G=\left\langle f_{i}:i\in I\right\rangle $ be a rational semigroup. 
\begin{enumerate}
\item \label{enu:criticalexponents-easy-inequalities-1}

\begin{itemize}
\item $s\left(G,x\right)\le t\left(I,x\right)$, for each $x\in\Chat$. 
\item $s\left(G\right)\le t\left(I\right)$.
\item $u\left(G\right)\le s\left(G\right)$. 
\end{itemize}
\item \label{enu:criticalexponents-freesemigroup-2}If $G$ is a free semigroup,
then for each $x\in\Chat$, 
\[
s\left(G,x\right)=t\left(I,x\right)\mbox{ and }s\left(G\right)=t\left(I\right).
\]
\item \label{enu:criticalexponents-hyperbolic-independence-ofx-in-julia-3}If
$G$ is hyperbolic containing an element of degree at least two (resp.
nicely expanding), then there exists a neighborhood $V$ of $J\left(G\right)$
with $V\subset\Chat\setminus P\left(G\right)$ (resp. $V\subset\Chat\setminus P_{0}\left(G\right)$)
such that $t\left(I\right)=t\left(I,x\right)$, for each $x\in V$. 
\end{enumerate}
\end{lem}
\begin{proof}
The assertions in (\ref{enu:criticalexponents-easy-inequalities-1})
were proved in \cite[Theorem 4.2]{MR1625124}. The assertions in (\ref{enu:criticalexponents-freesemigroup-2})
follow immediately from the definition.  In order to prove (\ref{enu:criticalexponents-hyperbolic-independence-ofx-in-julia-3}),
we  verify the following claim.
\begin{clm*} Let $t\geq 0$. If $\sum_{n\in\N}Z_{n}\left(I,t,x\right)=\infty$,
for some  $x\in J\left(G\right)$, then $\sum_{n\in\N}Z_{n}\left(I,t,y\right)=\infty$,
for every $y\in\Chat$. 
\end{clm*}
\begin{proof}[Proof of Claim]
Suppose that $\sum_{n\in\N}Z_{n}\left(I,t,x\right)=\infty$,
for some  $x\in J\left(G\right)$. We first show that
$\sum_{n\in\N}Z_{n}\left(I,t,y\right)=\infty$ for each $y\in J\left(G\right)$.
Since $G$ is hyperbolic,  containing an element of degree at least
two (resp. nicely expanding), we can apply Lemma \ref{lem:bounded_distortion_lemma}.
 If $P\left(I,t,x\right)<\infty$, then we have $\sum_{n\in\N}Z_{n}\left(I,t,y\right)=\infty$, for each $y\in J(G)$, because there
exists a constant $C'>1$ such that $Z_{n}\left(I,t,y\right)\ge C'^{-1}Z_{n}\left(I,t,x\right)$. If  $P\left(I,t,x\right)=\infty$, then it follows
from the last assertion in Lemma \ref{lem:bounded_distortion_lemma} that $Z_{n}\left(I,t,y\right)=\infty$,  for each $y\in J(G)$ and 
for all sufficiently large $n\in\N$. The next step is to show that $\sum_{n\in\N}Z_{n}\left(I,t,y\right)=\infty$
for each $y$ in a neighborhood  $V$ of $J\left(G\right)$. Since $G$
is hyperbolic (resp. nicely expanding), there exists $r>0$ such that
$B\left(y,r\right)\subset\Chat\setminus P\left(G\right)$ (resp. $B\left(y,r\right)\subset\Chat\setminus P_{0}\left(G\right)$),
for each $y\in J\left(G\right)$. Set $V:=\bigcup_{y\in J\left(G\right)}B\left(y,r\right)$.
An application of Koebe's distortion theorem shows that there exists
a constant $C_{1}=C_{1}\left(r\right)$ such that $\sum_{n\in\N}Z_{n}\left(I,t,z\right)\ge C_{1}^{-1}\sum_{n\in\N}Z_{n}\left(I,t,y\right)=\infty$,
for each $y\in J\left(G\right)$ and $z\in B\left(y,r\right)$.

In order to finish the proof of the claim, we will distinguish two cases.

\emph{Case (1): $\card\left(J\left(G\right)\right)>1$.} In this case, we have $\card(J(G))\geq 3$ by Lemma \ref{lem:moebiussemigroup-expanding-implies-J-not-twoelements}. First, let
$y\in\Chat\setminus E\left(G\right)$ be given. Then there exists $\omega\in I^{*}$ such that $f_{\omega}^{-1}\left(y\right)\cap V\neq\emptyset$ (\cite[Lemma 3.2]{MR1397693}, \cite[Lemma 2.3]{MR1767945}).
Since $\sum_{n\in\N}Z_{n}\left(I,t,z\right)=\infty$, for each $z\in V$,
we conclude that also $\sum_{n\in\N}Z_{n}\left(I,t,y\right)=\infty$.
Finally, by Proposition \ref{prop:nicelyexpanding-characterisation},  each element of $G \cap \Aut(\Chat)$ is loxodromic. Since  $\card\left(J\left(G\right)\right)\ge 3$,  we have $\card(E(G))\le 2$ (\cite[Lemma 3.3]{MR1397693}, \cite[Lemma 2.3]{MR1767945}). Therefore,  by Lemma \ref{lem:moebiussemigroup-expanding-implies-goodexceptionalset} and  \cite[Theorem 4.1.2]{MR1128089},  for each  $y\in E\left(G\right)$,  there is $g\in G$ such that
$g\left(y\right)=y$ and $\left\Vert g'\left(y\right)\right\Vert<1$. Hence,
$\sum_{n\in\N}Z_{n}\left(I,t,y\right)=\infty$. 

\emph{Case (2): $\card\left(J\left(G\right)\right)=1$}. Let $g\in G$. By Lemma \ref{lem:moebiussemigroup-expanding-implies-loxodromic}, we have that $g$ is loxodromic. Let $a$ denote the repelling fixed point of $g$, and let $b$ denote the attracting fixed point of $g$. Since $a \in J(G)$, we have that,  for each $y\in\Chat\setminus\left\{ b\right\} $, there
exists $n\in\N$ such that $g^{-n}\left(z\right)\in V$. Hence, 
$\sum_{n\in\N}Z_{n}\left(I,t,y\right)=\infty$. Finally, since $g\left(b\right)=b$
and $\left\Vert g'\left(b\right)\right\Vert <1$, we have  $\sum_{n\in\N}Z_{n}\left(I,t,b\right)=\infty$.
The proof of the claim is complete.
\end{proof}

Let us now complete the proof of the lemma. It follows from the  claim that  $t\left(I,x\right)\le t\left(I,y\right)$, for each
$x\in J\left(G\right)$ and $y\in\Chat$. In particular,  we have $t\left(I\right)=t\left(I,x\right)$,  for every $x\in J\left(G\right)$.
Again, by Koebe's distortion theorem, letting $V$ be the neighborhood of $J(G)$ in the proof of Claim, we conclude that $t\left(I\right)=t\left(I,x\right)=t\left(I,y\right)$
for all $x,y\in V$, which proves the  assertion in (\ref{enu:criticalexponents-hyperbolic-independence-ofx-in-julia-3}).  Thus we have proved Lemma \ref{lem:critical-exponents}.
\end{proof}
\begin{prop}
\label{prop:critical-exponents}Let $I$ be a  countable
set. Let $G=\left\langle f_{i}:i\in I\right\rangle $ denote a nicely
expanding rational semigroup. Then  $\beta\mapsto\mathcal{P}\left(\beta\right)$
is strictly decreasing on $\mathcal{F}:=\left\{ x\in\R:\mathcal{P}\left(x\right)<\infty\right\} $,
and we have  
\begin{equation}
t\left(I\right)=\inf\left\{ \beta\in\R:\mathcal{P}\left(\beta\right)\le0\right\} =\inf\left\{ \beta\in\R:\mathcal{P}\left(\beta\right)<0\right\} =\sup_{F\subset I,\card\left(F\right)<\infty}t\left(F\right).\label{eq:criticalexponent-vs-pressure}
\end{equation}
Furthermore, we have that $\mathcal{P}\left(t\left(I\right)\right)\le0$. \end{prop}
\begin{proof}
We may assume that $I=\{ 1,\dots ,n\}$, for some $n\in \N$, or $I=\N$. Using that  $G$ is expanding,  it is straightforward to verify that   the map $\beta\mapsto\mathcal{P}\left(\beta\right)$
is strictly decreasing on $\mathcal{F}:=\left\{ \beta \in\R:\mathcal{P}\left(\beta \right)<\infty\right\} $.
Let $x\in J\left(G\right)$. 
 Since $G$ is nicely expanding,  Lemma \ref{lem:pressures-coincide} yields  $\mathcal{P}\left(\beta\right)=P\left(I,\beta,x\right)$.  By definition of $t\left(I,x\right)$  we have 
\[
\inf\left\{ \beta\in\R:\mathcal{P}\left(\beta\right)\le0\right\} \le t\left(I,x\right)\le\inf\left\{ \beta\in\R:\mathcal{P}\left(\beta\right)<0\right\} .
\]
Since $\beta\mapsto\mathcal{P}\left(\beta\right)$ is strictly decreasing
on $\mathcal{F}$, we have $\inf\left\{ \beta\in\R:\mathcal{P}\left(\beta\right)\le0\right\} =\inf\left\{ \beta\in\R:\mathcal{P}\left(\beta\right)<0\right\} $,
which proves that $t\left(I,x\right)=\inf\left\{ \beta\in\R:\mathcal{P}\left(\beta\right)\le0\right\} =\inf\left\{ \beta\in\R:\mathcal{P}\left(\beta\right)<0\right\} $. By Lemma \ref{lem:critical-exponents}
(\ref{enu:criticalexponents-hyperbolic-independence-ofx-in-julia-3})
we have  $t\left(I,x\right)=t\left(I\right)$. 
It remains to show that  $t\left(I\right)=\sup_{F\subset I,\card\left(F\right)<\infty}t\left(F\right)$.
Clearly, we have  $t\left(F\right)\le t\left(I\right)$,  for each
$F\subset I$. Hence,   $\sup_{F\subset I,\card\left(F\right)<\infty}t\left(F\right)\le t\left(I\right)$.
For the reversed inequality, let $\epsilon>0$. Since $t\left(I\right)=\inf\left\{ \beta\in\R:\mathcal{P}\left(\beta\right)\le0\right\} $,
we have  $\mathcal{P}\left(t\left(I\right)-\epsilon\right)>0$.
Hence, by Lemma \ref{lem:pressures-coincide}, there exists $n\in\N$
such that $\mathcal{P}_{n}\left(t\left(I\right)-\epsilon\right)>0$.
Therefore, we have that 
\[
\sup_{F\subset I,\card\left(F\right)<\infty}t\left(F\right)\ge t\left(I\cap\left\{ 1,\dots,n\right\} \right)=\inf\left\{ \beta\in\R:\mathcal{P}_{n}\left(\beta\right)\le0\right\} \ge t\left(I\right)-\epsilon.
\]
Letting  $\epsilon$ tend to zero, finishes the proof of  (\ref{eq:criticalexponent-vs-pressure}).

Next, we will show that $\mathcal{P}\left(t\left(I\right)\right)\le0$.
Since $t\left(I\right)=\inf\left\{ \beta\in\R:\mathcal{P}\left(\beta\right)\le0\right\} $
and $\beta\mapsto\mathcal{P}\left(\beta\right)$ is a closed function
by Proposition \ref{prop:exhaustion_principle-etc} (\ref{enu:pressure-is-convex}),
the claim follows. The proof is complete. 
\end{proof}
In order to state the main theorem of this section, we  give
the  definition of the following subsets of $\Jpre\left(G\right)$. 
\begin{defn}
For a rational semigroup $G$, 
we set 
\[
\Jur\left(G\right):=\bigcup_{H\text{ finitely generated subsemigroup of }G}\Jpre\left(H\right)\text{ and }\Jr\left(G\right):=\bigcup_{\gamma \in G^{\N}:\exists g\in G:\gamma_{i}=g\text{ infinitely often}}J_{\gamma}.
\]
 
\end{defn}
\begin{rem*} If $G=\left\langle f_{i}:i\in I\right\rangle $ is
a rational semigroup,  where $I=\{ 1,\dots ,n\}$, for some $n\in \N$, or $I=\N$, then 
\[
\Jur\left(G\right)=\bigcup_{F\subset I:\card\left(F\right)<\infty}J_{F} \text{ and } \Jr\left(G\right)=\bigcup_{\omega\in I^{\N}:\liminf_{i}\omega_{i}<\infty}J_{\omega}.
\]
We remark that the subscript \emph{ur} in  $\Jur$ means uniformly radial, and the  subscript \emph{r} in  $\Jr$ means radial. 
\end{rem*}

\begin{thm}
\label{thm:bowen-for-prejulia}For a nicely expanding rational semigroup
$G=\left\langle f_{i}:i\in I\right\rangle $,  the following holds.
\begin{enumerate}
\item The $u\left(G\right)$-dimensional outer Hausdorff measure of $\Jpre$$\left(G\right)$
is finite. In particular, we have that \linebreak $\dim_{H}\left(\Jpre\left(G\right)\right)\le u\left(G\right)$.
\item If $I$ is countable, then 
$\dim_{H}\left(\Jpre\left(G\right)\right)\le u\left(G\right)\le s\left(G\right)\le 
t\left(I\right)=\inf \{ \beta \in \R : \mathcal{P}(\beta )<0\} $. 
\item If $I$ is  countable, and if $\left\{ f_{i}:i\in I\right\} $
satisfies the open set condition, then we have 
\[
\dim_{H}\left(\Jur\left(G\right)\right)=\dim_{H}\left(\Jr\left(G\right)\right)=\dim_{H}\left(\Jpre\left(G\right)\right)=u\left(G\right)=s\left(G\right)=
t\left(I\right) =\inf \{ \beta \in \mathbb{R} : \mathcal{P}(\beta )<0\}.
\]

\end{enumerate}
\end{thm}
\begin{proof}
We start with the proof of (1). Let $\delta:=u\left(G\right)$ and let $\left(\delta_{n}\right)_{n\in\N}$
be a sequence such that $\lim _{n}\delta _{n}=\delta $, 
 such that $\delta _{n}>\delta $ for each $n\in \N$, 
and such that there exists a $\delta _{n}$-subconformal measure 
$\mu _{\delta _{n}}$ for each $n\in \N$. We may assume that $\left(\mu_{\delta_{n}}\right)_{n\in\N}$
converges weakly to a Borel probability  measure $\mu_{\delta}$ on $\Chat$. Then the measure  $\mu_{\delta}$ is  $\delta$-subconformal.
By \cite[Proposition 4.3]{MR1625124},  Lemmas \ref{lem:moebiussemigroup-expanding-implies-J-not-twoelements},  \ref{exceptionalset-good}, 
\cite[Lemma 2.3]{MR1767945},  and \cite[Theorem 4.1.2]{MR1128089}   we have  $\supp\left(\mu_{\delta}\right)\supset J\left(G\right)$.
In order to show that the $\delta$-dimensional Hausdorff
measure of $\Jpre$$\left(G\right)$ is finite, we will show that
there exists a constant $C>0$ such that,  for all $z\in\Jpre\left(G\right)$, 
\begin{equation}
\limsup_{r\rightarrow0}\frac{\mu_{\delta}\left(B\left(z,r\right)\right)}{r^{\delta}}\ge C.\label{eq:localdimension-lowerbound}
\end{equation}
It then follows from  the uniform mass distribution principle (\cite[Proposition 4.9 (b)]{falconerfractalgeometryMR2118797})
and its proof that the $\delta$-dimensional outer Hausdorff
measure of $\Jpre\left(G\right)$ is finite. Hence, $\dim_{H}\left(\Jpre\left(G\right)\right)\le\delta$.
In order to prove (\ref{eq:localdimension-lowerbound}) we extend
 \cite[Proof of Theorem 3.4]{MR1625124} to our setting.
Since $G$ is nicely expanding, we have  $d\left(J\left(G\right),P_{0}\left(G\right)\right)>0$,
and we can fix some $0<a<d\left(J\left(G\right),P_{0}\left(G\right)\right)/2$.
Let $z\in\Jpre\left(G\right)$. Then there exists $\omega\in I^{\N}$ such
that $z\in J_{\omega}$. Hence,  we have $f_{\omega|_{n}}\left(z\right)\in J_{\sigma^{n}\omega}\subset\Jpre\left(G\right)$, for each $n\in\N$. We set $g_{n}:=f_{\omega|_{n}}$,
$z_{n}:=g_{n}\left(z\right)$ and we denote by $S_{n}$ the unique
holomorphic branch of $g_{n}^{-1}$ on $B\left(z_{n},a\right)$ such
that $S_{n}\left(g_{n}\left(z\right)\right)=z$. It follows from Koebe's
distortion theorem that there is $c_{0}>1$ such that  $c_{0}^{-1}\le\left\Vert S_{n}'\left(x\right)\right\Vert /\left\Vert S_{n}'\left(y\right)\right\Vert \le c_{0}$, for all $z\in\Jpre\left(G\right)$,
$n\in\N$ and for all $x,y\in B\left(z_{n},a\right)$. 
We conclude that 
\[
S_{n}\left(B\left(z_{n},a\right)\right)\subset B\left(z_{0},ac_{0}\left\Vert S_{n}'\left(z_{n}\right)\right\Vert \right).
\]
Since $J\left(G\right)$ is compact and since $\supp\left(\mu_{\delta}\right)\supset J\left(G\right)$, 
we  have $M\left(a\right):=\inf_{z\in J\left(G\right)}\mu_{\delta}\left(B\left(z,a\right)\right)>0$.
Setting $r_{n}:=ac_{0}\left\Vert S_{n}'\left(z_{n}\right)\right\Vert $
and using that $\mu_{\delta}$ is $\delta$-subconformal, we estimate
\begin{eqnarray*}
\mu_{\delta}\left(B\left(z_{0},r_{n}\right)\right) & \ge & \mu_{\delta}\left(S_{n}\left(B\left(z_{n},a\right)\right)\right)\ge\int_{B\left(z_{n},a\right)}\left\Vert S_{n}'\right\Vert ^{\delta}d\mu_{\delta}\\
 & \ge & c_{0}^{-\delta}\left\Vert S_{n}'\left(z_{n}\right)\right\Vert ^{\delta}\mu_{\delta}\left(B\left(z_{n},a\right)\right)\ge c_{0}^{-\delta}\left(ac_{0}\right)^{-\delta}r_{n}^{\delta}M\left(a\right).
\end{eqnarray*}
Since $G$ is expanding, $r_{n}$ tends to zero as $n$ tends to infinity.
Hence, (\ref{eq:localdimension-lowerbound}) follows with $C:=c_{0}^{-\delta}\left(ac_{0}\right)^{-\delta}M\left(a\right)$. The proof of (1)  is complete. 

The assertion in (2) follows by combining (1) with  Lemma \ref{lem:critical-exponents} (\ref{enu:criticalexponents-easy-inequalities-1}) and Proposition \ref{prop:critical-exponents}. To prove (3),  suppose that $G$ satisfies the
open set condition and set $G_{n}:=\left\langle f_{i}:i\in I\cap\left\{ 1,\dots,n\right\} \right\rangle $,
for each $n\in\N$. By \cite[Theorem B]{MR2153926} we have  $\dim_{H}\left(J\left(G_{n}\right)\right)=t\left(I\cap\left\{ 1,\dots,n\right\} \right)$,
for each $n\in\N$. Since $J\left(G_{n}\right)=\Jur\left(G_{n}\right)\subset\Jur\left(G\right)\subset\Jr\left(G\right)\subset\Jpre\left(G\right)$,
for each $n\in\N$, we obtain that $t\left(I\cap\left\{ 1,\dots,n\right\} \right)=\dim_{H}\left(J\left(G_{n}\right)\right)\le\dim_{H}\left(\Jur\left(G\right)\right)\le\dim_{H}\left(\Jr\left(G\right)\right)\le\dim_{H}\left(\Jpre\left(G\right)\right)$.
By Proposition \ref{prop:critical-exponents} we have  $\lim_{n}t\left(I\cap\left\{ 1,\dots,n\right\} \right)=t\left(I\right).$
Combining with the upper bound in (2), finishes the proof of (3). The proof is complete.
\end{proof}

\section{Applications to non-hyperbolic rational semigroups\label{sec:Inducing-Methods}}
\label{Inducing}

In this section we apply the results of Section \ref{Bowen}  to give estimates for  the Hausdorff dimension of the \linebreak (pre-)Julia sets of   non-hyperbolic rational semigroups which possess an inducing structure.

\subsection{General Setting}
Throughout this section we assume that $I$ is   countable. For a subset $\Lambda$ of $\Rat$, we 
denote by $\langle \Lambda \rangle $ the rational semigroup generated by 
$\Lambda. $ Thus $\langle \Lambda \rangle =\langle g: g\in \Lambda \rangle 
=\{ g_{1}\circ \cdots \circ g_{n}: 
n\in \N, g_{j}\in \Lambda, j=1,\ldots, n\} .$
\begin{defn}
[Inducing structure] \label{def-setting-a} Let $G=\left\langle f_{i}:i\in I\right\rangle $ denote a rational semigroup. Suppose that $\deg(g)\ge 2$ for each $g\in G$. We say that  $G=\left\langle f_{i}:i\in I\right\rangle $ 
 has an  inducing structure (with respect to $\{ I_1, I_2\}$) if  there exists a decomposition
$I=I_{1}\cup I_{2}$ with $I_2 \neq \emptyset$, such that the following holds for the
rational semigroups $G_{j}:=\left\langle f_{i}:i\in I_{j}\right\rangle $,
$j\in\left\{ 1,2\right\} $, and $H:=\left\langle H_{0}\right\rangle $ given by 
$$H_{0}:=\left\{ f_{i}:i\in I_{2}\right\} \cup\left\{ f_{i}f_{j_{1}}\dots f_{j_{r}}:i\in I_{2},r\in\N,\left(j_{1},\dots,j_{r}\right)\in I_{1}^{r}\right\}.$$ There exists an  $H$-forward invariant non-empty compact set $L\subset F\left(H\right)$ such that  $P\left(G_{2}\right)\subset L$ and   $f_{i}\left(P\left(G_{1}\right)\right)\subset L$, for    $i\in I_{2}$.
\end{defn}
In the following, when $G=\langle f_{i}:i\in I\rangle $ has an 
inducing structure with respect to $\{ I_{1},I_{2}\}$, 
let $H_{0}$ and $H$ be as in Definition \ref{def-setting-a}. We endow $H_{0}$ with the 
discrete topology.
\begin{lem}
\label{lem:inducedsemigroup-is-nicelyexpanding} Suppose that  $G=\left\langle f_{i}:i\in I\right\rangle $
 has an inducing structure. Then $H$ is nicely expanding. \end{lem}
\begin{proof}
By Definition \ref{def-setting-a}  we have 
\[
P\left(H\right)=\overline{\bigcup_{h\in H\cup \{ \id \}}h\left(\bigcup_{f\in H_{0}}\CV\left(f\right)\right)}\subset\overline{\bigcup_{h\in H\cup \{ \id \}}h\left(P\left(G_{2}\right)\cup\bigcup_{i\in I_{2}}f_{i}\left(P\left(G_{1}\right)\right)\right)}\subset L\subset F\left(H\right).
\]
Since $P(G_2)\subset P(H)$ and $P(G_2)\neq \emptyset$,   we have that  $H$ is nicely expanding by Proposition \ref{prop:nicelyexpanding-characterisation}. 
\end{proof}
The proof of the next lemma is straightforward and therefore omitted.
\begin{lem}
\label{lem:inducedsemigroup-satisfies-osc}Suppose that  $G=\left\langle f_{i}:i\in I\right\rangle $
has an inducing structure. If $\left\{ f_{i}:i\in I\right\} $  satisfies the open
set condition with open set $U$, then $H_0$ satisfies the open set condition with open set $U$.
\end{lem}
The following lemma holds for arbitrary finitely generated rational
semigroups of degree at least two. 
\begin{lem}
\label{lem:subsemi-has-samejuliaset}Let $\Gamma$ denote  a  rational semigroup with $\deg(g)\ge 2$,
for each $g\in \Gamma$.  Let $\Gamma_0$ be a finitely generated subsemigroup of $\Gamma$.  Let $\Omega$ denote a subsemigroup of $\Gamma$
with the property that,  for each $g\in\Gamma$,  there exists $h\in\Gamma_0$
such that $hg\in\Omega$. Then we have $J\left(\Gamma\right)=J\left(\Omega\right)$. \end{lem}
\begin{proof}
Clearly have  $J\left(\Omega\right)\subset J\left(\Gamma\right)$, and it remains to show the opposite  inclusion. By the density of the repelling
fixed points in the Julia set (\cite[Corollary 3.1]{MR1397693}), we have  $J\left(\Gamma\right)=\overline{\bigcup_{g\in\Gamma}J\left(g\right).}$
Hence, it suffices to prove that $J\left(g\right)\subset J\left(\Omega\right)$
for each $g\in\Gamma$.  Let $g\in\Gamma$ and let $A$ be a finite set of generators of $\langle \Gamma_0 \cup \{ g\} \rangle$ with $g\in A$. Let  $\gamma \in A^\N$ be given by $(g,g,\dots)$. For each  $n\in\N$, by our assumptions on $\Omega$ and $\Gamma$, there exists   $h_{n}\in\Gamma_0$ such that $h_{n}g^{n}\in\Omega$. For each $n\in\N$,  let $\alpha(n)\in A^n$ be given by $(g,g,\dots,g)$. Let $\beta(n) \in A^{m_n}$ be given by $\beta(n)_{m_n}\circ \dots \circ \beta(n)_1=h_n$.  Further, for each $n\in \N$,  we define the sequence $\gamma(n) \in A^\N$, given by $\gamma(n):=(\alpha(n),\beta(n), \alpha(n), \beta(n),\dots) \in A^\N$. 
We observe that $\gamma(n)\rightarrow\gamma$ with respect to the
product topology on $A^{\N}$, as $n$ tends to infinity.  By lower semicontinuity of
the Julia set (\cite[Proposition 2.2]{MR2237476}), we conclude that, 
for each $z\in J_{\gamma}$,  there is a sequence $\left(y_{n}\right)\in\Chat^{\N}$
with $y_{n}\in J_{\gamma(n)}$ such that $\lim_{n}y_{n}=z$. Consequently,
we have $J\left(g\right)=J_{\gamma}\subset\overline{\bigcup_{n\in\N}J_{\gamma(n)}}\subset J\left(\Omega\right)$, where the last inclusion holds, because we have  $J_{\gamma(n)}=J\left(h_{n}g^{n}\right)$
and $h_{n}g^{n}\in\Omega$,  for each $n\in\N$. The proof is complete.\end{proof}
\begin{lem}
\label{lem:juliasetof-H-equals-G}Suppose that  $G=\left\langle f_{i}:i\in I\right\rangle $
has an inducing structure.  Then  $J\left(G\right)=J\left(H\right)$.\end{lem}
\begin{proof}
Let $h\in I_2$ be an element. By Definition \ref{def-setting-a}, we have that $hg\in H$, for each  $g\in G$.  Now, the lemma follows from Lemma  \ref{lem:subsemi-has-samejuliaset} applied to $\Gamma:=G$,   $\Gamma_0:=\langle h \rangle$ and  $\Omega:=H$.\end{proof}

\begin{defn} Let $I$ be a topological space and let $(f_i)_{i\in I}\in C(I,\Rat)$. Let $G=\left\langle f_{i}:i\in I\right\rangle $   and let $\tilde{f}:J\left(\tilde{f}\right)\rightarrow I^\N \times \Chat \rightarrow I^\N \times \Chat$
be the associated skew product. For each $K \subset I$ and $\omega \in   K^\N$, we set $\hat{J}_{\omega,K}:=\pi_{\Chat}( \overline{J^{K}} \cap \pi_1^{-1}(\{\omega\}) )$, where the closure is taken in $I^\N \times \Chat$. 
\end{defn}

\begin{lem}[\cite{MR2736899}, Lemma 3.5] \label{lem:fibrejulia-is-intersection}  Let $G=\left\langle f_{i}:i\in I\right\rangle $ be a rational semigroup with  $\card(I)<\infty$.  Let  $K \subset I$ and  set $G_K:=\left\langle f_{i}:i\in K\right\rangle $. For each  $\omega \in   K^\N$, we have $\hat{J}_{\omega,K}=\bigcap_{n=1}^\infty f_{\omega |_n}^{-1}(J(G_K))$.
\end{lem}

\begin{lem}
\label{lem:complement-of-prejulia-osc} Let $G=\left\langle f_{i}:i\in I\right\rangle $ be a rational semigroup with  $\card(I)<\infty$.  Suppose that $\left\{ f_{i}:i\in I\right\} $ satisfies the open
set condition. Let $I_1 \subset I$ be a non-empty  subset and set $G_1:=\left\langle f_{i}:i\in I_1\right\rangle$. Suppose that $f_i(P(G_1)) \subset F(G)$, for all $i\in I \setminus I_1$, and that there exists a $G_1$-forward invariant compact subset $L_1 \subset F(G)$. 
Then we have  $\hat{J}_{\omega,I_1}= \hat{J}_{\omega,I}$, for each $\omega \in I_1$.  \end{lem}
\begin{proof}
Let $\omega \in I_1^\N$ and suppose by way of  contradiction that there exists $z_\infty \in  \hat{J}_{\omega,I} \setminus \hat{J}_{\omega,I_1}$.  Then there exist  sequences $(\beta^n)_{n\in \N}$ and $(z_n)_{n\in \N}$ with $\beta^n \in I^\N \setminus I_1^\N $ and $z_n \in J_{\beta^n}$, for each $n\in \N$,  such that $\lim_n \beta^n = \omega $ and $\lim_n z_n =z_\infty$.  Since $z_\infty \notin  \hat{J}_{\omega,I_1}=\bigcap_{n=1}^\infty f_{\omega |_n}^{-1}(J(G_1))$ by Lemma \ref{lem:fibrejulia-is-intersection},   there exists $m\in \N$ such that $f_{\omega|_m}(z_\infty)\in F(G_1)$.  We may assume that $\beta_{|m}^n =\omega_{|m}$, for each $n\in \N$. Define the sequence $(r_n) \in \N^\N$, given by $r_n:=\min\{k\in \N:\beta^n_k \notin I_1 \}$. Clearly, we have $r_n >m$, for each $n\in \N$,  and  that $r_n$ tends to infinity, because $\beta^n$ tends to $\omega \in I_1^\N$, as $n$ tends to infinity. Since $f_{\omega|_m}(z_\infty)\in F(G_1)$ and  $\beta^n_{|r_n -1}\in I_1^{r_n -1}$,  we have that $(f_{\beta^n_{| r_n -1}})_{n\in \N}$ is normal in a neighborhood of $z_\infty$ in $F_\omega$.  Let  $(n_j)\in \N^\N$ be a sequence tending to infinity,  such that the sequence $(g_j) \in G_1^\N$, given by \ $g_j:=f_{\beta^{n_j}_{| r_{n_j} -1}}  $,  converges uniformly in a neighborhood  $V$ of $z_\infty$. We may also assume that there exists $i_0 \in I\setminus I_{1}$ such that $\beta^{n_j}_{ r_{n_j} }=i_0$, for all $j\in \N$. We will now distinguish two cases.

\emph{Case 1}: Suppose there exists a  constant $c\in \Chat$, such that $g_j  \rightrightarrows c$ on $V$. Since  $g_j(z_{n_j})\in J(G) $, we have $c\in J(G)$. Hence, we have that  $c\notin  L_1\subset F(G)$ and  that  there exists a $G_1$-forward invariant neighborhood $W$ of $L_1$ in $F(G)$, such that  $c \notin \overline{W}. $ (To take such $W$, let $\delta _{0}>0$ be a small number 
such that setting $A=\{ z\in F(G): d(z,L_{1})<\delta _{0}\} $, 
we have $c\not\in \overline{\cup _{g\in G_{1}\cup \{ \id\} }g(A)}.$ 
Let $W:=\cup _{g\in G_{1}\cup \{ \id\}}g(A).$) To prove that $c\in P(G_1)$, suppose on the contrary that $c\notin  P(G_1)$.  Then there exists    $\delta>0$ such that $\overline{B(c,\delta)}\cap W = \emptyset$ 
and,  for each large $j$, there exists  a well defined inverse branch $h_{j}:B(c,\delta)\rightarrow \Chat$ of $g_j$,   such that  $h_{j}(g_j (z_\infty))=z_\infty$ and $g_j \circ h_j=\id$ on $B(c,\delta)$. Since $W$ is  $G_1$-forward invariant, we  conclude that $h_j (B(c,\delta))\cap W = \emptyset$. Hence, we obtain that   $( h_{j}) $ is normal in $z_\infty$, which is a contradiction (see the argument in the proof of Lemma \ref{lem:fatou-forward-dynamics}). We have thus shown that    $c\in P(G_1)$. Consequently, we have   $\lim_j f_{\beta^{n_j}_{| r_{n_j} }}(z_\infty)=f_{i_0}(c)\in f_{i_0}(P(G_1))\subset F(G)$.  On the other hand, we have $f_{i_0}(c)=\lim_j f_{\beta^{n_j}_{| r_{n_j} }}(z_{n_j}) \in J(G)$, which gives a contradiction.\linebreak
\emph{Case 2}: Suppose there exists a  non-constant holomorphic map $\varphi:V\rightarrow \Chat$, such that $g_j  \rightrightarrows \varphi$ on  $V$.  Suppose that $\left\{ f_{i}:i\in I\right\} $ satisfies the open
set condition with open set $U$.  Since $f_{i_0}(\varphi(z_\infty))=\lim_j f_{\beta^{n_j}_{| r_{n_j} }}(z_{{n_j}}) \in J(G)\subset \overline{U}$ and $f_{i_0}\circ \varphi$ is non-constant, there exists $z$ in a neighborhood of $z_\infty$, such that $f_{i_0}(\varphi(z))\in U$. Moreover, there exists $j_0 \in \N$ such that $f_{\beta^{n_j}_{| r_{n_j} }}(z) \in U$,  for all $j\ge j_0$.  We may also assume that $r_{n_{j+1}}>r_{n_j}$, for all $j\ge j_0$. Hence, for each $j\ge j_0$,  we have 
\begin{equation*}
z \in \bigg( f^{-1}_{\beta^{n_j}_{1}} \dots  f^{-1}_{\beta^{n_j}_{r_{n_j} }}(U) \bigg)  \cap  \bigg( f^{-1}_{\beta^{n_{j+1}}_{1}} \dots  f^{-1}_{\beta^{n_{j+1}}_{r_{n_{j}} }} \dots  f^{-1}_{\beta^{n_{j+1}}_{r_{n_{j+1}} }}(U) \bigg) \subset
 \bigg( f^{-1}_{\beta^{n_j}_{1}} \dots  f^{-1}_{\beta^{n_j}_{r_{n_j} }}(U) \bigg) \cap  \bigg( f^{-1}_{\beta^{n_{j+1}}_{1}} \dots  f^{-1}_{\beta^{n_{j+1}}_{r_{n_j} }}(U) \bigg).
\end{equation*}
Since $\left\{ f_{i}:i\in I\right\} $ satisfies the open set condition, we conclude that $\beta^{n_j}_{1}\dots \beta^{n_j}_{r_{n_j} }=\beta^{n_{j+1}}_{1}\dots \beta^{n_{j+1}}_{r_{n_j} }$. This is a contradiction, because $\beta^{n_j}_{r_{n_j} }=i_0 \in I\setminus I_1$ and $\beta^{n_{j+1}}_{r_{n_j} }\in I_1$. The proof is complete.
\end{proof}

\begin{lem} \label{lem:nicelyexpanding-fibres} Suppose that  $G=\left\langle f_{i}:i\in I\right\rangle $
has an inducing structure and that $\card(I)<\infty$.  Let $\omega \in I^\N$ and suppose that there exists   $\gamma \in H_0^{\N}$ and a  sequence $(n_k)\in \N^\N$ tending to infinity, such that $f_{\omega|_{n_k}}=\gamma_k \circ \dots \circ \gamma_1$, for each $k\in \N$.  Then we have $J_\gamma = J_\omega=\hat{J}_{\omega,I}$. 
\end{lem}
\begin{proof}
Clearly, we have $J_\gamma  \subset J_\omega \subset \hat{J}_{\omega,I}$. Suppose for a contradiction that there exists $z\in \hat{J}_{\omega,I} \setminus J_\gamma$.  Since $H$ is nicely expanding by Lemma \ref{lem:inducedsemigroup-is-nicelyexpanding} ,  there exists $c\in P(H) \subset F(H)$ and a subsequence $(n_k')$ of $(n_k)$ tending to infinity, such that $f_{\omega|_{ n'_k}}\rightrightarrows c$ in a neighborhood of $z$ by Lemma \ref{lem:fatou-forward-dynamics}.   On the other hand, it follows from Lemma \ref{lem:fibrejulia-is-intersection}  that  $\hat{J}_{\omega,I}=\bigcap_{k=1}^\infty f_{\omega |_{n'_k}}^{-1}(J(G))$. Hence, we have $c=\lim_k  f_{\omega|_{ n'_k}}(z)\in J(G)$. Since $J(G)=J(H)$ by Lemma \ref{lem:juliasetof-H-equals-G}, we get the desired contradiction. 
\end{proof}

\begin{lem}
\label{lem:complement-of-prejulia}  Suppose that  $G=\left\langle f_{i}:i\in I\right\rangle $
has an inducing structure with respect to $\{ I_{1},I_{2}\}$ and that 
$\card(I)<\infty.$ Let $G_{1}=\langle f_{i}:i\in I_{1}\rangle .$   If there exists a $G_1$-forward invariant compact set $L_1\subset F(G)$ and if $\{f_{i}: i\in I\}$ satisfies the open set condition, then we have 
 \[ J\left(G\right)=\Jpre\left(H\right)\cup\bigcup_{g\in G}g^{-1}\left(J\left(G_{1}\right)\right).\]
\end{lem}
\begin{proof}
Let $z\in J\left(G\right)$. Since $\card\left(I\right)<\infty$, there exists $\omega \in I^\N$ such that $z\in \hat{J}_{\omega,I}$ by Proposition  \ref{prop:skewproduct-facts} (\ref{enu:ratsemi-facts-3}). We now distinguish two cases. If there exists $\ell\in \N$ and $\tau \in I_1^\N$ such that $\omega =(\omega_1,\dots,\omega_\ell, \tau_1, \dots )$, then $f_{\omega|_\ell}(z) \in \hat{J}_{\tau,I}$. By Lemma \ref{lem:complement-of-prejulia-osc}, we have   $\hat{J}_{\tau,I}=\hat{J}_{\tau,I_1}\subset J(G_1)$.  We have thus shown that $z\in f_{\omega|_\ell}^{-1}(J(G_1))$.  If no such $\ell$  exists, then there exist   $\gamma \in H_0^{\N}$ and a  sequence $(n_k)\in I^\N$ tending to infinity, such that $f_{\omega|_{ n_k}}=\gamma_k \circ \dots \circ \gamma_1$, for each $k\in \N$. Hence, we have $z\in \hat{J}_{\omega,I}=J_\gamma \subset \Jpre(H)$ by Lemma \ref{lem:nicelyexpanding-fibres}.
\end{proof}

The following two lemmata give conditions under which we can  bound  the Hausdorff dimension of  the Julia set of a polynomial semigroup from above.  
\begin{lem}   \label{decoupling-lemma} Let   $G=\left\langle f_{i}:i\in I\right\rangle $ be a rational semigroup with $\card(I)<\infty$  such that the following holds.
\begin{enumerate}
\item \label{enu:decoupling} There exists a compact $G$-forward invariant set $K\subset F(G)$, such that $f_j(P(f_i))\subset K$, for all $i,j\in I$ with $i\neq j$.
\item \label{enu:G-polynomial} $f_i$ is a polynomial of degree at least two, for each $i\in I$. 

\end{enumerate} Then we have $\dim_H (\Jpre(G)) \le \max \{ s(G),  \max_{i\in I} \{ \dim_H(J(f_i)) \} \}$.
\end{lem}
\begin{proof}
For each $i\in I$,  let  $i^\infty = (i,i,i,\dots)\in I^\N$.  We will show that  \[ \dim_H\left(\bigcup_{\omega \notin \bigcup_{n\in \N_0} \sigma^{-n}(\bigcup_{i\in I}i^\infty)} {J}_\omega \right) \le s(G),\]  from which the lemma follows.  To prove this, let $\tilde{f}:I^\N\times \Chat \rightarrow 
I^{\N }\times \Chat $ be the skew product associated 
to $\{ f_{i}:i\in I\} $. We set 
\[
P(\tilde{f}):=\overline{ \bigcup_{n\in \N} \tilde{f}^n(\{(\omega , z):f_{\omega_1}'(z)=0\})}\subset I^{\N}\times\Chat, 
\] and we  first verify that 
\begin{equation} \label{bad-fibres}
J(\tilde{f})\cap P(\tilde{f})=\bigcup_{i\in I}\{(i^\infty,x) \in J(\tilde{f}) : x\in P(f_i)\}.
\end{equation}
Let $(\omega,x)\in J(\tilde{f}) \cap P(\tilde{f})$ be given.  By (\ref{enu:decoupling}), there exists $i\in I$ such that $x\in P(f_i)$. Since $f_{\omega_{| n}}(x)\in J(G)$, for each $n\in \N$, we have  $\omega =i^\infty$ by  (\ref{enu:decoupling}). Thus (\ref{bad-fibres}) holds.
For each $\omega \notin \bigcup_{n\in \N_0} \sigma^{-n}(\bigcup_{i\in I}i^\infty)$ and for each $n_0 \in \N$, there exists $n\ge n_0$ such that ${J}_{\sigma^{n-1}(\omega)}\subset f_i^{-1}f_j^{-1}(J(G))$, for some $i,j\in I$ with $i\neq j$.     By (\ref{bad-fibres}) we  then  have 
\[
\min \big\{ d(a,b): a\in {J}_{\sigma^{n-1}(\omega)}, b\in P(G) \big\}\ge \min  \big\{ d(a,b): i,j \in I, i\neq j, a\in f_i^{-1}f_j^{-1}(J(G)), b\in P(G)\big\}>0.
\]

Now, the claim follows from  \cite[Proposition 2.11 and Proposition 2.20]{MR2237476}.
\end{proof}

\begin{lem} \label{polynomial-semigroup-prejulia} Under the hypothesis of Lemma \ref{decoupling-lemma}, suppose that  $\{f_i:i\in I \}$ additionally satisfies the open set condition.  
Then we have $\Jpre(G)=J(G)$. In particular, we have \[ \dim_H (J(G)) \le \max \{ s(G),  \max_{i\in I} \{ \dim_H(J(f_i))\} \}.\]
\end{lem}
\begin{proof}
We will  show  that $\hat{J}_{\omega, I}=J_{\omega}$, for each $\omega \in I^\N$. Suppose for a contradiction that there exists $\omega \in I^\N$ and $z\in  \hat{J}_{\omega, I} \setminus J_{\omega}$. Since we have  $\hat{J}_{i^\infty,I}=J_{i^\infty}$, for each $i\in I$ by Lemma \ref{lem:complement-of-prejulia-osc} applied to $I_1:=\{i\}$, we  conclude that  there exist $i,j \in I$ with $i\neq j$ and a sequence $(n_k)\in \N^\N$ tending to infinity, such that $\omega_{n_k + 1}=i$ and $\omega_{n_k + 2}=j$. We may assume that there exists $g:V \rightarrow \Chat$ in a neighborhood $V$ of $z$, such that $f_{\omega|_{n_k}}\rightrightarrows g$ on $V$. We show that $g$ is non-constant. Otherwise, similarly as in the proof of Lemma \ref{lem:complement-of-prejulia-osc}  (Case 1), we can show that $g(z) \in P(G)$, which then implies that  $\lim_k f_{\omega_{| n_k + 2}}(z) = f_j f_i g(z) \in  F(G)$. This contradicts that $z \in \hat{J}_{\omega, I}$. We have thus shown that $g$ is non-constant. We may assume that there exists $\rho \in I^\N$ such that $\lim_k \sigma^{n_k}(\omega)=\rho$. Clearly, we have  $\rho_1 =i$ and $\rho_2 =j$.  Now, it follows from \cite[Lemma 2.13]{MR1827119} that $(\{\rho\}\times \hat{J}_{\rho,I}) \cap P(\tilde{f}) \neq \emptyset$, which contradicts (\ref{bad-fibres}).
\end{proof}

In order to state the main result of this section, let us introduce
regularity of the pressure function  associated to  rational semigroups. We adapt the definitions from  \cite[p.78]{MR2003772} in the context of graph directed Markov systems.
\begin{defn}
Let $I$ be a finite or countable set, let $\left(f_{i}\right)_{i\in I}\in\Rat^{I}$
and let $\tilde{f}:J\left(\tilde{f}\right)\rightarrow J\left(\tilde{f}\right)$
be the associated skew product. Suppose that $\langle f_{i}:i\in I\rangle $ is nicely expanding  and denote by $\mathcal{P}(t)$ the pressure function of the system $\{f_i: i\in I \}$ for $t\in \R$.  We say that $\left\{ f_{i}:i\in I\right\} $
is regular if there exists $t\ge0$ such that $\mathcal{P}(t)=0$.
Otherwise, we say that $\left\{ f_{i}:i\in I\right\} $ is irregular.
If $\left\{ f_{i}:i\in I\right\} $ is regular and if there exists
$u\in\R$ such that $0<\mathcal{P}(u)<\infty$, 
then $\left\{ f_{i}:i\in I\right\} $ is called strongly regular,
if $\left\{ f_{i}:i\in I\right\} $ is regular and no such $u\in\R$
exists, then $\left\{ f_{i}:i\in I\right\} $ is called critically
regular. Moreover, we set $\Theta(I):=\inf\{ \beta \in \R : \sup\{ Z_1 (I,\beta,x)   : x\in J(\langle f_i: i\in I  \rangle) \}< \infty \}$. \end{defn}
\begin{thm}
\label{thm:inducing-prejulia}Suppose that  $G=\left\langle f_{i}:i\in I\right\rangle $ has an inducing structure with respect to $\{ I_{1},I_{2}\}$ and that there exists a $G$-forward invariant compact set $L_{0}\subset F\left(G\right)$.  Let $H_{0},H,G_{1}$ be as in Definition \ref{def-setting-a}. 
Let $H_{0}$ be endowed with the discrete topology.
Then, we have the following.
\begin{enumerate}
\item \label{enu:inducingthm-prejuliaset-decomposition}$\Jpre\left(G\right)\setminus\Jpre\left(H\right)= \big( \bigcup_{g\in G\cup\left\{ \id\right\} }g^{-1}\left(\Jpre\left(G_{1}\right)\right)\big) \setminus \Jpre(H)$. 
\item \label{enu:inducingthm-criticalexponents}
If $I$ is countable, then $H$ is nicely expanding and $\dim_{H}\left(\Jpre\left(H\right)\right) \le s(H)\le  t(H_0)= \inf \{ \beta \in \mathbb{R} : \mathcal{P}(\beta )<0\}$, where $\mathcal{P}$ denotes the pressure function of the system $\{ h: h\in H_0\}$. 
\item \label{enu:inducingthm-bowen-formula}If $I$ is  countable,  and if $\{f_i :i\in I \}$ satisfies the open set
condition, then we have  
$$s\left(G_{1}\right)=t(I_1) \le\Theta\left(H_{0}\right)\le t\left(H_{0}\right)=s\left(H\right)=s\left(G\right)=t(I)=\dim_{H}\left(\Jpre\left(H\right)\right)$$
and  
\[
\dim_{H}\left(\Jpre\left(G\right)\right)=\max\left\{ s\left(G\right),\dim_{H}\left(\Jpre\left(G_{1}\right)\right)\right\} .
\]
If moreover  $\card\left(I\right)<\infty$, then  we have 
\[
\dim_{H}\left(J\left(G\right)\right)=\max \{ s(G), \dim_H(J(G_1))\}.
\]

\item \label{enu:inducingthm-decoupling} If $\{f_i :i\in I \}$ satisfies the open set condition,  $\card\left(I\right)<\infty$,  $f_i$ is a polynomial for each $i\in I_1$, and if there exists a compact $G_1$-forward invariant subset $K\subset F(G_1)$, such that $f_j(P(f_i))\subset K$ for all $i,j\in I_1$ with $i \neq j$, then 
\[
\dim_{H}\left(J\left(G\right)\right)=\max\Big\{ s\left(G\right), \max_{i\in I_1} \{ \dim_{H}\left(J(f_i)\right)\} \Big\} .
\]
 \item   Suppose that  $I$ is  countable. Then we have all of the following.
\begin{enumerate}
\item \label{enu:inducingthm-a-1}$\Theta\left(H_{0}\right)=\inf\left\{ \beta:P\left(H_{0},\beta,x\right)<\infty\right\} $, 
for each $x\in J\left(H\right)$.
\item \label{enu:inducingthm-b-1}$H_{0}$ is strongly regular if $\Theta\left(H_{0}\right)<t(H_0)$. 
\item \label{enu:inducingthm-c-1}$H_{0}$ is critically regular or irregular
if $\Theta\left(H_{0}\right)=t(H_0)$. 
\end{enumerate}
\end{enumerate}
\end{thm}
\begin{proof}
To prove the assertion in (\ref{enu:inducingthm-prejuliaset-decomposition}),
we first  verify that $J_{\gamma}=J_{\omega}$, for all $\gamma\in H_{0}^{\N}$ and $\omega\in I^{\N}$, 
for which there exists a sequence $\left(n_{k}\right) \in\N^{\N}$
tending to infinity,  such that $\gamma_{k}\circ\gamma_{k-1}\circ\dots\circ\gamma_{1}=f_{\omega|_{n_{k}}}$, 
for each $k\in\N$.  Since the inclusion
$J_{\gamma}\subset J_{\omega}$ is obviously true, we only address
the opposite  inclusion. Since $L_{0}$ is a compact, $G$-forward invariant
subset of $F\left(G\right)$, there exists a forward $G$-invariant neighborhood
$V$ of $L_0$, such that $\overline{V}\subset F\left(G\right)$. Now,
suppose by way of contradiction that there exists $z\in J_{\omega}\cap F_{\gamma}$.
Since $z\in J_{\omega}$ we have $f_{\omega|_{n_{k}}}(z) \in J\left(G\right)$ for each $k\in \N$.
Since $V$ is a relatively  compact subset of $F\left(G\right)$,  there exists $\epsilon >0$  such that  $d\big(V,f_{\omega|_{n_{k}}}(z)\big)>\epsilon$, for each $k\in \N$.
Combining this with our assumption that $z\in F_{\gamma}$, we obtain
that there exists $\delta>0$ such that $\gamma_{k}\circ\gamma_{k-1}\circ\dots\circ\gamma_{1}\left(B\left(z,\delta\right)\right)\cap V=\emptyset$,
for all $k\in\N$. Since $V$ is $G$-forward invariant, we conclude that
$f_{\omega|_{n}}\left(B\left(z,\delta\right)\right)\cap V=\emptyset$,
for all $n\in\N$, which implies that $z\in F_{\omega}.$ This contradiction
finishes the proof of $J_{\gamma}=J_{\omega}$.    We now let $z\in \Jpre(G)\setminus \Jpre(H).$ 
Then there exists $\omega \in I^\N$ such that 
$z\in J_{\omega }.$ 
Suppose $\card(\{ k\in I:\omega _{k}\in I_{2}\})=\infty .$ 
Then there exist $\gamma \in H_{0}^\N$ and a sequence 
$(n_{k})\in \N^\N$ tending to infinity such that 
$\gamma _{k}\circ \cdots \circ \gamma _{1}=f_{\omega |_{n_{k}}}$ 
for each $k.$ By the above observation, we have 
$z\in J_{\omega }=J_{\gamma }\subset \Jpre(H).$ However 
this is a contradiction. Hence 
$\card(\{ k\in I: \omega _{k}\in I_{2}\})<\infty $ 
and $z\in \cup _{g\in G\cup \{ \id\} }g^{-1}(\Jpre(G_{1})).$ 
Thus the assertion in (1) holds.

The assertion in (\ref{enu:inducingthm-criticalexponents}) follows from Lemma \ref{lem:inducedsemigroup-is-nicelyexpanding}  and Theorem \ref{thm:bowen-for-prejulia} (2).  To prove (\ref{enu:inducingthm-bowen-formula}), first observe that, since $ \left\langle f_{i}:i\in I\right\rangle$  satisfies the open set condition, we have that  $H_0$ satisfies the open set condition by Lemma \ref{lem:inducedsemigroup-satisfies-osc}. Hence,    $G=\left\langle f_{i}:i\in I\right\rangle $, $H=\langle H_0 \rangle $ and $G_1= \left\langle f_{i}:i\in I_1\right\rangle $ are  free semigroups.  In particular, we have 
$s\left(G_{1}\right)=t(I_1)$, $t\left(H_{0}\right)=s\left(H\right)$ and $s\left(G\right)=t(I)$ by Lemma \ref{lem:critical-exponents} (\ref{enu:criticalexponents-freesemigroup-2}). To prove  $t(I_1)\le\Theta\left(H_{0}\right)$, we first note that, for each $\alpha \ge 0$, $j\in I_2$, $x\in \Chat$ and   $y\in f_j^{-1}\left(x\right)$,   
\[
\left\Vert f_j'\left(y\right)\right\Vert ^{-\alpha}  \sum_{g\in G_{1}}\sum_{z\in g^{-1}\left(y\right)}\left\Vert g'\left(z\right)\right\Vert ^{-\alpha}\le Z_{1}\left(H_{0}, \alpha,x\right).
\]
For each $x\in J(G_2)$ and $y\in f_j^{-1}\left(x\right)$ we have $\Vert f_j'\left(y\right)\Vert\neq 0$. Hence, we have 
\begin{equation} \label{G1-includedin-H}
\sum_{g \in G_1}\sum_{z\in g^{-1}\left(y\right)}\left\Vert g'\left(z\right)\right\Vert ^{-\alpha}\le \left\Vert f_j'\left(y\right)\right\Vert ^{\alpha}  Z_1(H_0,\alpha,x).
\end{equation}
By the   definition of $\Theta\left(H_{0}\right)$, we have $Z_{1}\left(H_{0}, \alpha ,x\right)<\infty$, for each  $x\in J(G_2)\subset J(H)$,  $\epsilon>0$ and   $\alpha= \Theta(H_0) +\epsilon$.  Hence,   we have  $\sum_{g \in G_1}\sum_{z\in g^{-1}\left(y\right)}\left\Vert g'\left(z\right)\right\Vert ^{-\alpha}<\infty$ by (\ref{G1-includedin-H}). 
We have thus shown that $t\left(I_{1}\right)\le t\left(I_{1},y\right)\le\Theta\left(H_{0}\right)$. In order to verify $\Theta\left(H_{0}\right)\le t\left(H_{0}\right)$,
recall that since $H$ is nicely expanding, we have for each $x\in J\left(H\right)$
that $t\left(H_{0}\right)=t\left(H_{0},x\right)$ by Lemma \ref{lem:critical-exponents}
(\ref{enu:criticalexponents-hyperbolic-independence-ofx-in-julia-3}).
Consequently, for each $\epsilon>0$ and for each $x\in J\left(H\right)$,
we have that $\sum_{n\in\N}Z_{n}\left(H_{0},t\left(H_{0}\right)+\epsilon,x\right)<\infty$.
In particular,  we have $P\left(H_{0},t\left(H_{0}\right)+\epsilon,x\right)\le0<\infty$.
By Lemma \ref{lem:bounded_distortion_lemma} we conclude that $\sup_{y\in J\left(H\right)}Z_{1}\left(H_{0},t\left(H_{0}\right)+\epsilon,y\right)<\infty$,
which proves $\Theta\left(H_{0}\right)\le t\left(H_{0}\right)+\epsilon$.
Therefore, $\Theta\left(H_{0}\right)\le t\left(H_{0}\right)$.  We now prove $s\left(H\right)=s\left(G\right)$.
It is easy to see that $s(G)\ge s(H)$. In order to show the opposite inequality, let $t > s(H)$.  Then by the Claim in the proof of Lemma \ref{lem:critical-exponents} and Lemma \ref{lem:juliasetof-H-equals-G}, there exists a point $x_0\in J(H)=J(G)$ such that $\sum_{n\in \N}Z_n(H_0,t,x_0)<\infty$. Let $h\in \{ f_i :i\in I_2\}$ and let $x_1 \in h^{-1}(x_0)\subset J(G)=J(H)$.  Then $\sum_{n\in \N}Z_n(I_1,t,x_1)<\infty$.  Moreover, we have  $\sum_{n\in \N}Z_n(H_0, t,x_1)<\infty$ by  the Claim in the proof of Lemma \ref{lem:critical-exponents} and Lemma \ref{lem:juliasetof-H-equals-G}. Furthermore, by Lemma \ref{lem:bounded_distortion_lemma}, there exists a constant $C'>1$ such that $Z_n(H_0,t,x)\le C' Z_n(H_0,t,x_0)$, for each $x\in J(G)$. Therefore,  we have 
\begin{equation} \label{poincareseries1}
\sum_{n\in \N}Z_n(I,t,x_1)=\sum_{n\in \N}Z_n(I_1,t,x_1)+\sum_{n\in \N}Z_n(H_0,t,x_1)+\sum_{g\in G_1}\sum_{a\in g^{-1}(x_1)} \Vert g'(a)\Vert^{-t} \sum_{n\in \N}Z_n(H_0,t,a)<\infty.
\end{equation}
Thus, we have $t>s(G)$, which finishes the proof of $s(G)=s(H)$. That $t(H_0)$ is equal to $\dim_{H}\left(\Jpre\left(H\right)\right)$ follows from  Theorem \ref{thm:bowen-for-prejulia} because  $H$ is nicely expanding by Lemma \ref{lem:inducedsemigroup-is-nicelyexpanding}
and  $H_0$ satisfies the open set condition by Lemma \ref{lem:inducedsemigroup-satisfies-osc}. Combining this with the fact that the  Hausdorff-dimension is $\sigma$-stable
and that Lipschitz continuous maps do not increase Hausdorff-dimension
(we can apply this fact to holomorphic inverse branches of the elements
of $G$ defined locally in the complement of the critical values),
we obtain that  $\dim_{H}\left(\Jpre\left(G\right)\right)=\max \{ s(G), \dim_H(\Jpre(G_1))\}$ by (\ref{enu:inducingthm-prejuliaset-decomposition}).
Finally, if $\card(I)<\infty$, then we have  $\dim_{H}\left(J\left(G\right)\right)=\max \{ s(G), \dim_H(J(G_1))\}$ by  Lemma \ref{lem:complement-of-prejulia}. 

The assertion in (\ref{enu:inducingthm-decoupling}) follows from (\ref{enu:inducingthm-bowen-formula}) and  Lemma \ref{polynomial-semigroup-prejulia}.  Finally, (\ref{enu:inducingthm-a-1})
follows from Lemma \ref{lem:bounded_distortion_lemma} and  the statements in  (\ref{enu:inducingthm-b-1})
and (\ref{enu:inducingthm-c-1}) are derived  from (\ref{enu:inducingthm-a-1})
and  Proposition \ref{prop:critical-exponents}.  The proof is complete.
\end{proof}

\subsection{Special cases of polynomial semigroups}
In this section, we provide a class of polynomial semigroups which
have an   inducing structure.  Moreover, we can prove further refinements of our main result. 
\vspace{-5mm}
\selectlanguage{english}%
\begin{defn}
[PB-OSC] We say that $G=\left\langle f_{1},f_{2}\right\rangle $ (or the generator system $\{ f_1, f_2 \}$)  satisfies PB-OSC if $f_{1}$ and $f_{2}$ are polynomials of
degree at least two, such that each of the following holds. 
\selectlanguage{british}%
\begin{enumerate}
\item \label{enu:postcritically-bounded}$P(G) \setminus\left\{ \infty\right\} $
is a bounded subset of $\C$.
\item \label{enu:Kg1-in-interiorKg2}$K\left(f_{1}\right)\subset\Int K\left(f_{2}\right)$
\item \label{enu:OSC}$\{ f_1, f_2 \}$ satisfies the  open set condition with the open set  $(\Int K\left(f_{2}\right))\setminus K\left(f_{1}\right)$

\item \label{enu:.CVg2-in-interiorKg1}$\CV\left(f_{2}\right)\setminus\left\{ \infty\right\} \subset\Int K\left(f_{1}\right)$.
\end{enumerate}
\end{defn}
\selectlanguage{british}%
\vspace{-2mm}
We will frequently make use of the following facts for  $G=\left\langle f_{1},f_{2}\right\rangle $ satisfying PB-OSC.   For $\omega \in \{1,2\}^\N$, we set $K_{\omega}=\big\{ z\in\C:\big(f_{\omega|_{n}}\left(z\right)\big)_{n\in\N}\mbox{ is bounded}\big\} $.

By \cite[Lemma 3.6]{MR2736899} it follows from (\ref{enu:postcritically-bounded})
that $J_{\omega}$ is connected for each $\omega\in\left\{ 1,2\right\} ^{\N}$.
We also have that the corresponding filled Julia set $K_{\omega}$
is connected. Moreover, we have that $\Chat\setminus K_{\omega}$
is a connected component of $\Chat\setminus J_{\omega}$ and that
$\Chat\setminus K_{\omega}$ is the basin of attraction of infinity
of $\big(g_{\omega|_{n}}\big)$. By  (\ref{enu:postcritically-bounded}), (\ref{enu:Kg1-in-interiorKg2}) and (\ref{enu:OSC})
we have that $J_{2,1,1,\dots}=f_{2}^{-1}\left(J\left(f_{1}\right)\right)$
and $J\left(f_{1}\right)$ are disjoint. So, by \cite[Lemma 3.9]{MR2736899}
we have that either $f_{2}^{-1}\left(J\left(f_{1}\right)\right)$
surrounds $J\left(f_{1}\right)$ or that $J\left(f_{1}\right)$ surrounds
$f_{2}^{-1}\left(J\left(f_{1}\right)\right)$, where, for two compact
connected subsets $K_{1}$ and $K_{2}$ of $\C$, we say that $K_{1}$
surrounds $K_{2}$ if $K_{2}$ is included in a bounded component
of $\C\setminus K_{1}$. The following argument shows that $J\left(f_{1}\right)$
does not surround $f_{2}^{-1}\left(J\left(f_{1}\right)\right)$: Otherwise,
we have $f_{2}^{-1}\left(J\left(f_{1}\right)\right)\subset\Int K\left(f_{1}\right)$,
which implies that $f_{2}^{-1}\left(K\left(f_{1}\right)\right)\subset K\left(f_{1}\right)$
(here we use that $f_{2}^{-1}\left(K\left(f_{1}\right)\right)$ containing
$f_{2}^{-1}\left(J\left(f_{1}\right)\right)$ is connected and that
$\Chat\setminus K\left(f_{1}\right)$ is a connected component of
$\Chat\setminus J\left(f_{1}\right)$). However, $f_{2}^{-1}\left(K\left(f_{1}\right)\right)\subset K\left(f_{1}\right)$
implies that $J\left(f_{2}\right)\subset K\left(f_{1}\right)$ contradicting
(\ref{enu:Kg1-in-interiorKg2}). We have thus shown that $f_{2}^{-1}\left(J\left(f_{1}\right)\right)$
surrounds  $J\left(f_{1}\right)$. Consequently, we have that $f_{2}\left(J\left(f_{1}\right)\right)\subset K\left(f_{1}\right)$,
so $f_{2}\left(K\left(f_{1}\right)\right)\subset K\left(f_{1}\right)$.
Since $f_{2}\left(J\left(f_{1}\right)\right)\cap J\left(f_{1}\right)=\emptyset$, it follows that $f_{2}\left(K\left(f_{1}\right)\right)\subset\Int K\left(f_{1}\right)$.
Combining with the fact that $f_{1}\left(\Int K\left(f_{1}\right)\right)\subset\Int K\left(f_{1}\right)$,
we obtain that $\Int K\left(f_{1}\right)\subset F\left(G\right)$
by Montel's Theorem. We have thus shown that 
$f_{2}\left(K\left(f_{1}\right)\right)\subset\Int K\left(f_{1}\right)\subset F\left(G\right)$.

Our first observation is that PB-OSC implies the existence of an inducing structure.  
\begin{lem}
\label{PB-OSC-inducing}
If $G=\left\langle f_{1},f_{2}\right\rangle $ satisfies PB-OSC, then 
$\{ f_1, f_2 \}$ has an inducing structure with  respect to $I_1=\{ 1\}$ and $I_2=\{ 2\}$.  Moreover, there exists a $G$-forward
invariant compact subset of $F\left(G\right)$. Furthermore, we have 
$J\left(G\right)=\Jpre\left(G\right)$.
\end{lem} \vspace{-5mm}
\begin{proof}
We verify that $G=\left\langle f_{1},f_{2}\right\rangle $ has an inducing structure.  Let $G_{i}=\langle f_{i}\rangle $,  for $i=1,2$,   
let $H_{0}:=\{ f_{2}\} \cup \{ f_{2}\circ f_{1}^{r}: r\in \N \}$ and let 
$H=\langle H_{0}\rangle $. Set $L:=f_{2}\left(K\left(f_{1}\right)\right)\cup\CV\left(f_{2}\right)$
and note that we have shown above that $L$ is an  $H$-forward invariant
compact subset of $F\left(G\right)$. Moreover,
we have shown that $P\left(H\right)\subset L$, which implies that
$P\left(G_{2}\right)\subset P\left(H\right)\subset L$. Furthermore,
we have  that $\CV\left(f_{1}\right)\setminus\left\{ \infty\right\} \subset K\left(f_{1}\right)$.
Hence, $f_{2}\left(\overline{\bigcup_{n\in\N}f_{1}^{n}\left(\CV\left(f_{1}\right)\setminus\left\{ \infty\right\} \right)}\right)\subset f_{2}\left(K\left(f_{1}\right)\right)\subset L$.
Thus, $f_2\left(P\left(G_{1}\right)\right)\subset L$. We have thus
shown that $G$ has  an inducing structure.  To finish the proof, note that since $G$ is a finitely generated
polynomial semigroup, we have  $\infty \in F(G)$ and that  $\{ \infty\}$ is $G$-forward invariant.  Consequently, by Lemma \ref{lem:complement-of-prejulia}, we have 
\[J(G)\subset \Jpre(H)\cup \bigcup_{g\in G}g^{-1}(J(f_1)) \subset \Jpre(G).
\]
\end{proof} \vspace{-3mm}
The main result of this section is the following corollary of Theorem
\ref{thm:inducing-prejulia}.
 \vspace{-1mm}
\begin{cor}
\label{c:B1main}
If  $G=\left\langle f_{1},f_{2}\right\rangle $ satisfies PB-OSC, then we have 
$\dim_{H}\left(J\left(G\right)\right)=\max \{ s\left(G\right), \dim _{H}(J(f_{1}))\}$. Moreover, all  assertions in (1)(2)(3)(5) of Theorem  \ref{thm:inducing-prejulia} hold, where 
$I_{1}=\{ 1\}, I_{2}=\{ 2\} .$
\end{cor}
 \vspace{-3mm}
\section{Remarks on the Cone Condition}
\label{Remarks}
 \vspace{-2mm}
We comment on the cone condition used  in the context of conformal iterated
function systems (\cite{MR1387085}). 
\begin{rem}
For the results of this paper, the cone condition is not needed. We
have seen in Section \ref{ExampleSection}  that there are many
examples of rational semigroups  which do not satisfy the cone condition, and
for which our results can be applied. 
\end{rem}

In \cite[Theorem 3.15]{MR1387085} it is proved that, for the Hausdorff
dimension of the limit set $J\left(\Phi\right)$ of an infinitely
generated conformal iterated function system $\Phi$ satisfying the cone condition, we have 
\begin{equation}
\dim_{H}J\left(\Phi\right)=\inf\left\{ \delta:P\left(\delta\right)<0\right\} =\sup_{\Phi_F}\left\{ J\left(\Phi_{F}\right)\right\} .\label{eq:cifs-bowen}
\end{equation}
Here,  $P$ refers to the  associated pressure function
and $\Phi_{F}$ runs over all finitely generated subsystems of $\Phi$.
\begin{rem}
By the methods employed in the proof of Theorem \ref{thm:bowen-for-prejulia},
one can show that (\ref{eq:cifs-bowen}) holds, even if the cone condition
is not satisfied.  Instead of the cone condition (2.7) in \cite{MR1387085}, we need to assume that $\vert\phi_i'(x)\vert \le s$, for each $x\in X$ in the notation of  \cite{MR1387085}.  Since the upper bound for the Hausdorff dimension
is straightforward, we only comment on the lower bound of the Hausdorff
dimension. Since the pressure satisfies an exhaustion principle (see
\cite[Theorem 2.15]{MR2003772} or Proposition \ref{prop:exhaustion_principle-etc}
(\ref{enu:exhaustion-principle})) it suffices to verify (\ref{eq:cifs-bowen})
for finitely generated conformal iterated function systems, which  can
be obtained by extending the proof of \cite[Theorem 4.3]{falconerfractalgeometryMR2118797}
via  the bounded distortion property of the conformal iterated function
system. For finitely generated expanding rational semigroups, the dimension formula in (\ref{eq:cifs-bowen})
was proved in \cite{MR2153926}. 
\end{rem}
\vspace{-0.25cm}
\section{Acknowledgements}
The authors thank the referee for reading this paper carefully and 
giving many valuable comments. The authors thank R.  Stankewitz for valuable comments. The first author thanks Osaka University for its kind hospitality during his postdoctoral fellowship 2011--2014 funded by the research fellowship JA 2145/1-1 of the German Research Foundation (DFG). The research of the first author was partially supported by the JSPS postdoctorial research fellowship for foreign researchers (P14321) and the JSPS KAKENHI 15H06416. The research of the second author was partially
supported  by  JSPS KAKENHI 24540211. 
\vspace{-0.25cm}
\providecommand{\bysame}{\leavevmode\hbox to3em{\hrulefill}\thinspace}
\providecommand{\MR}{\relax\ifhmode\unskip\space\fi MR }
\providecommand{\MRhref}[2]{%
  \href{http://www.ams.org/mathscinet-getitem?mr=#1}{#2}
}
\providecommand{\href}[2]{#2}

\end{document}